\documentclass[11pt,a4paper]{amsart}

\usepackage{aliascnt}
\usepackage{hyperref}
\usepackage{amsfonts}
\usepackage{mathrsfs}
\usepackage{amsmath}
\usepackage{amsmath,graphics}
\usepackage{amssymb}
\usepackage{indentfirst,latexsym,bm,amsmath,pstricks,amssymb,amsthm,graphicx,fancyhdr,float,color}
\usepackage[all]{xy}
\usepackage{amsthm}
\usepackage{amscd}
\usepackage[all]{hypcap}

\newtheorem{theorem}{Theorem}[section]

\newaliascnt{lemma}{theorem}
\newtheorem{lemma}[lemma]{Lemma}
\aliascntresetthe{lemma}

\newaliascnt{conjecture}{theorem}
\newtheorem{conjecture}[conjecture]{Conjecture}
\aliascntresetthe{conjecture}

\newaliascnt{proposition}{theorem}
\newtheorem{proposition}[proposition]{Proposition}
\aliascntresetthe{proposition}

\newaliascnt{corollary}{theorem}
\newtheorem{corollary}[corollary]{Corollary}
\aliascntresetthe{corollary}

\newaliascnt{problem}{theorem}

\aliascntresetthe{problem}

\newaliascnt{claim}{theorem}
\newtheorem{claim}[claim]{Claim}
\aliascntresetthe{claim}

\theoremstyle{definition}

\newaliascnt{definition}{theorem}
\newtheorem{definition}[definition]{Definition}
\aliascntresetthe{definition}

\newaliascnt{example}{theorem}
\newtheorem{example}[example]{Example}
\aliascntresetthe{example}

\theoremstyle{remark}

\newaliascnt{remark}{theorem}
\newtheorem{remark}[remark]{Remark}
\aliascntresetthe{remark}

\newaliascnt{remarks}{theorem}
\newtheorem{remarks}[remarks]{Remarks}
\aliascntresetthe{remarks}

\numberwithin{equation}{section}
\renewcommand{\theequation}{\arabic{section}-\arabic{equation}}

\numberwithin{figure}{section}

\def\wt{\widetilde}
\def\ol{\overline}
\def\ra{\rightarrow}
\def\lra{\longrightarrow}

\def\Aut{\text{\rm{Aut\,}}}
\def\({$($}
\def\){$)$}
\def\chit{\chi_{\rm top}}
\def\bbp{\mathbb P}
\def\cala{\mathcal A}

\def\calm{\mathcal M}

\def\cals{\mathcal S}
\def\calt{\mathcal T}

\def\rank{\text{{\rm rank\,}}}

\def\ols{\overline{S}}
\def\oly{\overline{Y}}

\def\olb{\overline{B}}
\def\olf{\bar{f}}

\def\Alb{{\rm Alb}}

\newcommand{\Mcal}{\mathcal{M}}

\newcommand{\Acal}{\mathcal{A}}
\newcommand{\Acalg}{\mathcal{A}_g}

\newcommand{\Jac}{\mathrm{Jac}}
\newcommand{\Tcal}{\mathcal{T}}
\newcommand{\Tcalg}{\mathcal{T}_g}
\newcommand{\THcal}{\mathcal{TH}}

\newcommand{\Scal}{\mathcal{S}}

\newcommand{\TScalgn}{\mathcal{TS}_{g,n}}
\newcommand{\Xcal}{\mathcal{X}}

\newcommand{\Cbb}{\mathbb{C}}
\newcommand{\Gbb}{\mathbb{G}}

\newcommand{\Pbb}{\mathbb{P}}
\newcommand{\Qbb}{\mathbb{Q}}
\newcommand{\Rbb}{\mathbb{R}}
\newcommand{\Sbb}{\mathbb{S}}

\newcommand{\Zbb}{\mathbb{Z}}

\newcommand{\Gbf}{\mathbf{G}}
\newcommand{\Hbf}{\mathbf{H}}
\newcommand{\Tbf}{\mathbf{T}}
\newcommand{\Ubf}{\mathbf{U}}

\newcommand{\mrm}{\mathrm{m}}

\newcommand{\Res}{\mathrm{Res}}
\newcommand{\bsh}{\backslash}
\newcommand{\ad}{\mathrm{ad}}
\newcommand{\der}{\mathrm{der}}
\newcommand{\isom}{\simeq}
\newcommand{\mono}{\hookrightarrow}

\newcommand{\GSp}{\mathrm{GSp}}
\newcommand{\Sp}{\mathrm{Sp}}
\newcommand{\PGL}{\mathrm{PGL}}

\newcommand{\SL}{\mathrm{SL}}
\newcommand{\Nrd}{\mathrm{Nrd}}

\newcommand{\Mat}{\mathrm{Mat}}

\newcommand{\wrt}{\textrm{with\ respect\ to\ }}
\newcommand{\BB}{{BB}}
\newcommand{\dec}{\mathrm{dec}}
\newcommand{\Mbar}{\overline{M}}
\newcommand{\sing}{\mathrm{sing}}
\newcommand{\inv}{{-1}}

\begin{document}

\title{Finiteness of hyperelliptic and superelliptic curves with CM Jacobians}

\author{Ke CHEN}

\address{Department of mathematics, Nanjing University, Nanjing, China, 210093}
\email{kechen@ustc.edu.cn}

\author{Xin Lu}
\address{Institut f\"ur Mathematik, Universit\"at Mainz, Mainz, Germany, 55099}
\email{x.lu@uni-mainz.de}

\author{Kang Zuo}
\address{Institut f\"ur Mathematik, Universit\"at Mainz, Mainz, Germany, 55099}
\email{zuok@uni-mainz.de}

\thanks{This work is supported by SFB/Transregio 45 Periods, Moduli Spaces and Arithmetic of Algebraic Varieties of the DFG (Deutsche Forschungsgemeinschaft),
and partially supported by National Key Basic Research Program of China (Grant No. 2013CB834202) and NSFC}






\phantomsection

\maketitle

\begin{abstract}
	In this paper we study the Coleman-Oort conjecture for superelliptic curves,
	i.e. curves defined by affine equations $y^n=F(x)$ with $F$ a separable polynomial.
	We prove that up to isomorphism there are at most finitely many superelliptic curves of fixed genus $g\geq 8$ with CM Jacobians.
	The proof relies on the geometric structures of Shimura subvarieties in Siegel modular varieties
	and the stability properties of Higgs bundles associated to fibred surfaces.
\end{abstract}

\tableofcontents

\section{Introduction} This paper is dedicated to the Coleman-Oort conjecture for superelliptic Torelli locus. Our main result is the following:

\begin{theorem}[superelliptic Coleman-Oort]\label{thm-main}
	For fixed genus $g\geq 8$, there exist, up to isomorphism,
	at most finitely many smooth superelliptic curves of genus $g$ whose Jacobians are CM abelian varieties.

\end{theorem}
The notion of superelliptic curves is generalized from the hyperelliptic case:

\begin{definition}\label{def-superelliptic}
For an integer $n>1$, an $n$-superelliptic curve is an algebraic curve
(or simply superelliptic curve if $n$ is clear from the text)
defined by an $n$-superelliptic equation,
i.e. an affine equation of the form $$y^n=F(x)$$ with $F$ a separable polynomial (i.e. admitting no multiple root).
When $n=2$ we obtain the usual notion of hyperelliptic curves.
Note that the genus of such an curve can be computed explicitly in terms of $n$ and $\deg F$,  cf. \eqref{eqn-3-9},
and that for any fixed $g\geq 2$, there are finitely many possibilities of $(n,\deg F)$
such that the curve defined by $y^n=F(x)$ is of genus $g$.
\end{definition}

In the rest of this section we briefly review the Coleman-Oort conjecture including its original formulation and the superelliptic analogue of interest, and explain the main ideas in the proof.

\subsection{Coleman-Oort conjecture}
The original conjecture of Coleman \cite{coleman-87} predicts that when the genus $g\geq 4$,
there should be, up to isomorphism, at most finitely many smooth projective complex algebraic curves of genus $g$
whose Jacobians are CM, i.e. abelian varieties with complex multiplication.
Naturally one may restate the problem as the finiteness of CM points in the open Torelli locus $\Tcalg^\circ$,
by which we mean the schematic image of the Torelli morphism $$j^\circ:\Mcal_g\ra \Acal_g,\ \ C\mapsto\Jac(C).$$
Here we write $\Mcal_g$ resp. $\Acal_g$ for the moduli space of smooth projective curves of genus $g$ with
level-$\ell$ structure resp. of principal polarized abelian varieties of dimension $g$ plus level-$\ell$ structure,
usually denoted as $\Mcal_{g,\ell}$ resp. $\Acal_{g,\ell}$ in the literature,
with $\ell\geq3$ a fixed integer to assure the representability of the moduli spaces by schemes.

There have been many studies on counterexamples to the conjecture for small $g$.
For example, from cyclic covers of $\Pbb^1$, typically given by families of curves over $\Pbb^1$
with an affine equation $$y^n=F_\lambda(x)$$ with $F_\lambda(x)$ a polynomial in $x$ depending
on an parameter $\lambda$, one obtains infinitely many curves, non-isomorphic to each other,
whose Jacobians are CM; see \cite{moonen-10} for the example with $n=9$ and $F_\lambda(x)=x(x-1)(x-\lambda)$,
in which the curves are of genus 7.
This suggests that the correct version of the conjecture should have an assumption that $g$ is large enough (at least $g\geq 8$).

On the other hand, there are also considerable positive progresses toward the conjecture. Many of them are reformulated in terms of Shimura subvarieties in $\Acal_g$ (see \autoref{sec-shimura} below for details on Shimura varieties), justified by the Andr\'e-Oort conjecture affirming the behavior of an arbitrary infinite family of CM points under taking the Zariski closure:

\begin{theorem}[Andr\'e-Oort conjecture for $\Acal_g$]
	\label{tsimerman andre oort} Let $\Sigma$ be an infinite subset of CM points in $\Acal_g$. Then the Zariski closure of $\Sigma$ equals a finite union of Shimura subvarieties.	
\end{theorem} There have been many works focusing on the Andr\'e-Oort conjecture. In \cite{klingler yafaev annals} and \cite{ullmo yafaev annals} the conjecture is proved for all Shimura varieties assuming the Generalized Riemann Hypothesis, which is inspired by earlier works of Edixhoven and Yafaev, cf. \cite{noot bourbaki}. The o-minimality approach of Pila first establishes unconditionally the case of products of Siegel modular varieties of low genus cf. \cite{pila annals} \cite{pila tsimerman}, and the case of general Siegel modular varieties is proved recently by Tsimerman \cite{tsimerman-15}. His proof relies on the average Faltings height conjecture of Colmez  proved by Andreatta-Goren-Howard-Madapusi Pera \cite{andreatta goren howard madapusi pera} and Yuan-Zhang\cite{yuan-zhang},
which are way beyond the scope of the present work.
It suffices to keep in mind that this theorem transforms the conjecture of Coleman into:

\begin{conjecture}[Coleman-Oort]\label{coleman-oort conjecture} For $g$ sufficiently large,
	there exists no Shimura subvariety of positive dimension contained generically in the Torelli locus $\Tcal_g$ in $\Acal_g$.
	
\end{conjecture}
Here $\Tcal_g$ is the closure of $\Tcal_g^\circ$, in which $\Tcal_g^\circ$ is open;
and a Shimura subvariety $M\subset\Acal_g$ is said to be contained generically in $\Tcal_g$
if the intersection $\Tcal_g^\circ\cap M$ is Zariski dense open in $M$.
We refer to \cite{moonen-oort-13} (and the references therein) for a thorough discussion on this subject.
The rich geometry of Shimura varieties has lead to various results confirming the conjecture
for many Shimura subvarieties of prescribed type, cf. \cite{hain}, \cite{de jong zhang}, \cite{chen-lu-zuo} etc.

\subsection{Variant for $n$-superelliptic curves}



Naturally one may formulate problems of Coleman-Oort type for moduli spaces of curves with additional data. Define $\Scal_{g,n}$ to be the moduli space of cyclic branched cover $C\ra\Pbb^1$ defined by an $n$-superelliptic equation as in \autoref{def-superelliptic}, with $C$ of fixed genus $g$. We have the evident morphism forgetting the cover $$\Scal_{g,n}\ra \Mcal_g,\ \ (C\ra\Pbb^1)\mapsto C,$$ and we write $\TScalgn^\circ$ for its image inside $\Acal_g$ under the Torelli morphism, referred to as the $n$-superelliptic open Torelli locus. Similar to the case of $\Tcal_g^\circ$, it is locally closed in $\Acal_g$, whose closure $\TScalgn=\overline{\TScalgn^\circ}$ is called the $n$-superelliptic Torelli locus. Often the integer $n$ is omitted when it is clear from the context.

Thanks to \autoref{tsimerman andre oort}, we'll focus on the following equivalent form of the main result:
\begin{theorem}\label{thm-main'}For $g\geq 8$, the superelliptic Torelli locus does not contain generically any Shimura subvariety in $\Acal_g$ of positive dimension.
\end{theorem}

Note that the main result is sharp due to the counterexample with $g=7$ given in \cite{moonen-10} mentioned above.

Precedent to our result, various cases of the superelliptic Coleman-Oort conjecture have been studied by Y. Zarhin in a serious of works (see for example \cite{zarhin crelle}, \cite{zarhin zeitschrift}, and the references therein), with emphasis on the endomorphism algebras  of the Jacobians when the Galois group of the cover   is the full permutation group $S_d$ or the alternative group $A_d$.
When $n$ is prime to 3, the problem for the $n$-superelliptic Legendre family $y^n=x(x-1)(x-\lambda)$ is already solved in \cite{de jong noot}.

In \cite{moonen-10}, Moonen has proved that the $n$-superelliptic Torelli locus $\TScalgn$ itself is not a Shimura subvariety when the genus $g$ is at least 8. In fact the main theorem in \cite{moonen-10} gives a complete finite list of subvarieties $Z(m,N,a)\subset\Acal_g$ which are Shimura subvarieties. Here $Z(m,N,a)$ is the subvariety of Jacobians of cyclic covers of $\Pbb^1$ admitting an affine equation of the form $$y^m=\prod_{i=1,\cdots,N}(x-t_i)^{a_i}$$ for distinct points $t_i$ with local monodromy datum $a=(a_1,\cdots,a_N)$. Note that the special subvarieties obtained this way are defined by products of unitary groups and symplectic groups. We mention also that the main result of \cite{venkataramana inventiones} implies that the monodromy group of these $Z(m,N,a)$ are arithmetic subgroups in the corresponding Mumford-Tate groups up to central part under suitable constraints upon the local monodromy data. Since the fundamental group of a Shimura variety only differs from an arithmetic subgroup of the derived part of its Mumford-Tate group by a finite quotient, hence a general $Z(m,N,a)$, which is of dimension $N-3$, cannot be a Shimura subvariety, using a direct computation of the dimension of a Shimura variety from its Mumford-Tate group, cf. \cite[formula 3.3.1]{moonen-10}.

\subsection{Strategy of the proof}\label{sec-strategy-pf}
The proof of \autoref{thm-main'} is divided into two main steps:
we first reduce the proof to the case of Shimura curves (i.e. Shimura subvarieties of dimension one),
and then we exclude the existence of Shimura curves using the stability properties of the associated logarithmic Higgs bundles.

The reduction to the Shimura curves is formulated as follows:

\begin{theorem}\label{reduction to shimura curves}
	The superelliptic Torelli locus contains generically some Shimura subvariety of positive dimension
	if and only if it contains generically some Shimura curve.
\end{theorem}
In fact one first reduces the above theorem to the statement for simple Shimura varieties of positive dimension,
cf. \autoref{non-simple Shimura data};
and then the boundary behavior of Baily-Borel compactification implies the dimensional reduction to Shimura curves,
using the crucial property that the open $n$-superelliptic Torelli locus contains no compact (i.e. complete) curves.
Note that when $n=2$, the open hyperelliptic Torelli locus is affine,
while the general superelliptic case follows from \autoref{invariantstheorem}.



Based on the above dimension reduction, the main theorem is thus reduced to:
\begin{theorem}\label{thm-curve}
	For any fixed $g\geq 8$, there does not exist any Shimura curve contained generically in the Torelli locus
	of superelliptic curves of genus $g$.
\end{theorem}

The proof of \autoref{thm-curve} is the most technical part of our paper.
The main idea is to study the logarithmic Higgs bundle for the family of semi-stable superelliptic curves
associated to such a possible Shimura curve $C$ contained generically in $\TScalgn$,
in particular its eigenspace decomposition with respect to the action of the cyclic group $G\cong \Zbb/n\Zbb$.
We apply Viehweg-Zuo's characterization for Shimura curves by the maximality of  Higgs fields on  Higgs eigen sub-bundles
and the geometrical properties of this family to obtain a new fibration on the total space of this family
with some extra properties, and deduce a contradiction by analyzing this new fibration, which establishes \autoref{thm-curve}. More precisely:
\begin{list}{}
	{\setlength{\labelwidth}{4mm}
		\setlength{\leftmargin}{6mm}
		\setlength{\itemsep}{2.5mm}}
	\item[(i).]
	
	Let $\bar f: \ol S\to \ol B$  be  the family of semi-stable superelliptic curves representing such a possible Shimura curve $C$
	with semi-stable singular fibres $\Upsilon\subset \ol S$ over the discriminate locus $\Delta \subset \ol B$, cf. \autoref{defrepresenting}.
	Then there exists a global action of $G\cong \mathbb Z/n\mathbb Z$ on $\ols$
	(after a possible base change of $\ol B$), which induces  an action
	on the  logarithmic Higgs bundle
	$$(E_{\olb}^{1,0}\oplus E_{\olb}^{0,1},~\theta_{\olb}):=
	\big(\bar f_*\Omega^1_{\ols/\olb}(\log \Upsilon)\oplus R^1\bar f_*\mathcal O_{\ols},\,\theta_{\olb}\big)$$
	corresponding to the $\mathbb Q$-local system
	$\mathbb V_B:=R^1\bar f_*(\mathbb Q_{\bar S\setminus \Upsilon})$ on $B=\bar B\setminus \Delta$
	under the Simpson correspondence, cf. \cite{sim90}.
	Hence one obtains an eigen-space decomposition
	$$(E_{\olb}^{1,0}\oplus E_{\olb}^{0,1},~\theta_{\olb})
	=\bigoplus_{i=0}^{n-1}(E_{\olb}^{1,0}\oplus E_{\olb}^{0,1},~\theta_{\olb})_i$$
	corresponding to the $G$-action on
	$$ \mathbb V_B\otimes\mathbb C=\bigoplus_{i=0}^{n-1}\mathbb V_i.$$
	By \cite{viehweg-zuo-04}  there is a unique  strictly maximal decomposition
	$$(E_{\olb}^{1,0}\oplus E_{\olb}^{0,1},~\theta_{\olb})=(A^{1,0}_{\ol B}\oplus A^{0,1}_{\ol B},\theta_{\ol B}|_A)\oplus (F_{\olb}^{1,0}\oplus F_{\olb}^{0,1},~0)$$
	such that $ \theta_{\ol B}|_A$ is an isomorphism at the generic point and $(F_{\olb}^{1,0}\oplus F_{\olb}^{0,1},~0)$ corresponds to the maximal unitary local sub-system $\mathbb V^u_B\subset \mathbb V_B\otimes\mathbb C.$
	The above decomposition is  stabilized by
	the action of $G\cong \mathbb Z/n\mathbb Z$. In particular, there is an induced eigen-space decomposition
	$$(F_{\olb}^{1,0}\oplus F_{\olb}^{0,1},~0)=\bigoplus_{i=0}^{n-1}
	(F_{\olb}^{1,0}\oplus F_{\olb}^{0,1},~0)_i,$$
	which corresponds to the eigen-space decomposition
	$$\mathbb{V}_{B}^u=\bigoplus_{i=0}^{n-1}\mathbb{V}_{B,i}^u.$$
	As the first step, we show that $(F_{\olb}^{1,0}\oplus F_{\olb}^{0,1},~0)\not=0$ in our situation.
	
	\item[(ii).] If $F^{1,0}_{\ol B}\simeq \mathcal O_{\ol B}^{\oplus r}$ is a trivial vector bundle (up to a suitable
	base change of $\ol B$) and if there exists an irregular fibration $\bar f':\ol S\to \ol B'$
	such that those 1-forms in $H^0(\ol S, \Omega^1_{\ol S})$ coming from
	$F^{1,0}_{\ol B}\subset \bar f_*\Omega^1_{\ol S/\ol B}(\log \Upsilon)$ are pulled back of 1-forms on $\ol B'$ via  $\bar f'$, then one can easily derive a contradiction from the existence of  $\bar f'$ once $g\geq 8$, cf. \autoref{lem-3-71}.
	However, it is a priori not clear whether such a fibration always exists in general.
	
	\vspace{1mm}
	We try to achieve this  by looking at first the most simple case when $C$ is non-compact, $g\geq 8$ and $n=p\geq 5$ is a prime number.
	In this case, by \cite[\S\,4]{viehweg-zuo-04} for non-compact Shimura curves one deduces that $\mathbb V_B^u$ is just the maximal trivial  local subsystem
	$\mathbb V_B^{tr}\subset\mathbb V_B$ (up to a possible base change), or equivalently  $(F_{\ol B}^{1,0}\oplus F_{\ol B}^{0,1},0)$  is a trivial vector bundle.
	Note that $\mathbb V_B^{tr}\subset \mathbb V_B$ is a local subsystem defined over $\mathbb Q$,
	and hence the  local subsystem of $\mathbb{Q}(\xi_p)$-vector spaces
	$\mathbb V_B^{tr}\otimes \mathbb Q(\xi_p)\subset \mathbb V_B\otimes \mathbb Q(\xi_p)$  is stabilized by the action of  the Galois group
	${\rm Gal\,}\big(\mathbb Q(\xi_p)/\mathbb Q\big)$. Since $p$ is prime, ${\rm Gal\,}\big(\mathbb Q(\xi_p)/\mathbb Q\big)$ induces a transitive permutation  on the  eigen-subspaces
	$$\mathbb V_B^{tr}\otimes \mathbb Q(\xi_p)=\bigoplus_i\mathbb V_{B,i}^{tr}.$$
	Applying the Hurwitz-Chevalley-Weil formula (cf. \cite[Propsition\,5.9]{moonen-oort-13}) to ramified cyclic covers of $\Pbb^1$,
	together with the strictly maximal decomposition
	and the transitivity of the ${\rm Gal\,}\big(\mathbb Q(\xi_p)/\mathbb Q\big)$-action, one sees that
	$$\rank F^{1,0}_{\ol B, (p-1)/2}\geq \rank F_{\olb,(p+1)/2}^{1,0}>0.$$
	Take any two non-zero holomorphic 1-forms  $\alpha$ and $\beta$, which come from $F^{1,0}_{\ol B, (p-1)/2}$
	and $F^{1,0}_{\ol B, (p+1)/2}$ respectively.
	Then the wedge product  $\alpha\wedge\beta$
	is  a $G$-invariant holomorphic 2-form, hence descends to a holomorphic 2-form on the ruled surface
	$\ol S/G\to \ol B$.
	As all 2-forms on a ruled surface vanish, we get $\alpha\wedge\beta=0.$
	Now applying the Castelnuovo-de Franchis lemma  (cf. \cite[Theorem\,IV-5.1]{bhpv-04}) to $\alpha,\beta$, one finds a fibration
	$\bar f': \ol S\to \ol B'$ such that $\alpha$ and $\beta$ are pulled back from holomorphic  1-forms on $\ol B'$ via $\bar f'$.
	In fact,  all holomorphic  1-forms from  $F_{\olb,(p-1)/2}^{1,0}\oplus F_{\olb,(p+1)/2}^{1,0}$ are obtained in this way.
	Note that  the  pull-back map
	$$(\bar f')^*: H^1(\ol B',\mathbb Q)\to H^1(\ol S,\mathbb Q)$$ is defined over $\mathbb Q$ and the Hodge symmetry  (under the complex conjugation) gives
	$$\overline{F^{1,0}_{\ol B, (p-1)/2}}=F^{0,1}_{\ol B, (p+1)/2}.$$
	This implies that all holomorphic and anti-holomorphic 1-forms from
	$F^{1,0}_{\ol B, (p+1)/2}\oplus F^{0,1}_{\ol B, (p+1)/2}$ are pulled back via $\bar f'$;
	or equivalently, all classes in
	$H^1\big(\ol S,\mathbb Q(\xi_p)\big)_{(p+1)/2}=\big(H^1(\ol S,\mathbb Q)\otimes \mathbb Q(\xi_p)\big)_{(p+1)/2}$
	are pulled back of classes from $H^1(\ol B',\mathbb Q(\xi_p))$ via $\bar f'$.
	Finally the transitivity of	${\rm Gal\,}\big(\mathbb Q(\xi_p)/\mathbb Q\big)$-action
	implies that all classes from $F^{1,0}_{\ol B}$ and $F^{0,1}_{\ol B}$ are pulled back via $\bar f'.$
	In particular, the 1-forms in $H^0(\ol S, \Omega^1_{\ol S})$ coming from
	$F^{1,0}_{\ol B}\subset \bar f_*\Omega^1_{\ol S/\ol B}(\log \Upsilon)$ are pulled back of 1-forms on $\ol B'$ via  $\bar f'$
	
	\vspace{1mm}
	When $C$ is compact, the situation is much more complicated, 
	mainly due to two difficulties:\vspace{1.5mm}
	\begin{enumerate}
		\item[(1)] the flat subbundle $F^{1,0}_{\ol B}$  does not have to be trivial, even after any finite base changes.
		\item[(2)] the ${\rm Gal\,}\big(\mathbb Q(\xi_p)/\mathbb Q\big)$-action
		does not stablilize the unitary local sub-system $\mathbb V_{B}^u\subset \mathbb V_B\otimes \mathbb C.$
	\end{enumerate}\vspace{1.5mm}
	Thereby the above argument in the non-compact case no longer works here.
	To remedy the situation  we establish a slope inequality,  cf. \autoref{prop-3-4},
	which implies, together with the Akakelov equality for characterizing Shimura curves
	and the Miyaoka-Yau type inequality, that $F_{\olb,i_m}^{1,0}\neq 0$ for some $i_m>p/2$ in the case when $n=p$ is prime.
	Applying the local property of the eigen-sheaves of differential forms
	of cyclic covers described by Esnault-Viehweg \cite{esnault-viehweg-92}
	together with the Bogomolov lemma \cite[Lemma\,7.5]{sak-80} 
	and Deligne's lemma on the triviality of rank one Higgs bundle \cite[\S\,4.2]{deligne-71},
	one proves the triviality of $F_{\olb,i}^{1,0}$'s with $p-i_m\leq i\leq i_m$.
	This again enables us to produce a new fibration on $\ols$ by the same type of arguments
	as in the non-compact case such that a ``large part" of 1-forms from $F^{1,0}_{\ol B}$
	are pulling-backs of 1-forms via this new fibration,
	which is sufficient to derive a contradiction for the case when $g\geq 8$ and $p$ is prime.
%
%
	
	\item[(iii).] The general case (i.e., when $n$ is not prime)
	follows by induction on the number of prime factors in $n$.
	If $n$ is not prime and $n_1~|~n$,
	there is natural way to define a map $\rho_{n,n_1}$ from $C$ to $\mathcal{TS}_{g_1,n_1}$
	once a superelliptic automorphism group is chosen on the general fiber of $\bar f$. Here $g_1$ is the genus for the $n_1$-superelliptic curve $y^{n_1}=F(x)$ using the same separable polynomial $F$ as before.
	The key point for the induction process is to prove that
	$\rho_{n,n_1}(C)$ is again a Shimura curve generically contained in $\mathcal{TS}_{g_1,n_1}$
	when $n_1$ is maximal possible.
	By induction, it suffices to deal with the cases when $g\geq 8$ but $g_1<8$,
	and only finitely many possibilities arise.
	We apply to each of these cases the same idea used in the $p$-superelliptic case,
	and derive a contradiction for each of them.
\end{list}

\vspace{1.5mm}
\begin{remark}
	The hyperelliptic case ($n=2$) of the main theorem has already been established
	in our unpublished preprint \cite{chen-lu-zuo-15}.
	However, soon after the announcement in \cite{chen-lu-zuo-15},
	we realize that the same idea should be fruitful for general superelliptic curves,
	which has thus grown into the present uniform treatment.
\end{remark}


The paper is organized as follows.
In \autoref{sec-shimura} we recall some preliminaries on Shimura subvarieties
and present the dimensional recurrence so that the main theorem
is reduced to the exclusion of Shimura curves.
In \autoref{sec-proof-curve}, we prove \autoref{thm-curve} along the idea explained above,
and hence complete the proof of the main theorem.
Finally in \autoref{sec-superelliptic-family} we
provide the necessary technical preparations for various numerical properties about families of superelliptic curves,
which are used in \autoref{sec-proof-curve}.

 \subsection*{Notations}
 \begin{enumerate}
 	\item Let $a$ and $b$ be two non-zero integers. We write $a~|~b$ if $a$ divides $b$, i.e., if $b=ac$ for some integer $c$.
 	\item If $x$ is a rational number, we denote its integral part and fractional part by $[x]$ and $\{x\}$ respectively;
 	e.g., $\big[\frac53\big]=1$ and $\big\{\frac53\big\}=\frac23$.
 	\item For an $n$-superelliptic curve in \autoref{def-superelliptic} defined by $y^n=F(x)$,
 	we denote by $\alpha_0=\deg(F)$ the degree of $F(x)$, and
 	\begin{equation}\label{eqn-3-45}
 		\alpha=\left\{\begin{aligned}
 		&\alpha_0, &~&\text{~if~} n~|~\alpha_0;\\
 		&\alpha_0+1, && \text{~if~} n\,{\not|~}\alpha_0.
 		\end{aligned}\right.
 	\end{equation}
 \end{enumerate}

%
%
%
%

 \section{Shimura subvarieties and dimensional reduction}\label{sec-shimura}

 \subsection{Shimura varieties and Shimura curves}
 We first review briefly some facts needed for Shimura varieties, including a rough classification of Shimura curves.

 The Shimura varieties we use in this paper are actually the connected Shimura varieties defined as in \cite{chen-lu-zuo} and \cite{lu-zuo-14}, and standard notations such as $\Sbb=\Res_{\Cbb/\Rbb}\Gbb_\mrm$, $\Gbf(\Qbb)^+$, etc. follow the convention in \cite{deligne pspm} and \cite{milne introduction}. It suffices to bear in mind that our Shimura varieties are complex algebraic varieties of the form $\Gamma\bsh X^+$, where: \begin{itemize}
 	\item $X^+$ comes from a connected Shimura datum $(\Gbf,X;X^+)$ as in \cite[Definition\,2.1.3]{chen-lu-zuo};
 	in particular, $(\Gbf,X)$ is a Shimura datum in the sense of \cite{deligne pspm} and
 	$X^+$ is a connected component of $X$, which is also an Hermitian symmetric domain;
 	
 	\item $\Gamma$ is a congruence subgroup of $\Gbf^\der(\Qbb)^+$; see also \cite[Remark\,2.1.5]{chen-lu-zuo}
 	for the difference between the choice of $\Gamma\subset\Gbf^\der(\Qbb)^+$
 	and congruence subgroups in $\Gbf(\Qbb)^+$, arithmetic subgroups in $\Gbf^\ad(\Qbb)^+$, etc.
 \end{itemize}

 Inside $\Gamma\bsh X^+$ we have Shimura subvarieties associated to Shimura subdata, cf. \cite[Definition\,2.1.4]{chen-lu-zuo}.
 The zero-dimensional ones among them are called CM points, motivated by classical CM points in Siegel modular varieties.

 \begin{remark}
 	\label{remark cm points}Note that any Shimura variety contains a Zariski dense subset of CM points, and the Andr\'e-Oort conjecture studies the converse of this property. In fact starting with an arbitrary Shimura datum $(\Gbf,X)$, a point $x\in X$ is actually an $\Rbb$-group homomorphism $\Sbb\ra\Gbf_\Rbb$, and it is conjugate, under $\Gbf(\Rbb)$, to some $x':\Sbb\ra\Gbf_\Rbb$ having image in some $\Tbf_\Rbb$ with $\Tbf$ a $\Qbb$-torus in $\Gbf$, which gives $(\Tbf,x')\subset(\Gbf,X)$   a subdatum defining a CM point in $\Gamma\bsh X^+$.
 \end{remark}

 It is useful to mention the following result from \cite{moonen linear}\begin{theorem}[Moonen]\label{moonen linearity} Let $M$ be a Shimura variety defined by $(\Gbf,X;X^+)$ and let $Z\subset M$ be a closed irreducible subvariety (over $\Cbb$) . Then $Z$ is a Shimura subvariety if and only if it is totally geodesic and contains a CM point. 	
 \end{theorem}

 The theorem implies that being a Shimura subvariety or not can be verified over the universal covering $\wp:X^+\ra M$: $Z$ is a Shimura subvariety if and only if one of the geometric irreducible component $Z^+$ in $\wp^\inv(Z)\subset X^+$ is an equivariantly embedded Hermitian symmetric subdomain containing some point $y\in X^+$ such that $y(\Sbb)\subset\Tbf_\Rbb$ for some $\Qbb$-torus $\Tbf$ in $\Gbf$, i.e. $(\Tbf,y)$ is a subdatum of $(\Gbf,X;X^+)$ defining some CM point. Here an Hermitian symmetric domain $X'^+$ is said to be equivariantly embedded in $X^+$ if it is defined by some semi-simple Lie subgroup $G'$ of $\Gbf^\ad(\Rbb)^+$ and the inclusion $X'^+\mono X^+$ is equivariant \wrt $G'\mono\Gbf^\ad(\Rbb)^+$ sending a Cartan involution of $G'$ to a Cartan involution of $\Gbf^\ad(\Rbb)^+$, in the sense of \cite{satake real embedding}.



 \begin{example}[Siegel modular varieties]\label{shimura subvarieties in siegel modular varieties}
 	
 	In this paper $\Acal_g$ is the Shimura variety $\Gamma_g(\ell)\bsh \Xcal_g^+$ defined by $(\GSp_{2g},\Xcal_g;\Xcal_g^+)$, where $\GSp_{2g}$ is the $\Qbb$-group of symplectic similitude on the standard symplectic $\Qbb$-space $\Qbb^{2g}$, $\Xcal_g$ is the double half space of Siegel of genus $g$, and $\Xcal_g^+$ is the upper half space.
 	The group $\Gamma_g(\ell)$ is the principal congruence subgroup of level-$\ell$ in $\Sp_{2g}(\Zbb)$, equal to the kernel of the reduction modulo $\ell$: $\Sp_{2g}(\Zbb)\ra \Sp_{2g}(\Zbb/\ell)$, where for simplicity we choose $\ell$ to be an odd integer at least 3. It is well-known that $\Acal_g$ represents the moduli functor parametrizing principally polarized abelian varieties with full level-$\ell$ structure. 
 \end{example}

 Problems of Coleman-Oort type only involve  Shimura subvarieties in $\Acal_g=\Acal_{g,\ell}$. Since we have chosen $\ell$ large enough which makes $\Gamma_g(\ell)$ torsion free, the Shimura subvarieties are smooth complex submanifolds in $\Acal_g$.

 \begin{example}[Shimura curves]\label{example shimura curves}
  Shimura curves are one-dimensional Shimura varieties, hence they are associated to Shimura data of the form $(\Gbf,X;X^+)$ with $X^+$ the Poincar\'e upper half plane, which is the only one-dimensional Hermitian symmetric domain.
  This forces $\Gbf^\ad$ to be of the form $\Res_{L/\Qbb}\Hbf$ where $L$ is some totally real number field and $\Hbf$ is an $F$-form of $\PGL_2$;
  moreover, $\Hbf$ becomes $\PGL_{2,\Rbb}$ along exactly one real embedding $\sigma=\sigma_1:L\mono\Rbb$,
  and it is compact along the other real embeddings $\sigma_2,\cdots,\sigma_d:L\mono\Rbb$.
 	
 In this paper we are interested in Shimura curves inside $\Acal_g$,
 and these curves are given by Shimura subdata $(\Gbf,X;X^+)$ of $(\GSp_{2g},\Xcal_g;\Xcal_g^+)$.
 If the $\Qbb$-group $\Gbf^\der=\Res_{L/\Qbb}\Hbf$ as above comes from an $L$-form $\Hbf$ of $\PGL_2$,
 then the variation of Hodge structures (VHS for short) on $X^+$ associated to any algebraic representation of $\Gbf$
 is of even weights, which contradicts the existence of the canonical VHS of odd weight associated to
 the standard symplectic representation (equivalently from the universal abelian scheme of the moduli problem).
 Hence $\Hbf$ is an $L$-form of $\SL_2$. The classification of such $\Hbf$ is found in \cite{platonov rapinchuk}:
 \begin{enumerate}
 		\item[(i)] $\Hbf$ is an inner form of $\SL_2$, realized as the kernel of $\Nrd:\Gbb_\mrm^{D/L}\ra \Gbb_{\mrm,L}$ where $D$ is a quaternion $L$-algebra, $\Gbb_\mrm^{D/L}$ is the $L$-group sending an $L$-algebra $R$ to $(R\otimes_LD)^\times$, and $\Nrd$ is the reduced norm map;
 		\item[(ii)] $\Hbf$ is an outer form of $\SL_2$, which is the derived group of a unitary $L$-group $\Ubf$:
 		\begin{enumerate}
 			\item[(ii-a)] either $\Ubf$ is the $L$-group of automorphisms of an Hermitian form $E^2\times E^2\ra E$ for an imaginary quadratic extension of $L$;
 			\item[(ii-b)] or for some imaginary quadratic extension $E$ over $L$ and some quaternion division $E$-algebra $D$ there exists an involution $*$ of second hand on $D$ together with a $*$-Hermitian form $D\times D\ra E$, of which $\Ubf$ is the $L$-group of automorphisms commuting with the natural action of $D$. (Note that (ii-a) can be seen as the case when $D=\Mat_2(E)$.)
 		\end{enumerate}
 	\end{enumerate}	

  It is also known that Shimura curves associated to $(\Gbf,X;X^+)$ are compact (i.e. proper)
  if and only if the $\Qbb$-rank of $\Gbf^\ad$ is zero.	
 \end{example}
 The following result will be used later in \autoref{cor-3-4}.
 \begin{lemma}\label{lem-evenrank}
 	 Let $C\subset\Acal_g$ be a compact Shimura curve, with $E_C$ the Higgs bundle on $C$ associated to the universal abelian scheme over $C$ defined by the moduli problem. Write $E_C=F_C\oplus A_C$ for the decomposition into the flat part and the maximal part. Then the rank of $A_C$ is divided by 4, namely the rank of $A_C^{1,0}$ is even.
 	
 \end{lemma}

 \begin{proof}
 	The decomposition of Higgs bundles $E_C=F_C\oplus A_C$ is determined by the representation of the fundamental group of $C$
 	and is invariant under base change, which in turn is determined by the algebraic representation $\Gbf^\der\ra \Sp_{2g}$
 	on $\Qbb^{2g}$, following \cite{chen-lu-zuo}. From the classification in \cite{satake rational embedding},
 	we see that $\Gbf^\der\ra\Sp_{2g}$ decomposes as $\Qbb^{2g}=T\oplus V$,
 	where the action of $\Gbf^\der$ on $T$ is trivial, and the $\Qbb$-vector space $V$ has the structure of an $F$-subspace,
 	such that the action of $\Gbf^\der$ is the scalar restriction from an $L$-group homomorphism $\Hbf\ra\Sp_{V,L}$
 	for some $F$-linear  symplectic structure on $V$.
 	Moreover, following the cases listed in \autoref{example shimura curves}:
 	\begin{enumerate}
 		\item[(i)] if $L$ is totally real and $\Hbf$ comes from a quaternion $L$-algebra $D$, then $V$ is a $D$-vector space, and $D$ must be a division $L$-algebra because $C$ is compact and $\Gbf^\der$ is of $\Qbb$-rank zero; this forces the $L$-dimension of $V$ to be a multiple of 4;
 		
 		\item[(ii)] if for some CM field $E$ of real part $L$ and degree $2d$ over $\Qbb$, and $\Hbf$ is associated to $D$ a quaternion $E$-algebra: \begin{enumerate}
 			\item[(ii-a)] either $D\isom\Mat_2(E)$, and $V$ is a direct sum of copies of the standard representation of $\Hbf$ on $E^2$, whose $E$-dimension is even, and $L$-dimension is divided by 4;
 			\item[(ii-b)] or $D$ is a quaternion division $E$-algebra, and $V$ is a $D$-vector space, whose $E$-dimension is divided by 4 and $L$-dimension divided by 8.
 		\end{enumerate}
 	\end{enumerate}
 	Hence the $\Qbb$-dimension of $V$, which also equals the rank of $A_C$, must be a multiple of $4d$, with $d$ the degree of $L$. 	
 \end{proof}
 \subsection{Dimensional reduction}
 In  this subsection we recall some properties of Shimura varieties, with focus on Shimura subvarieties in $\Acal_g$, and explain the reduction of the main result to the exclusion of Shimura curves.

 \begin{lemma}[finite intersection]\label{finite intersection} The collection of finite unions of Shimura subvarieties in $\Acal_g$ is stable under finite intersections.
 	
 \end{lemma}
 \begin{proof}
 	As is explained in \cite{deligne milne hodge classes}, Shimura subvarieties in $\Acal_g$ are geometrically connected component of moduli subspaces parametrizing principally polarized abelian varieties with prescribed Hodge classes. It is clear that for two such moduli subspaces $M_1$ and $M_2$, the intersection $M_1\cap M_2$ remains a moduli subspace by joining the two sets of Hodge classes defining $M_1$ and $M_2$ respectively. Hence if $M_1^+\subset M_1$ and $M_2^+\subset M_2$ are Shimura subvarieties given as their connected components, then $M_1^+\cap M_2^+$ is a finite union of Shimura subvarieties that are connected components in $M_1\cap M_2$. This shows that the collection of finite unions of Shimura subvarieties in $\Acal_g$ is stable under finite intersection.
 \end{proof}




 Shimura varieties are in general non-complete, and among the various theories of compactification we simply mention the minimal one:
 \begin{theorem}[Baily-Borel compactification]\label{baily-borel compactification} Let $M=\Gamma\bsh X^+$ be a Shimura variety. Then the following hold:
 	
 	(1) $M$ is a normal quasi-projective algebraic variety over $\Cbb$. It admits a compactification, called the Baily-Borel compactification $M^\BB$, which is a projective algebraic variety over $\Cbb$ containing $M$ as a dense open subvariety, together with the universal property that if $M\ra Z$ is a morphism of complex algebraic varieties with $Z$ projective, then it admits a unique factorization of the form $M\mono M^\BB\ra Z$.
 	
 	(2) The boundary components of $M$, i.e. irreducible components of $M^\BB \bsh M$, are of codimension at least 2, unless $\Gbf^\ad$ admits a $\Qbb$-factor isomorphic to $\PGL_{2,\Qbb}$.
 	
 	(3) If $f:(\Gbf',X';X'^+)\ra(\Gbf,X;X^+)$ is a morphism of connected Shimura data compatible with the choices of congruence subgroups $\Gamma'\subset\Gamma$ in $\Gbf'^\der(\Qbb)^+$ and in $\Gbf^\der(\Qbb)^+$ respectively, then the evident map $M'=\Gamma'\bsh X'^+\ra M=\Gamma\bsh X^+$ extends uniquely to the compactifications $M'^\BB\ra M^\BB$.
 	
 	
 	
 	
 \end{theorem}	

 In fact from \cite{baily borel compactification} we know that the boundary components are lower dimensional Shimura subvarieties associated to proper parabolic $\Qbb$-subgroups of $\Gbf$. In particular, if $\Gbf$ admits no proper parabolic $\Qbb$-subgroups, then no boundary component is needed, and the Shimura variety in question is projective itself.

 \begin{corollary}[boundary components]\label{boundary components} Let $M=\Gamma\bsh X^+$ be a Shimura variety defined by $(\Gbf,X;X^+)$ and a congruence subgroup $\Gamma\subset\Gbf^\der(\Rbb)^+$. Let $M'\subset M$ be a Shimura subvariety defined by some subdatum $(\Gbf',X';X'^+)$ such that $\Gbf'^\ad$ admits no $\Qbb$-factor isomorphic to $\PGL_{2,\Qbb}$. Write $\Mbar'$ for the closure of $M'$ inside $M^\BB$, then the irreducible components of $\Mbar'\bsh M'$ are of codimension at least 2 in $\Mbar'$.
 	
 \end{corollary}
 \begin{proof}
 	The closed immersion $M'\mono M$ extends to a morphism between their compactifications $M'^\BB\ra M^\BB$, which is generically injective, and the closure $\Mbar'$ of $M'$ in $M^\BB$ is also the closure of the image of $M'^\BB$. The conclusion is clear because $M'^\BB$ only joints to $M'$ finitely many boundary components of codimension at least 2.
 \end{proof}

 We also need the notion of decomposable locus in $\Acal_g$.

 \begin{definition}[decomposable locus]\label{decomposable locus} A principally polarized abelian variety $A$ over $\Cbb$ is said to be decomposable if it is isomorphic to a direct product $A=A_1\times A_2$ with $A_1$ and $A_2$ both principally polarized of dimension $>0$ whose polarizations induce the polarization of $A$. We thus get the locus $\Acal_g^\dec\subset \Acal_g$ of decomposable principally polarized abelian varieties.
 	
 \end{definition}
 \begin{lemma} The decomposable locus $\Acal_g^\dec$ is a finite union of Shimura subvarieties in $\Acal_g$.
 	
 \end{lemma}	
 \begin{proof}
 	It suffices to notice that if a $g$-dimensional principally polarized abelian variety $A$ admits a decomposition $A\isom A_1\times A_2$   as in \autoref{decomposable locus}, with $\dim A_1=m>0$ and $\dim A_2=g-m>0$, where we assume for simplicity $m\leq g-m$, then the point in  $\Acal_g$ corresponding to $A$ lies in the Shimura subvariety $\Acal_{m,g-m}$ of $\Acal_g$ which is defined by the subdatum $(\GSp_{2m,2g-2m},\Xcal_{m,g-m};\Xcal_{m,g-m}^+)$.
 	
 	Here $\GSp_{2m,2g-2m}$ is the $\Qbb$-subgroup of $\GSp_{2g}$ consisting of symplectic similitude on $\Qbb^{2g}$ preserving the direct sum $\Qbb^{2g}=\Qbb^{2m}\oplus\Qbb^{2g-2m}$ into two symplectic $\Qbb$-subspace using the evident symplectic basis, $\Xcal_{m,g-m}^+=\Xcal_{m}^+\times \Xcal_{g-m}^+$ is the product of two Siegel upper half spaces, and $\Xcal_{m,g-m}$ is the orbit of $\Xcal_{m,g-m}$ inside $\Xcal_{g}$ under $\GSp_{2m,2g-2m}(\Rbb)$. Since $\GSp_{2m,2g-2m}^\der=\Sp_{2m}\times\Sp_{2g-2m}\subset\Sp_{2g}$, one verifies easily that $\Xcal_{m,g-m}$ consists of only two copies of $\Xcal_{m,g-m}^+$.
 	
 	The conclusion is thus clear because $$\Acal_g^\dec=\bigcup_{1\leq m\leq g/2}\Acal_{m,g-m}$$ is a finite union of Shimura subvarieties. \end{proof}

 In the rest of this section we prove the reduction of the main results to the exclusion of Shimura curves. We first reduce to the case where the Shimura subvarieties in question are simple, i.e. defined by Shimura data $(\Gbf,X;X^+)$ with $\Gbf^\ad$ simple $\Qbb$-groups, essentially because non-simple Shimura varieties admit proper Shimura subvarieties, and their Hecke orbits are Zariski dense.

 We start with the fact that a non-simple Shimura variety contains proper Shimura subvariety of dimension $>0$, the proof of which is reduced to the following lemma on subdata of non-simple Shimura data.

 \begin{lemma}[non-simple Shimura varieties]\label{non-simple Shimura data} Let $(\Gbf,X;X^+)$ be a connected Shimura datum, which is non-simple in the sense that $\Gbf^\ad$ is non-simple as a $\Qbb$-group. Let $M=\Gamma\bsh X^+$ be a Shimura variety associated to it. Then $M$ contains proper Shimura subvarieties $M'\subsetneq M$ of dimension $>0$.
 	
 \end{lemma}

 \begin{proof}
 	We first consider the case where $\Gbf=\Gbf^\ad$. Write $\Gbf=\prod_{i=1}^{r}\Gbf_i$ for the decomposition into the direct product of two semi-simple $\Qbb$-groups of adjoint type, each of which admits no compact $\Qbb$-factors as $\Gbf$ does. Take a base point $x\in X^+$, the composition $x_i:\Sbb\overset{x}{\ra}\Gbf_\Rbb\ra\Gbf_{i,\Rbb}$ induces a Shimura datum $(\Gbf_i,X_i;X_i^+=\Gbf_i(\Rbb)^+x_i)$ and we have a decomposition of Shimura data $(\Gbf,X;X^+)\isom(\Gbf_1,X_1;X_1^+)\times(\Gbf_2,X_2;X_2^+)$. Note that $\dim X_i>0$ for $i=1,2$, hence we may form subdata in $(\Gbf,X;X^+)$ of the form $(\Gbf',X';X'^+)=(\Tbf_1,y_1)\times(\Gbf_2,X_2;X_2^+)$ with $\Tbf_1$ some $\Qbb$-torus in $\Gbf_1$ by \autoref{remark cm points}, and it defines a proper Shimura subvariety because $0<\dim X'=\dim X_2<\dim X$.
 	
 	In general, to produce a proper Shimura subvariety amounts to find an equivariantly embedded Hermitian symmetric subdomain $X'^+$ in $X^+$ containing some CM point $x'$ according to \autoref{moonen linearity}. The subdomain $X'^+$ can be found using only the Shimura $(\Gbf^\ad,X^\ad;X^+)$ because $X^+$ is homogeneous under $\Gbf^\ad(\Rbb)^+$ (here we put $X^\ad$ the orbit of $X^+$ under $\Gbf^\ad(\Rbb)$). We may simply take $X'^+$ from a subdatum $(\Gbf',X';X'^+)$ of $(\Gbf^\ad,X^\ad;X^+)$ as in the adjoint case above. Note that the evident morphism $(\Gbf,X;X^+)\ra(\Gbf^\ad,X^\ad;X^+)$ is the identity on $X^+$ with $\Gbf\ra\Gbf^\ad$ the canonical projection modulo the center of $\Gbf$, hence a CM point given by $(\Tbf,x)\subset(\Gbf^\ad,X^\ad;X^+)$ lifts uniquely to $(\Hbf,y)\subset(\Gbf,X;X^+)$: here $\Hbf$ is the pre-image of $\Tbf$ in $\Gbf$, which remains a $\Qbb$-torus, and $y$ is the same point as $x$, viewed as a homomorphism $\Sbb\ra\Gbf_\Rbb$ which reduces to $x$ modulo the center.
 \end{proof}

 Inside a Shimura variety $M=\Gamma\bsh X^+$ defined by $(\Gbf,X;X^+)$ we can talk about Hecke translation associated to elements in $\Gbf(\Qbb)^+$, as is defined in \cite{chen-lu-zuo}[Definition 2.1.9]. We also have the following:




 \begin{lemma}[density of Hecke orbits]\label{density of hecke orbits} Let $M\subset\Acal_g$ be a Shimura subvariety defined by $(\Gbf,X;X^+)$, which is contained generically in $\TScalgn$.
 	
 	(1) Assume that $M$ contains a proper Shimura subvariety $M'\subsetneq M$ of dimension $>0$ defined by some subdatum $(\Gbf',X';X'^+)$. Then there exists a Hecke translate $M''=\wp(qX'^+)$ of $M'$ in $M$ contained generically in $\TScalgn$.
 	
 	(2) It suffices to prove the main theorem for $M$ such that $\Gbf^\ad$ is $\Qbb$-simple.
 	
 \end{lemma}

 \begin{proof} (1) This is exactly the same as \cite[Lemma\,2.1.10]{chen-lu-zuo}
 	when one replaces the usual Torelli locus $\Tcal_g$ by $\TScalgn$.

 	
 	(2) Assume that $M$ is defined by some subdatum $(\Gbf,X;X^+)$ such that $\Gbf^\ad$ is NOT $\Qbb$-simple. Then by \autoref{non-simple Shimura data} $(\Gbf,X;X^+)$ contains some subdatum $(\Gbf',X';X'^+)$ with $0<\dim X'^+<\dim X^+$.
 	It defines a Shimura subvariety $M'\subset M$ of dimension $>0$.
 	If $M$ is not contained generically in $\THcal_g$, then after passing to a suitable Hecke translate we may assume that $M'$ is contained generically in $\THcal_g$, and to prove the main theorem it suffices to exclude the generic inclusion of $M'$ in $\THcal_g$.
 \end{proof}

 Aside from the basic properties of Shimura varieties listed above, we need  the following crucial fact:

 \begin{theorem}\label{affiness} The open $n$-superelliptic Torelli locus $\TScalgn$ contains no complete curves.
 \end{theorem}

 This is a consequence of \autoref{invariantstheorem} studied in Section 4: if $\TScalgn$ contains a complete curve $C$,
 then the family of $n$-superelliptic curves $\bar{f}:\overline{S}\ra\overline{B}$ representing $C$ is smooth,
 i.e. admits no singular fibers.
 This would give a Hodge bundle $\bar{f}_*\omega_{\overline{S}/\overline{B}}$ of degree $0$
 by the formulas in \autoref{invariantstheorem} (see also \cite[Proposition\,4.7]{cornalba-harris-88} for the hyperelliptic case),
 which is absurd.

 \begin{remark} When treating the hyperelliptic case in the unpublished preprint \cite{chen-lu-zuo-15}, we have made use of the stronger property that the hyperelliptic open Torelli locus is affine. For the general $n$-superelliptic case the affiness is not true, but it suffices to exclude complete curves like above.
 	
 \end{remark}

 \begin{theorem}[dimensional reduction]\label{dimensional reduction} If the $n$-superelliptic  Coleman-Oort conjecture holds for Shimura curves in $\Acal_g$, then it holds for any Shimura subvarieties in $\Acal_g$.
 	
 \end{theorem}

 \begin{proof}
 	By \autoref{density of hecke orbits}, it suffices to consider the case where $M$ is a simple Shimura variety in $\Acal_g$ of dimension $\geq 2$, defined by some Shimura datum $(\Gbf,X;X^+)$. This implies that $\Gbf^\ad$ admits no $\Qbb$-factor isomorphic to $\PGL_{2,\Qbb}$. Write $\Mbar$ for the closure of $M$ in the Baily-Borel compactification $\Acal_g^\BB$, whose boundary components in $\Mbar \bsh M$ are of codimension at least 2. One may thus take a generic projective curve $C\subset\Mbar$ which meets $\Mbar \bsh M$ trivially.
 	
 	If the singular locus in $\TScalgn$, namely the intersection $\TScalgn^\sing:=\TScalgn\cap\Acal_g^\dec$, also meets $\Mbar$ in codimension at least 2, then we may further choose $C$ meeting $\TScalgn^\sing$ trivially. But this would produce a projective curve $C$ contained in $\TScalgn^\circ$ which contradicts \autoref{affiness}. Hence at least one of the irreducible components in $\TScalgn^\sing\cap M=\Acal_g^\dec\cap M$ is of codimension 1 in $M$.
 	
 	Sinc $\Acal_g^\dec$ is a finite union of Shimura subvarieties in $\Acal_g$, we deduce that $\Acal_g^\dec\cap M$ is a finite union of Shimura subvarieties, one of which is of codimension 1 in $M$. Hecke translation moves it into a proper Shimura subvariety $M'\subsetneq M$ contained generically in $\TScalgn$ of strictly lower dimension, and the reduction is proved.
 \end{proof}

 \section{Exclusion of Shimura curves generically in the superelliptic Torelli locus}\label{sec-proof-curve}
 In this section, we prove \autoref{thm-curve}.
 In \autoref{sec-represent}, we briefly recall the construction of the family of semi-stable
 curves representing a Shimura curve contained generically in the Torelli locus,
 and derive some general restraints for such families. 
 In \autoref{sec-non-p}, we prove \autoref{thm-curve}
 for the special case when $n=p$ is prime.
 In \autoref{sec-non-general}, we complete the proof of \autoref{thm-curve} for general $n$ by induction on
 the number of prime factors contained in $n$.

 \subsection{Representation of a Shimura curve by a family of semi-stable curves}\label{sec-represent} 

 In this section we associate a family of semi-stable curves to a given Shimura curve contained generically in the Torelli locus, and we analyze some numerical properties of its Higgs bundle. The construction of the family is similar to \cite[\S\,3]{lu-zuo-14}, which is briefly recalled for readers' convenience.

 For $\ell$ the fixed integer indicating the level structures as before,
 let $\calm^{ct}_g=\calm^{ct}_{g,\ell}\supseteq \calm_g=\calm_{g,\ell}$ be the partial compactification of
 the moduli space of smooth projective genus-$g$ curves  with level-$\ell$ structure by adding
   stable curves with  compact Jacobians.
 When $n\geq 3$, it carries a universal family of stable curves with compact Jacobians (cf. \cite{popp-77})
 \begin{equation*}
 \mathfrak f:~ \cals_g^{ct} \lra {\mathcal M}^{ct}_{g}.
 \end{equation*}
 The Torelli morphism $j^{\rm o}$ can be naturally extended to $\calm^{ct}_{g}$:
 $$ j:~\calm^{ct}_{g} \lra \cala_{g}, \qquad\text{with~}
 \calt_{g}=j\big(\calm^{ct}_{g}\big).$$
 The morphism $j^{\rm o}$ is 2:1 and ramified exactly over the locus of hyperelliptic curves (cf. \cite{os-79}).
 However, the relative dimension of $j$ is positive along the boundary $\calt_g\setminus \calt^{\rm o}_g$.

 Let $C$ be any smooth closed curve contained generically in $\calt_g$,
 and $B$ be the normalization of the strict inverse image $j^{-1}(C)$ of $C$.
 Denote by $j_B:\,B \to C$ the induced morphism.
 If $B$ is reducible, then we replace $B$ by one of its irreducible components.
 By pulling back the universal family $\mathfrak f:\,\cals_g^{ct} \to \calm^{ct}_{g}$
 to $B$ and resolving the singularities,
 we obtain a family  $f: S\to B$ of semi-stable curves that extends uniquely to a family
 $ \bar f:\ol S\to \ol B$ of semi-stable curves
 over the smooth compactification $\ol B\supseteq B$.
 \begin{definition}\label{defrepresenting}
 The family $\bar f:\,\ol S \to \ol B$ is called the family of semi-stable curves representing $C\subseteq \calt_g$
 via the Torelli morphism.
 \end{definition}

 We briefly recall some basic properties of the family $\bar f$ as follows, more details of which are found in \cite[\S\,3]{lu-zuo-14}.
 \begin{list}{}
 {\setlength{\labelwidth}{4mm}
 \setlength{\leftmargin}{6mm}}

 \item[(i)] Let $\ol C$ be the compactification of $C$ by joining a finite set of cusps $\Delta_{\ol C}$.
 The morphism $j_B: B\to C$ extends uniquely to a morphism $j_{\olb}: \ol B\to \ol C$
 such that $\Delta_{nc}:=\ol B\setminus B =j_{\olb}^{-1}(\Delta_{\ol C})$.
 Denote by $h:\,X \to C$ the universal family of abelian varieties over $C$,
 and by $\Upsilon \subseteq \ols$ the singular fibers over the  discriminate locus $\Delta\subseteq \olb$ of $\bar f$.
 Let $\mathbb V_{C}:=h_*\mathbb Q_X$ (resp. $\mathbb V_{B}:=\bar f_*\mathbb Q_{\ols\setminus \Upsilon}$) be the
 $\mathbb Q$-local system over $C$ (resp. $B$),
 and $\left(E_{\ol C}^{1,0}\oplus E_{\ol C}^{0,1},~\theta_{\ol C}\right)$
 and $\left(E_{\ol B}^{1,0}\oplus E_{\ol B}^{0,1},~\theta_{\ol B}\right)$
 be the corresponding logarithmic Higgs bundles via Simpson's correspondence over $\ol C$ and $\ol B$ respectively.
 Then
 \begin{equation}\label{eqn-3-98}
 \left(E_{\ol B}^{1,0}\oplus E_{\ol B}^{0,1},~\theta_{\ol B}\right)
 \cong j_{\olb}^*\left(E_{\ol C}^{1,0}\oplus E_{\ol C}^{0,1},~\theta_{\ol C}\right).
 \end{equation}

 \item[(ii)]
 The morphism $j_{B}$ is either an isomorphism or a double cover.
 In the first case,
 $$\deg E_{\ol B}^{1,0} = \deg E^{1,0}_{\ol C},\quad
 \deg \Omega_{\ol B}^1(\log\Delta_{nc})=\deg\Omega_{\ol C}^1(\log\Delta_{\ol C});$$
 and in the second case,
 \begin{equation*}
 \deg E^{1,0}_{\ol B} = 2\deg E^{1,0}_{\ol C},\quad
 \deg \Omega_{\ol B}^1(\log\Delta_{nc})=2\deg\Omega_{\ol C}^1(\log\Delta_{\ol C})+|\Lambda|,
 \end{equation*}
 where $\Lambda\subseteq B$ is the ramification locus of the double cover $j_B:\,B \to C$.
 Moreover, any fiber over $\Lambda$ is a semi-stable (possibly singular) hyperelliptic curve with a compact Jacobian.

 \item[(iii)] Let
 \begin{equation}\label{decompB}
 \left(E_{\ol B}^{1,0}\oplus E_{\ol B}^{0,1},\theta_{\ol B}\right)
 =\left(A_{\ol B}^{1,0}\oplus A_{\ol B}^{0,1},~\theta_{\ol B}\big|_{A_{\ol B}^{1,0}}\right)\oplus \left(F_{\ol B}^{1,0}\oplus F_{\ol B}^{0,1},~ 0\right)
 \end{equation}
 be the decomposition of the associated Higgs bundle into its ample and flat parts.
 Then
 \begin{equation*} 
 \qquad\begin{aligned}
 &~\text{the curve $C$ is a Shimura curve},\\
 \Longleftrightarrow&\left\{\begin{aligned}
 &\deg E_{\ol B}^{1,0}=\deg A_{\ol B}^{1,0}
 ={\rank A^{1,0}_{\ol B}\over 2}\cdot\deg\Omega^1_{\ol B}(\log\Delta_{nc}),&&\text{if $\deg(j_{B})=1$};\\
 &\deg E_{\ol B}^{1,0}=\deg A_{\ol B}^{1,0}
 ={\rank A^{1,0}_{\ol B}\over 2}\cdot\left(\deg\Omega^1_{\ol B}(\log\Delta_{nc})-|\Lambda|\right),
 &&\text{if $\deg(j_B)=2$}.
 \end{aligned}\right.
 \end{aligned}
 \end{equation*}
 In particular, if there is no hyperelliptic fiber over $B$, then $\Lambda=\emptyset$ and hence
 \begin{equation}\label{eqn-3-191}
 \quad \text{$C$ is a Shimura curve,}
 \quad \Longleftrightarrow\quad  \deg E_{\ol B}^{1,0}=\deg A_{\ol B}^{1,0}
 ={\rank A^{1,0}_{\ol B}\over 2}\cdot\deg\Omega^1_{\ol B}(\log\Delta_{nc}).
 \end{equation}

 \item[(iv)] If $C$ is a non-compact Shimura curve, then
 \begin{equation}\label{eqn-3-124}
 g(\ol F) = \rank F_{\ol B}^{1,0},\qquad
 \text{for any fiber $\ol F$ over $\Delta_{nc}=j_{\olb}^{-1}(\Delta_{\ol C})$},
 \end{equation}
 where $g(\ol F)$ is the geometric genus of $\ol F$.
 \end{list}

 \vspace{2mm}
 In the case when $C$ is contained generically in the superelliptic Torelli locus, 
 the family $\bar f$ constructed above is subject to   more restraints.
 In the rest part of this subsection, we always assume that $C$ is contained generically in $\mathcal{TS}_{g,n}$.

 Fix $\xi$ a primitive $n$-th root of 1 which gives the a trivialization $G\cong\Zbb/n\Zbb$.  For an $n$-superelliptic curve $\ol F$ defined as in \autoref{def-superelliptic},
 the map given by $y \mapsto \xi y$
 defines a natural action of the group $G\cong \mathbb Z/n\mathbb Z$
 on the $n$-superelliptic curve $\ol F$.%
 where $\xi$ is any primitive $n$-th root of $1$.
 We call $G$ an {\it $n$-superelliptic automorphism group} of $\ol F$,
 and the induced cover $\pi:\,\ol F \to \ol F/G\cong \bbp^1$ an {\it $n$-superelliptic cover}.
 To give an overview of the restraints,
 let us assume for the moment that the group $G\cong \mathbb Z/n\mathbb Z$
 admits an action on the surface $\ol S$ whose restriction on the general fiber
 is the $n$-superelliptic automorphism group.
 In fact, one can achieve this by a suitable finite base change (not necessarily \'etale);
 see \autoref{rems-3-1}\,(iii) below.

 Since the group $G\cong \mathbb Z/n\mathbb Z$ admits an action on the surface $\ol S$,
 it induces a natural action on the cohomological groups of $\ols$,
 and on the local system $\mathbb{V}_B:=R^1\bar f_*\mathbb Q_{\ols\setminus \Upsilon}$
 and hence also on its associated logarithmic Higgs bundle
 $\left(E_{\ol B}^{1,0}\oplus E_{\ol B}^{0,1},\,\theta_{\ol B}\right)$.
 In our case, the Higgs bundle has the form
 $$E^{1,0}_{\olb}\cong \bar f_*\Omega^1_{\ols/\olb}(\log \Upsilon),
 \qquad E^{0,1}_{\olb} \cong R^1\bar f_*\mathcal O_{\ols},$$
 and the Higgs field
 $$\theta_{\olb}:~E^{1,0}_{\olb} \lra E^{0,1}_{\olb}\otimes \Omega_{\olb}(\log \Delta_{nc})$$
 is induced by the edge morphism of the tautological sequence
 $$0\lra \bar f^*\Omega^1_{\ol B}(\log\Delta)\lra \Omega^1_{\ol S}(\log\Upsilon)
 \lra \Omega^1_{\ol S/\ol B}(\log\Upsilon) \lra 0.$$
 Consider the corresponding eigenspace decompositions
 \begin{equation}\label{eqn-3-38}
 \mathbb V_{B}\otimes \mathbb{C}=\bigoplus_{i=0}^{n-1} \mathbb V_{B,i};\qquad
 \left(E_{\ol B}^{1,0}\oplus E_{\ol B}^{0,1},~\theta_{\ol B}\right)
 =\bigoplus_{i=0}^{n-1} \left(E_{\ol B}^{1,0}\oplus E_{\ol B}^{0,1},~\theta_{\ol B}\right)_i.
 \end{equation}
 Since the quotient $\oly=\ols/G$ is ruled over $\olb$,
 it follows that
 $$\mathbb V_{B,0}=0,\qquad\text{and}\qquad \left(E_{\ol B}^{1,0}\oplus E_{\ol B}^{0,1},~\theta_{\ol B}\right)_0=0.$$
 Moreover, if we assume that the general fiber is given by $y^n=F(x)$ with a separable polynomial $F(x)$ of
 $\deg(F)=\alpha_0$,
 then according to the Riemann-Hurwitz formula and a formula due to Hurwitz-Chevalley-Weil (cf. \cite[Proposition\,5.9]{moonen-oort-13}), one has
 \begin{equation}\label{eqn-3-9}
 	g=\sum_{i=1}^{n-1} \rank E_{\ol B,i}^{1,0}=
 	\left\{\begin{aligned}
 	&\frac{(n-1)(\alpha_0-2)}{2},&\quad&\text{if $n\,|\,\alpha_0$},\\
 	&\frac{(n-1)(\alpha_0-2)+\big(n-\gcd(\alpha_0,n)\big)}{2},
 	&&\text{if $n{\not|}~\alpha_0$;}
 	\end{aligned}\right.
 \end{equation}
 \begin{equation}\label{eqn-3-35}
 \rank E_{\ol B,i}^{1,0}=\rank E_{\ol B,n-i}^{0,1}
 =\left\{\begin{aligned}
 &\frac{(n-i)\alpha_0}{n}-1, &&\text{if $n~|~\alpha_0$,\,~or $n{\not|}~\alpha_0$ and $\frac{(n-i)\alpha_0}{n}\in\mathbb Z$},\\
 &\Big[\frac{(n-i)\alpha_0}{n}\Big],&~&\text{if~}n{\not|}~\alpha_0\text{~and~}\frac{(n-i)\alpha_0}{n}\not\in\mathbb Z.
 \end{aligned}\right.
 \end{equation}
 By \cite{fujita-78-2,kollar-87},
 the Higgs bundle admits a natural decomposition
 $$\left(E_{\ol B}^{1,0}\oplus E_{\ol B}^{0,1},~\theta_{\ol B}\right)
 =\left(A_{\ol B}^{1,0}\oplus A_{\ol B}^{0,1},~\theta_{\ol B}\big|_{A_{\ol B}^{1,0}}\right)
  \oplus \left(F_{\ol B}^{1,0}\oplus F_{\ol B}^{0,1},~0\right),$$
 where $A_{\ol B}^{1,0}$ is ample,
 and $F_{\ol B}^{1,0}\oplus F_{\ol B}^{0,1}$ is flat corresponding to a unitary local subsystem
 $\mathbb{V}_{B}^u\subseteq \mathbb{V}_B\otimes \mathbb{C}$.
 The eigenspace decomposition on $\left(E_{\ol B}^{1,0}\oplus E_{\ol B}^{0,1},\,\theta_{\ol B}\right)$
 induces thus eigenspace decompositions on its associated subbundles:
 \begin{equation}\label{eqn-3-33}
 \left\{\begin{aligned}
 \left(A_{\ol B}^{1,0}\oplus A_{\ol B}^{0,1},~\theta_{\ol B}\big|_{A_{\ol B}^{1,0}}\right)&\,=
 \bigoplus_{i=1}^{n-1} \left(A_{\ol B}^{1,0}\oplus A_{\ol B}^{0,1},~\theta_{\ol B}\big|_{A_{\ol B}^{1,0}}\right)_i,\\[1mm]
 \left(F_{\ol B}^{1,0}\oplus F_{\ol B}^{0,1},~0\right)&\,=
 \bigoplus_{i=1}^{n-1} \left(F_{\ol B}^{1,0}\oplus F_{\ol B}^{0,1},~0\right)_i.
 \end{aligned}\right.
 \end{equation}

 By construction, each local subsystem $\mathbb V_{B,\,i}$
 is defined over the $n$-th cyclotomic field $\mathbb Q(\xi_n)$, where $\xi_n$ is a primitive $n$-th root of $1$.
 Thus the arithmetic Galois group ${\rm Gal\,}\big(\mathbb Q(\xi_n)/\mathbb Q\big)$
 has a natural action on the above decompositions \eqref{eqn-3-38} and \eqref{eqn-3-33}.

 \begin{lemma}\label{lem-3-12}
 	Let $\bar f:\,\ol S \to \ol B$ be the family of semi-stable curves representing
 	a Shimura curve $C$ contained generically in $\mathcal{TS}_{g,n}$ as above.
 	Let $\mathbb V_{B,\,i}^{tr} \subseteq \mathbb V_{B,\,i}$
 	be the trivial local subsystem,
 	and $\left(\big(F^{1,0}_{\olb,i}\big)^{tr}\oplus \big(F^{0,1}_{\olb,i}\big)^{tr},\,0\right)$
 	be the associated trivial flat subbundle.
 	If $\mathbb V_{B,\,i}$ and $\mathbb V_{B,\,j}$
 	are in one ${\rm Gal\,}\big(\mathbb Q(\xi_n)/\mathbb Q\big)$-orbit,
 	then
 	$$\begin{aligned}
 	\rank \mathbb V_{B,\,i}^{tr}&\,=\rank \mathbb V_{B,\,j}^{tr};\\
 	\rank \big(F^{1,0}_{\olb,i}\big)^{tr}+ \rank \big(F^{1,0}_{\olb,n-i}\big)^{tr}&\,=
 	\rank \big(F^{1,0}_{\olb,j}\big)^{tr}+ \rank \big(F^{1,0}_{\olb,n-j}\big)^{tr}.
 	\end{aligned}$$
 	In particular, if $n=p$ is prime, then for any $1\leq i<j\leq p-1$, one has
 	$$\begin{aligned}
 	\rank \mathbb V_{B,\,i}^{tr}&\,=\rank \mathbb V_{B,\,j}^{tr};\\
 	\rank \big(F^{1,0}_{\olb,i}\big)^{tr}+ \rank \big(F^{1,0}_{\olb,p-i}\big)^{tr}&\,=
 	\rank \big(F^{1,0}_{\olb,j}\big)^{tr}+ \rank \big(F^{1,0}_{\olb,p-j}\big)^{tr}.
 	\end{aligned}$$
 \end{lemma}
 \begin{proof}
 	Since trivial local subsystems correspond to trivial representations,
 	and trivial representations remain trivial under any Galois conjugation,
 	it follows that if $\mathbb V_{B,\,i}$ and $\mathbb V_{B,\,j}$
 	are in one ${\rm Gal\,}\big(\mathbb Q(\xi_n)/\mathbb Q\big)$-orbit, then
 	$$\begin{aligned}
 	\rank \mathbb V_{B,\,i}^{tr}&\,=\rank \mathbb V_{B,\,j}^{tr};\\
 	\rank \big(F^{1,0}_{\olb,i}\big)^{tr}+ \rank \big(F^{0,1}_{\olb,i}\big)^{tr}&\,=
 	\rank \big(F^{1,0}_{\olb,j}\big)^{tr}+ \rank \big(F^{0,1}_{\olb,j}\big)^{tr}.
 	\end{aligned}$$
 	Note also that $\left(\big(F^{1,0}_{\olb,i}\big)^{tr}\oplus \big(F^{0,1}_{\olb,i}\big)^{tr},\,0\right)$
 	is mapped isomorphically to $\left(\big(F^{1,0}_{\olb,n-i}\big)^{tr}\oplus \big(F^{0,1}_{\olb,n-i}\big)^{tr},\,0\right)$
 	under the complex conjugation for any $1\leq i\leq n-1$.
 	Moreover, under this isomorphism,
 	$\big(F^{1,0}_{\olb,i}\big)^{tr}\cong \big(F^{0,1}_{\olb,n-i}\big)^{tr}$
 	and $\big(F^{0,1}_{\olb,i}\big)^{tr}\cong \big(F^{1,0}_{\olb,n-i}\big)^{tr}$.
 	In particular,
 	$$\rank \big(F^{1,0}_{\olb,i}\big)^{tr}= \rank \big(F^{0,1}_{\olb,n-i}\big)^{tr},\qquad
 	\rank \big(F^{0,1}_{\olb,i}\big)^{tr}= \rank \big(F^{1,0}_{\olb,n-i}\big)^{tr}.$$
 	Combining the above equalities together, we prove the first part.	
 	For the second part, since $n=p$ is prime, it is clear that the  Galois subgroup
 	${\rm Gal}\big(\mathbb{Q}(\xi_p)/\mathbb{Q}\big)$ permutes these eigen-subspaces. This completes the proof.
 \end{proof}

 \begin{lemma}\label{lem-3-42}
 	Let $\bar f:\,\ol S \to \ol B$ be the family of semi-stable curves representing
 	a Shimura curve $C$ contained generically in $\mathcal{TS}_{g,n}$ as above.
     Then
 	\begin{equation}\label{eqn-3-102}
 	\rank A_{\olb,i}^{1,0}=\rank A_{\olb,i}^{0,1}=\rank A_{\olb,n-i}^{1,0}, \qquad\forall~1\leq i \leq n-1.
 	\end{equation}
 	In particular,
 	\begin{eqnarray}
 	\rank F_{\olb,i}^{1,0} &\neq &0, \hspace{2.9cm} \text{if~} \rank E^{1,0}_{\olb,i}>\rank E^{1,0}_{\olb,n-i}; \label{eqn-3-171}\\
 	\rank F_{\ol B,n-i}^{1,0} &\geq& \rank F_{\ol B,i}^{1,0}, \qquad\qquad \text{if~} i\geq n/2.\label{eqn-3-172}
 	\end{eqnarray}
 \end{lemma}
 \begin{proof}
 	Since $C$ is Shimura,
 	the associated Higgs bundle $\big(E_{\ol C}^{1,0}\oplus E_{\ol C}^{0,1},\theta_{\ol C}\big)$
 	admits a decomposition
 	$$\left(E_{\ol C}^{1,0}\oplus E_{\ol C}^{0,1},\theta_{\ol C}\right)
 	=\left(A_{\ol C}^{1,0}\oplus A_{\ol C}^{0,1},~\theta_{\ol C}\big|_{A_{\ol C}^{1,0}}\right)\oplus \left(F_{\ol C}^{1,0}\oplus F_{\ol C}^{0,1},~ 0\right),$$
 	such that the restricted Higgs field $\theta_{\ol C}\big|_{A_{\ol C}^{1,0}}$ is an isomorphism.
 	By \eqref{eqn-3-98}, one obtains that
 	$\big(A_{\ol B}^{1,0}\oplus A_{\ol B}^{0,1},~\theta_{\ol B}\big|_{A_{\ol B}^{1,0}}\big)$
 	is nothing but the pulling-back of
 	$\big(A_{\ol C}^{1,0}\oplus A_{\ol C}^{0,1},~\theta_{\ol C}\big|_{A_{\ol C}^{1,0}}\big)$.
 	In particular, the restricted Higgs field $\theta_{\ol B}\big|_{A_{\ol B}^{1,0}}$ is
 	an isomorphism on the generic point of $\olb$.
 	Restricting to each eigenspace
 	$\big(A_{\ol B}^{1,0}\oplus A_{\ol B}^{0,1},~\theta_{\ol B}\big|_{A_{\ol B}^{1,0}}\big)_i$,
 	one sees that the restricted Higgs field $\theta_{\ol B}\big|_{A_{\ol B,i}^{1,0}}$
 	must be again an isomorphism on the generic point of $\olb$.
 	In particular,
 	$\rank A_{\ol B,i}^{1,0}=\rank A_{\ol B,i}^{0,1}$,
 	i.e., the first equality of \eqref{eqn-3-102} holds.
 	The second equality in \eqref{eqn-3-102} follows by the complex conjugation.
 	%
 	%

     Finally, \eqref{eqn-3-171} follows directly from \eqref{eqn-3-102};
     and \eqref{eqn-3-172} follows from \eqref{eqn-3-102} and \eqref{eqn-3-35}.
 \end{proof}

 \begin{proposition}[Moonen]\label{prop-moonen}
 There does not exist any Shimura curve contained generically in $\mathcal{ST}_{g,n}$ with $n>g\geq 8$.
 \end{proposition}
 \begin{proof}
 	Assume that there exists a Shimura curve $C$ contained generically in $\mathcal{TS}_{g,n}$ with $n>g\geq 8$.
 	Let $\bar f:\,\ols \to \olb$ be the family of semi-stable $p$-superelliptic curves representing $C$ as in \autoref{defrepresenting}.
 	Assume that the general fiber of $\bar f$ is given by $y^n=F(x)$,
 	where $F(x)$ is a separable polynomial in $x$ with $\deg(F)=\alpha_0$.
 	By the Riemann-Hurwitz formula \eqref{eqn-3-9} one has $\alpha_0\leq 3$, since $n>g\geq 8$. 	 
 	It is also clear that $\alpha_0\geq 3$; otherwise $\bar f$ would be isotrivial.
 	Hence $\alpha_0=3$. However, such a family $\bar f$ must be universal in sense that
 	the moduli space of $n$-superelliptic curves defined
 	by $y^n=F(x)$ with $\deg(F)=3$ is exactly of dimension one.
 	Hence according to a result of Moonen \cite[Theorem\,3.6]{moonen-10},
 	the curve $C$ can not be Shimura once $g\geq 8$.
 	This completes the proof.
 \end{proof}

 The following two propositions give finer information on the ranks of $A_{\ol B}^{1,0}$ and $F_{{\ol B},i}^{1,0}$, the proofs of which are given later in Sections \ref{sec-pf-3-11} and \ref{sec-pf-3-13} respectively.

 \begin{proposition}\label{prop-3-11}
 	Let $\bar f:\,\ol S \to \ol B$ be the family of semi-stable curves representing
 	a Shimura curve $C$ contained generically in $\mathcal{TS}_{g,n}$ with $n\geq 3$.
 	Assume that the general fiber is given by $y^n=F(x)$ with a separable polynomial $F(x)$ of
 	$\deg(F)=\alpha_0$, and that $\alpha$ is given by {\red \eqref{eqn-3-45}}.
 	
 	{\rm(i)}. If $C$ is compact and $g\geq n$,
 	then
 	\begin{equation}\label{eqn-3-101-2}
 	\rank A_{\ol B}^{1,0} \leq \frac{4g-4}{\lambda_{n,c}},
 	\end{equation}
 	where
 	\begin{equation}\label{eqn-3-30}
 	\lambda_{n,c}=\left\{
 	\begin{aligned}
 	&12-\frac{9(\alpha-1)}{2(\alpha-3)}, &~&\text{if~}n=3 \text{~and~} \alpha\geq 6;\\
 	&12-\frac{3(\alpha-1)}{\alpha-3}, &&\text{if~}n=4 \text{~and~} \alpha=4k+3\text{~with~}k\geq 1;\\
 	&12-\frac{48(\alpha-1)}{\big(n^2-(n/d)^2\big)(\alpha-3)}, &&\text{otherwise; here~}
 	d=\left\{\begin{aligned}
 	&n,&&\text{if~}n\,|\,\alpha;\\[1mm]
 	&\frac{n}{\gcd(n,\alpha_0)},&&\text{if~}n{\not|}~\alpha.
 	\end{aligned}\right.
 	\end{aligned}\right.
 	\end{equation}
 	
 	{\rm(ii)}. Assume that $C$ is non-compact, $g\geq 4$ and $q_{\bar f}:=q(\ol S)-g(\olb)>0$.
 	If either $n=3$ or $4$, then
 	\begin{equation}\label{eqn-3-101-1}
 	\rank A_{\ol B}^{1,0} < \frac{4g-4}{\lambda_{n,nc}},
 	\end{equation}
 	where
 	\begin{eqnarray}
 	\lambda_{3,nc} &=& \left\{\begin{aligned}
 	&\frac{6\alpha-18}{\alpha-2},&\qquad&\text{if $\alpha=3k+2$ with $k\geq 2$},\\[0.5mm]
 	&\frac{15\alpha-63}{2(\alpha-3)}, &\quad&\text{otherwise};
 	\end{aligned}\right.\label{eqn-3-154}\\[1mm]
 	\lambda_{4,nc} &=& \left\{\begin{aligned}
 	&\frac{6\alpha-16}{\alpha-3}, &\,\qquad\,&\text{if $\alpha=4k+2$ with $k\geq 1$},\\[0.5mm]
 	&\frac{9\alpha-33}{\alpha-3},&&\text{otherwise}.
 	\end{aligned}\right.\label{eqn-3-155}
 	\end{eqnarray}
 \end{proposition}

 \begin{proposition}\label{prop-3-13}
 	Let $C$ and $\bar f$ be as in \autoref{lem-3-42}.
 	Assume that $\rank F_{\ol B,i_0}^{1,0} \neq 0$ for some $i_0\geq n/2$.
 	Then after a suitable finite \'etale base change, the flat Higgs subbundle
 	$$\quad \bigoplus_{i=n-i_0}^{i_0} \left(F_{\ol B}^{1,0}\oplus F_{\ol B}^{0,1},\,0\right)_i
 	\cong \left(\mathcal{O}_{\olb}^{\oplus r},\,0\right),
 	\qquad \text{where~}r=\sum_{i=n-i_0}^{i_0} \left(\rank F_{\ol B,i}^{1,0}+\rank F_{\ol B,i}^{0,1}\right),$$
 	becomes a trivial Higgs bundle, i.e.,
 	\begin{equation}\label{eqn-3-183}
 	\dim H^0\big(\ol S,\,\Omega^1_{\ol S}\big)_{i}=\rank F_{\ol B,i}^{1,0},\quad \forall~n-i_0\leq i\leq i_0,
 	\end{equation}
 	and there exists a unique fibration $\bar f':\,\ol S \to \ol B'$ such that
 	these one-forms in $H^0\big(\ols,\Omega_{\ols}^1\big)$
 	lifted from $\bigoplus\limits_{i=n-i_0}^{i_0}F_{\ol B,i}^{1,0}$
 	are the pulling-back of one-forms on $\olb'$ via $\bar f'$,
 	i.e.,
 	\begin{equation}\label{eqn-3-42-0}
 	\bigoplus_{i=n-i_0}^{i_0} H^0\big(\ol S,\,\Omega^1_{\ol S}\big)_{i}
 	\subseteq \big(\bar f'\big)^*H^0\big(\ol B',\,\Omega^1_{\ol B'}\big).
 	\end{equation}
 	%
 	%
 \end{proposition}

The next lemma gives a criterion to exclude Shimura curves generically in $\mathcal{TS}_{g,n}$.


 \begin{lemma}\label{lem-3-71}
 	Let $\bar f:\,\ol S \to \ol B$ be the family of semi-stable curves representing
 	a Shimura curve $C$ contained generically in $\mathcal{TS}_{g,n}$.
 	Assume that after a suitable base change of $\olb$,
 	there exists an irregular fibration $\bar f':\,\ols \to \olb'$ on $\ols$ different from $\bar f$ and with $g(\olb')\geq \rank F_{\ol B}^{1,0}$.
 	Then $g<8$.
 \end{lemma}
 \begin{proof}
 	This lemma is clear if $n=2$, since there is no Shimura curve contained generically in $\mathcal{TS}_{g,2}=\mathcal{TH}_g$ with $g\geq 8$ by \cite[Theorem E]{lu-zuo-14}. Combining this with \autoref{prop-moonen}, we may assume $g\geq n$ and $n\geq 3$ in the following.
 	
 	Firstly, we claim that
 	\begin{equation}\label{eqn-3-131}
 	2g(\ol F)-2\geq 2\big(2g(\olb')-2\big),\qquad \text{for any fiber $\ol F$ of $\bar f$},
 	\end{equation}
 	where $g(\ol F)$ is the geometric genus of $\ol F$.
 	In fact, by restricting $\bar f'$ to the fiber $\ol F$, one obtains a map
 	$$\bar f'|_{\ol F}:~\ol F \lra \olb'.$$
 	It is clear that $\deg(\bar f'|_{\ol F})$ does not depend on the choice of $\ol F$.
 	Since $\bar f$ is non-isotrivial,
 	it follows that $\deg(\bar f'|_{\ol F})\geq 2$.
 	Hence \eqref{eqn-3-131} follows directly from the Riemann-Hurwitz formula.
 	
 	Secondly, according to \autoref{lem-3-42} and \eqref{eqn-3-35}, one proves easily that
 	$$\rank F_{\ol B}^{1,0} \geq \sum_{i=1}^{[n/2]}\rank F_{\ol B,i}^{1,0}
 	\geq \sum_{i=1}^{[n/2]}\Big(\rank E_{\ol B,i}^{1,0}-\rank E_{\ol B,i}^{0,1}\Big)\geq 2,
 	\qquad \text{if~}g\geq 8 \text{~and~}n\geq 3.$$
 	
 	We now prove the lemma by contradiction. Assume that $g\geq 8$.
 	Consider first the case when $C$ is non-compact. In this case, by taking
 	an arbitrary fiber $\ol F$ over $\Delta_{nc}=j_{\olb}^{-1}(\Delta_{\ol C})$ in \eqref{eqn-3-131},
 	one obtains a contradiction to \eqref{eqn-3-124} since $g(\olb')\geq \rank F_{\ol B}^{1,0}\geq 2$.
 	
In the remaining case where $C$ is compact, we claim that
 	\begin{equation}\label{eqn-3-144}
 	g\geq 2g(\olb').
 	\end{equation}
 	In fact, by restricting $\bar f'$ to any singular fiber $\ol F$,
 	as we have $\deg(\bar f'|_{\ol F})\geq 2$,
 	we obtain that either there are at least two components of $\ol F$ whose geometric genera $\geq g(\olb')$,
 	or there is at least one component of $\ol F$ whose genus $\geq 2g(\olb')-1$
 	plus another component \big(contracted by $\bar f'|_{\ol F}$\big) of $\ol F$ with positive genus.
 	
 	Consider firstly the case when $n\geq 4$.
 	The assumption $g(\olb')\geq \rank F_{\ol B}^{1,0}$ together with \eqref{eqn-3-144}
 	gives $\rank A^{1,0}_{\olb} \geq \frac{g}{2}$.
 	On the other hand, since $g\geq n$,
 	one has the bound of $\rank A^{1,0}_{\olb}$ as in \eqref{eqn-3-101-2}.
 	This gives a contradiction.
 	
 	We now assume that $n=3$.
 	The assumption $g(\olb')\geq \rank F_{\ol B}^{1,0}$ implies the following lower bound of the
 	the relative irregularity: $q_{\bar f}=q(\ols)-g(\olb)\geq \rank F_{\ol B}^{1,0}$.
 	It follows that $q_{\bar f}=\rank F_{\ol B}^{1,0}$,
 	i.e., the flat subbundle $F_{\ol B}^{1,0} \cong \mathcal{O}_{\olb}^{\oplus q_{\bar f}}$ is trivial.
 	Moreover,
 	\begin{equation*} 
 	H^0\big(\ols,\Omega^1_{\ols}\big)=\bar f^*H^0\big(\olb,\Omega^1_{\olb}\big) ~\bigoplus~
 	(\bar f')^*H^0\big(\olb',\Omega^1_{\olb'}\big).
 	\end{equation*}
 	On the other hand,  up to a suitable finite base change,
 	we may assume that the group $G\cong \mathbb Z/3\mathbb Z$
 	admits an action on $\ols$, and the above decomposition still exists.
 	We claim that the group $G$ induces an action on $\olb'$
 	such that $\olb'/G \cong \bbp^1$ and that $\bar f'$ is equivariant with respect to $G$;
 	otherwise, one obtains a third fibration $\bar f''$ on $\ols$ by the action of $G$,
 	which implies that $q_{\bar f} > g(\olb')$, a contradiction.
 	In particular, one has
 	\begin{equation*} 
 	\rank F_{\ol B,i}^{1,0}=\dim H^0\big(\ol S,\,\Omega^1_{\ol S}\big)_{i}
 	=\dim H^0\big(\olb',\,\Omega^1_{\olb'}\big)_{i},
 	\qquad\text{for $i=1$ or $2$}.
 	\end{equation*}
 	Combining this with \eqref{eqn-3-102}, one obtains
 	\begin{equation}\label{eqn-3-203}
 	\rank E_{\ol B,1}^{1,0}-\rank E_{\ol B,2}^{1,0}=\rank F_{\ol B,1}^{1,0}-\rank F_{\ol B,2}^{1,0}
 	=\dim H^0\big(\olb',\,\Omega^1_{\olb'}\big)_{1}-\dim H^0\big(\olb',\,\Omega^1_{\olb'}\big)_{2}.
 	\end{equation}
 	Assume that
 	$\{x_1,\cdots,x_\beta\}\subseteq \bbp^1$ is the branch locus of
 	the induced cyclic cover $\pi:\,\ol B' \to \ol B'/G\cong \bbp^1$,
 	and that $\pi$ is defined by
 	$$\mathcal L_{\pi}^{\otimes 3} \equiv \mathcal O_{\bbp^1}\Big(\sum_{j=1}^{\beta} r_jx_j\Big),
 	\qquad\text{where $1\leq r_i\leq 2$ for each $1\leq i \leq \beta$}.$$
 	 Here `$\equiv$' stands for linear equivalence.
 	According to Hurwitz-Chevalley-Weil's formula (cf. \cite[Proposition\,5.9]{moonen-oort-13}),
 	one has
 	\begin{equation*}
 	\dim H^0\big(\ol B',\,\Omega^1_{\ol B'}\big)_{i}=-1+\sum_{j=1}^{\beta}
 	\left\{\frac{-ir_j}{3}\right\}.
 	\end{equation*}
 	Hence
 	$$\begin{aligned}
 	\dim H^0\big(\ol B',\,\Omega^1_{\ol B'}\big)_{1}-\dim H^0\big(\ol B',\,\Omega^1_{\ol B'}\big)_{2}
 	 =\,&\sum_{j=1}^{\beta}\left\{\frac{-r_j}{3}\right\}-\sum_{j=1}^{\beta}\left\{\frac{-2r_j}{3}\right\}
 	\leq \frac{2\beta}{3}-\frac{\beta}{3}=\frac{\beta}{3},\\[1mm]
 	\Longrightarrow\quad
 	\dim H^0\big(\ol S,\,\Omega^1_{\ol S}\big)_{1}-\dim H^0\big(\ol S,\,\Omega^1_{\ol S}\big)_{2}
 	\leq\,& \Big[\frac{\beta}{3}\Big].
 	\end{aligned}$$
 %
 	Note that $g=\alpha-2$ and $g(\olb')=\beta-2$ by the Riemann-Hurwitz formula.
 	Combining these with \eqref{eqn-3-35}, \eqref{eqn-3-144} and \eqref{eqn-3-203}, we obtain a contradiction.
 	This completes the proof.
 \end{proof}


 We conclude this subsection by the following remarks on the group action.
 \begin{remarks}\label{rems-3-1}
 	(i). For any $n$-superelliptic curve $\ol F$,
 	its induced $n$-superelliptic cover $\pi:\,\ol F\to \bbp^1$
 	is a cyclic cover with covering group $G\cong \mathbb Z/n\mathbb Z$,
 	branch locus $R$, and local monodromy $a$ around $R$.
 	Here $R$ and $a$ are given by
 	\begin{equation}\label{eqn-3-46}
 	\left\{\begin{aligned}
 	&R=\{x_1,\cdots,x_{\alpha_0}\},~\text{~and~}~a=(1,\cdots,1), &\qquad&\text{if $n\,|\,\alpha_0$};\\
 	&R=\{x_1,\cdots,x_{\alpha_0},\infty\},\text{~and~} a=(1,\cdots,1,a_{\infty}), && \text{if $n{\not|}~\alpha_0$},
 	\end{aligned}\right.\end{equation}
 	where $\alpha_0=\deg(F(x))$,
 	$\{x_1,\cdots,x_{\alpha_0}\}$ are the roots of $F(x)$,
 	and $a_{\infty}=n\left(\big[\frac{\alpha_0}{n}\big]+1\right)-\alpha_0$.
 	In the case when $n{\not|}~\alpha_0$, the ramification index of $\pi$ at $\infty$ is $r_{\infty}=\frac{n}{\gcd(n,\alpha_0)}$.
 	
 	(ii). For a given $n$-superelliptic curve $\ol F$,
 	there might be more than one $n$-superelliptic automorphism group
 	(equivalently, more than one $n$-superelliptic cover) on $\ol F$.
 	For instance, the Fermat curve of degree $n$ admits at least
 	three different $n$-superelliptic automorphism groups.
 	Nevertheless, the degree of the polynomial $F(x)$
 	does not depend on the choice of the $n$-superelliptic automorphism group.
 	Throughout this paper,
 	we always choose and fix the choise of an $n$-superelliptic automorphism group on $\ol F$
 	if there are more than one.
 	
 	(iii). Let $\bar f:\,\ol S \to \ol B$ be the family of semi-stable curves representing
 	a Shimura curve $C$ contained generically in $\mathcal{TS}_{g,n}$.
 	Although the general fiber of $\bar f$, which is an $n$-superelliptic curve by construction,
 	has an action of the group $G\cong \mathbb Z/n\mathbb Z$,
 	it is not known whether $G$ admits an action on $\ols$.
 	The problem is that there may not exist a rational section of $\Aut_{\olb}(\ols) \to \olb$
 	that reduces to a generator of $G$ on the general fiber,
 	where $\Aut_{\olb}(\ols)$ stands for the Zariski sheafification of the automorphism group, which is a finite flat group scheme.
 	However, one can always produce sections of $\Aut_{\olb}(\ols)$ after a suitable finite base change, and hence insures an action of $G$
 	on $\ols$  whose restriction on the general fiber is the $n$-superelliptic automorphism group locally for the fpqc topology.
 	
%
 	
 	(iv). If the general fiber admits a unique $n$-superelliptic automorphism group
 	$G\cong \mathbb Z/n\mathbb Z$, and there eixsts a generator of $G$
 	commuting with any automorphism of the general fiber,
 	then there is always a rational section of $\Aut_{\olb}(\ols)$
 	which reduces to this generator of $G$ on the general fiber,
 	i.e., the action of $G$ on the general fiber can be always extended to the global surface $\ols$
 	without any base change.
 	
 	(v). In order to use the theory on cyclic covers (eg. the eigenspace decomposition of the Higgs bundles),
 	it is necessary to assume that the group $G\cong \mathbb Z/n\mathbb Z$
 	acts on $\ol S$ whose restriction on the general fiber is
 	an $n$-superelliptic automorphism group.
 	To achieve this, it is necessary to take base change which might be non-\'etale.
 	This process may destroy the Arakelov type equality in \eqref{eqn-3-191}.
 	Nevertheless, the direct sum decomposition \eqref{decompB},
 	the formula \eqref{eqn-3-124},
 	the ranks of $A_{\ol B}^{1,0}$ and $F_{\ol B}^{1,0}$,
 	and the upper bounds of $A_{\ol B}^{1,0}$ in \eqref{eqn-3-101-2} and \eqref{eqn-3-101-1}
 	remain true after any finite base change.
 \end{remarks}

 \subsection{Non-existence of Shimura curves contained generically in $\mathcal{TS}_{g,p}$}\label{sec-non-p}
 In this subsection, we prove \autoref{thm-curve} for the prime case.
 The case where $p=2$ has already treated in \cite[Theorem\,E]{lu-zuo-14}.
 Hence we assume $p\geq 3$ and prove
 \begin{theorem}\label{thm-curve-prime}
 Let $p\geq 3$ be any prime number.
 Then there does not exist any Shimura curve contained generically in
 the Torelli locus of $p$-superelliptic curves of genus $g\geq 8$.
 \end{theorem}
 The main idea of the proof is based on a contradiction argument:
 given such a Shimura curve $C$,
 we first produce a ``horizontal'' irregular fibration on the family
 $\bar f:\,\ols \to \olb$ of semi-stable superelliptic curves representing $C$;
 and then we derive a contradiction from the existence of this ``horizontal'' irregular fibration.
 As we have explained in \autoref{sec-strategy-pf},
 the techniques depend on whether $C$ is compact or not.~
 We remark that the methods used here are different from that in proving \cite[Theorem\,E]{lu-zuo-14},
 which is deduced directly from the Miyaoka-Yau type inequality and an improved slope inequality
 for a family of hyperelliptic curves.

 \begin{proof}[Proof of \autoref{thm-curve-prime}]
 Assume that there exists a Shimura curve $C$ contained generically in $\mathcal{TS}_{g,p}$
 with $g\geq 8$ and $p\geq 3$ being a prime number.
 We are going to derive a contradiction.

 Let $\bar f:\,\ols \to \olb$ be the family of semi-stable $p$-superelliptic curves representing $C$ as in \autoref{defrepresenting}.
 After a possible base change,
 we may assume that
 there exists an action of the Galois group $G\cong \mathbb Z/p\mathbb Z$ on $\ols$,
 and hence an induced action of $G$ on the
 associated logarithmic Higgs bundle $\left(E_{\ol B}^{1,0}\oplus E_{\ol B}^{0,1},\,\theta_{\ol B}\right)$
 and its subbundles with eigenspace decompositions as in \eqref{eqn-3-38} and \eqref{eqn-3-33}.~
 Assume that the general fiber of $\bar f$ is given by $y^p=F(x)$,
 where $F(x)$ is a separable polynomial in $x$ with $\deg(F)=\alpha_0$.
 By \autoref{prop-moonen},
 we may assume that $g\geq p$, or equivalently $\alpha_0\geq 4$ by the Riemann-Hurwitz formula \eqref{eqn-3-9}.
 The detailed proof is divided into two cases, according to whether $C$ is compact or not.

 \vspace{2mm}
 {\bf Case (I):
 $C$ is non-compact.~}
 In this case, according to \autoref{lem-3-71}, it suffices to prove that, up to base change,
 the following two statements hold:
 \begin{enumerate}
 \item the flat subbundle $F_{\ol B}^{1,0}\cong \mathcal{O}_{\olb}^{\oplus r_1}$ becomes a trivial vector bundle, where $r_1=\rank F_{\ol B}^{1,0}$;
 \item there exists an irregular fibration $\bar f':\, \ols \to \olb'$ different from $\bar f$ with $g(\olb')=\rank F^{1,0}_{\olb}$.
 \end{enumerate}
 The first statement is already proved in \cite{viehweg-zuo-04};
 in fact, since $C$ is non-compact, according to \cite[Corollary\,4.4]{viehweg-zuo-04},
 after a suitable finite \'etale base change,
 the unitary local subsystem $\mathbb{V}_{B}^u\subseteq \mathbb{V}_{B} \otimes \mathbb{C}$ becomes trivial,
 which is equivalent to saying that the flat Higgs subbundle
 $\big(F_{\ol B}^{1,0}\oplus F_{\ol B}^{0,1},\,0\big)\cong \big(\mathcal{O}_{\olb}^{\oplus 2r_1},\,0\big)$
 is trivial by Simpson's correspondence \cite{sim90}.
 This proves the first statement.
 For the second statement, we divide its proof into the next two lemmas.

 \begin{lemma}\label{lem-3-72}
 \begin{equation}\label{eqn-3-142}
 \rank F^{1,0}_{\olb,(p+1)/2}>0,
 \end{equation}
 \end{lemma}
 \begin{proof}[Proof of \autoref{lem-3-72}]
 	Note that the validity of \eqref{eqn-3-142} is independent on the base change.
 	This allows us to take any finite base change.
 	As we have seen above, by \cite[Corollary\,4.4]{viehweg-zuo-04},
 	we may assume that the unitary local subsystem $\mathbb{V}_{B}^u\subseteq \mathbb{V}_{B} \otimes \mathbb{C}$
 	is trivial after a suitable finite base change, i.e., $\mathbb{V}_{B}^u=\mathbb{V}_{B}^{tr}$.
 	Combining this with \autoref{lem-3-12}, we obtain
 	$$\rank F^{1,0}_{\olb,i}+\rank F^{0,1}_{\olb,i}=\rank F^{1,0}_{\olb,j}+\rank F^{0,1}_{\olb,j},
 	\qquad \forall~1\leq i\leq j \leq p-1.$$
 	By \eqref{eqn-3-35}, one checks easily that
 	$$\rank E^{1,0}_{\olb,i}+\rank E^{0,1}_{\olb,i}=\rank E^{1,0}_{\olb,j}+\rank E^{0,1}_{\olb,j},
 	\qquad \forall~1\leq i\leq j \leq p-1.$$
 	Combining these with \eqref{eqn-3-102}, we obtain
 	\begin{equation}\label{eqn-3-175}
 	\rank A^{1,0}_{\olb,i}=\rank A^{1,0}_{\olb,j},\qquad \forall~1\leq i\leq j\leq p-1.
 	\end{equation}
 	Hence
 	$$\begin{aligned}
 	\rank F^{1,0}_{\olb,i}&\,=\rank E^{1,0}_{\olb,i}-\rank A^{1,0}_{\olb,i}&&\\
 	&\,=\rank E^{1,0}_{\olb,i}-\rank A^{1,0}_{\olb,p-1}\hspace{-2mm}
 	&=~&\rank E^{1,0}_{\olb,i}-\left(\rank E^{1,0}_{\olb,p-1}-\rank F^{1,0}_{\olb,p-1}\right)\\
 	&&\geq~&\rank E^{1,0}_{\olb,i}-\rank E^{1,0}_{\olb,p-1}.
 	\end{aligned}$$
 	
 	If $p\geq 5$, then by taking $i=(p+1)/2$ in the above inequality and by using \eqref{eqn-3-35},
 	one proves \eqref{eqn-3-142}.
 	It remains to show \eqref{eqn-3-142} for $p=3$.

 	In the case when $p=3$, we prove \eqref{eqn-3-142} by contradiction.
 	Suppose that $\rank F^{1,0}_{\olb,2}=0$. Then by \eqref{eqn-3-175}, one obtains
 	\begin{equation}\label{eqn-3-176}
 	\rank A^{1,0}_{\olb}=2\,\rank A^{1,0}_{\olb,2}=2\,\rank E^{1,0}_{\olb,2}.
 	\end{equation}
 	On the other hand, by \eqref{eqn-3-171} together with \eqref{eqn-3-35}, one gets $F^{1,0}_{\olb,1}\neq 0$.
 	Since $\mathbb{V}_{B}^u$ is a trivial local subsystem,
 	it follows from Simpson's correspondence that $F^{1,0}_{\olb}=F^{1,0}_{\olb,1}$ is a trivial vector bundle.
 	In other words, the relative irregularity $q_{\bar f}=\rank F^{1,0}_{\olb}>0$.
 	It follows that there is a bound on $\rank A^{1,0}_{\olb}$ as in \eqref{eqn-3-101-1},
 	which contradicts \eqref{eqn-3-176} in view of \eqref{eqn-3-35}.
 %
 \end{proof}

 \begin{lemma}\label{claim-3-31}
 If $\rank F_{\ol B,i_0}^{1,0} \neq 0$ for some $i_0>p/2$,
 then after a suitable base change,
 there exists a fibration $\bar f':\,\ol S \to \ol B'$ such that
 these one-forms in $H^0\big(\ols,\Omega_{\ols}^1\big)$
 lifted from $F_{\ol B}^{1,0}$ are the pulling-back of one-forms on $\olb'$ via $\bar f'$,
 	i.e.,
 	\begin{equation}\label{eqn-3-177}
 	\bigoplus_{i=1}^{p-1} H^0\big(\ol S,\,\Omega^1_{\ol S}\big)_{i}
 	= \big(\bar f'\big)^*H^0\big(\ol B',\,\Omega^1_{\ol B'}\big).
 	\end{equation}
 In particular,
 $$g(\olb')= \sum\limits_{i=1}^{p-1} \dim H^0\big(\ol S,\,\Omega^1_{\ol S}\big)_{i}=\rank F_{\ol B}^{1,0}.$$
 \end{lemma}
 \begin{proof}[Proof of \autoref{claim-3-31}]
 The existence of the fibration $\bar f'$ can be deduced from \autoref{prop-3-13}.
 Nevertheless, we present a complete proof for this simple case here for readers's convenience.

 First note that by our assumption together with \eqref{eqn-3-172},
 it follows that
 $$\rank F_{\ol B,p-i_0}^{1,0}\geq \rank F_{\ol B,i_0}^{1,0} >0, \qquad \text{~for some~} i_0>p/2.$$
 In other words, one obtains that
 the spaces of one-forms in $H^0\big(\ols,\Omega_{\ols}^1\big)$ lifted from
 $F_{\ol B,i_0}^{1,0}$ and $F_{\ol B,p-i_0}^{1,0}$ are both non-zero, i.e.,
 $$H^0\big(\ol S,\,\Omega^1_{\ol S}\big)_{i_0}\neq 0,\qquad H^0\big(\ol S,\,\Omega^1_{\ol S}\big)_{p-i_0}\neq 0.$$
 Taking any two non-zero one-forms $\omega_1\in H^0\big(\ol S,\,\Omega^1_{\ol S}\big)_{i_0}$
 and $\omega_2\in H^0\big(\ol S,\,\Omega^1_{\ol S}\big)_{p-i_0}$,
 the wedge product gives a $G$-invariant two-form $\omega_1\wedge \omega_2\in H^0\big(\ol S,\,\Omega^2_{\ol S}\big)^G$,
 and hence descends to a two-form on the ruled surface $\ols/G$.
 As any ruled surface admits no non-vanishing two-form,
 it follows that $\omega_1\wedge \omega_2=0$.
 Hence by Castelnuovo-de Franchis lemma (cf. \cite[Theorem\,IV-5.1]{bhpv-04}), there exists a fibration $\bar f':\,\ol S \to \ol B'$ such that
 $$H^0\big(\ol S,\,\Omega^1_{\ol S}\big)_{i_0} \oplus H^0\big(\ol S,\,\Omega^1_{\ol S}\big)_{p-i_0}
 \subseteq  \big(\bar f'\big)^*H^0\big(\ol B',\,\Omega^1_{\ol B'}\big).$$
 Note also that  the  pulling-back map
 $$(\bar f')^*: H^1(\olb',\,\mathbb Q)\lra H^1(\ols,\,\mathbb Q)$$
 is defined over $\mathbb Q$ and the Hodge symmetry under the complex conjucation gives
 $$\overline{H^0\big(\ol S,\,\Omega^1_{\ol S}\big)_{p-i_0}}=H^1\big(\ol S,\,\mathcal{O}_{\ols}\big)_{i_0}.$$
 Hence
 $$H^0\big(\ol S,\,\Omega^1_{\ol S}\big)_{i_0} \oplus H^1\big(\ol S,\,\mathcal{O}_{\ols}\big)_{i_0}
 = \left(H^1(\ols,\,\mathbb Q) \otimes \mathbb{C}\right)_{i_0}
 \subseteq  \big(\bar f'\big)^*\left(H^1(\olb',\,\mathbb Q)\otimes \mathbb{C}\right).$$
 Note that the arithmetic Galois group ${\rm Gal}(\ol{\mathbb{Q}}/\mathbb{Q})$ acts naturally on
 the eigenspace decomposition
 $$\left(H^1(\ols,\,\mathbb Q) \otimes \mathbb{C}\right)
 =\left(H^1(\ols,\,\mathbb Q) \otimes \mathbb{C}\right)_{0} \bigoplus
 \left(\bigoplus_{i=1}^{p-1}\left(H^1(\ols,\,\mathbb Q) \otimes \mathbb{C}\right)_{i}\right),$$
 and the eigen-subspaces $\left(H^1(\ols,\,\mathbb Q) \otimes \mathbb{C}\right)_{i}$'s for $1\leq i\leq p-1$
 are permuted by this action.
 In fact, the arithmetic Galois subgroup ${\rm Gal}\big(\mathbb{Q}(\xi_p)/\mathbb{Q}\big)$
 already acts transitively on these eigen-subspaces with indices $1\leq i\leq p-1$,
 where $\xi_p$ is a primitive $p$-th root of the unit.
 Therefore,
 $$\bigoplus_{i=1}^{p-1}\left(H^1(\ols,\,\mathbb Q) \otimes \mathbb{C}\right)_{i}
 \subseteq \big(\bar f'\big)^*\left(H^1(\olb',\,\mathbb Q)\otimes \mathbb{C}\right),$$
 which is in fact an equality,
 because $G$ induces an action on $\olb'$ with $\olb'/G\cong \bbp^1$
 due to the fact that $\ols/G$ is a (may be singular) ruled surface.
 Thus \eqref{eqn-3-177} is established by taking the $(1,0)$-part.
 \end{proof}

 %

 \vspace{2mm}
 {\bf Case (II):
 	$C$ is compact.~}
 This situation is much more complicated than the non-compact case.
 We divide the detailed proof into three steps.

 {\vspace{1mm} \it Step I.~}
 First we show that $i_m>p/2$, where
 $$i_m:=\max\big\{i~|~F^{1,0}_{\olb,i}\neq 0\big\},$$
 and hence by \autoref{prop-3-13},
 after a suitable base change, there is a unique fibration $\bar f':\,\ols\to\olb'$
 such that
 \begin{equation}\label{eqn-3-178}
 	\bigoplus_{i=p-i_m}^{i_m} H^0\big(\ol S,\,\Omega^1_{\ol S}\big)_{i}
 	\subseteq \big(\bar f'\big)^*H^0\big(\ol B',\,\Omega^1_{\ol B'}\big).
 \end{equation}
 	
 Indeed, if $i_m\leq p/2$, i.e., $F^{1,0}_{\olb,i}=0$ for any $i>p/2$.
 Then $A^{1,0}_{\ol B,i}=E^{1,0}_{\ol B,i}$ for all $i>p/2$.
 Combining this with \eqref{eqn-3-102}, one obtains
 \begin{equation}\label{eqn-3-34}
 \rank A^{1,0}_{\ol B}=2\sum_{i=(p+1)/2}^{p-1}\rank E^{1,0}_{\ol B,i}.
 \end{equation}
 By \eqref{eqn-3-35}, one  verifies that this contradicts the upper bound of $\rank A^{1,0}_{\ol B}$ given in in \eqref{eqn-3-101-2}.
To illustrate the idea, we give the proof for the case when $p\,|\,\alpha_0$.
 Let $\alpha_0=kp$ with $k\geq 1$. By \eqref{eqn-3-34} and \eqref{eqn-3-35}, one obtains
 $$
 \rank A^{1,0}_{\ol B}
 =2\sum_{i=(p+1)/2}^{p-1}\big(k(p-i)-1\big)
 =\frac{(p-1)\big(k(p+1)-4\big)}{4}.
 $$
 Since $g\geq 8$, it follows that $k\geq 4$ if $p=3$; and $k\geq 2$ if $p=5$.
 Hence this gives a contradiction to the bound of $\rank A^{1,0}_{\ol B}$ in \eqref{eqn-3-101-2}.

 {\vspace{1mm} \it Step II.~}
 Since the fibration $\bar f'$ obtained in Step I is unique,
 the group $G\cong \mathbb{Z}/p\mathbb{Z}$ induces an action on $\olb'$.
 In this step, we want to prove that a ``large part" of $1$-forms from $F^{1,0}_{\olb}$
 are pulled back of $1$-forms on $\olb'$ via $\bar f'$, i.e., we give a lower bound on $i_m$.
 More precisely, we show that

 \begin{lemma}\label{lem-3-28}
 	Let $\beta$ be the number of fixed points of $G$ on $\olb'$; equivalently, $\beta$ is the number of branch points of the induced cover $\pi:\,\olb' \to \olb'/G$. Then $\beta\geq 4$, and the following inequalities hold:
 	\begin{eqnarray}
 	i_m &\geq&\left\{\begin{aligned}
 	&p-1,&\quad&\text{if~}\beta>p,\\
 	&p-1-\left[\frac{p}{\beta-1}\right], &&\text{if~}\beta\leq p,
 	\end{aligned}\right.\label{eqn-3-126}\\[1mm]
 	2\,\rank E^{1,0}_{\olb,i_m} &\geq& \frac{2g}{p-1}+4-\beta,\label{eqn-3-127-2}\\
 	\frac{2g}{p-1} &\geq& 2\beta-3.\label{eqn-3-127-1}
 	\end{eqnarray}
 \end{lemma}
 \begin{proof}[Proof of \autoref{lem-3-28}]
 First, similar to \autoref{claim-3-31},
 one shows that $\olb'/G\cong \mathbb{P}^1$ and
 	\begin{equation*}
 	\bigoplus_{i=1}^{p-1} H^0\big(\ol S,\,\Omega^1_{\ol S}\big)_{i}
 	= \big(\bar f'\big)^*H^0\big(\ol B',\,\Omega^1_{\ol B'}\big).
 	\end{equation*}
 Moreover, the pulling-back map
 $\big(\bar f'\big)^*:\,H^0\big(\ol B',\,\Omega^1_{\ol B'}\big) \to H^0\big(\ol S,\,\Omega^1_{\ol S}\big)$
 is equivariant with respect to the induced actions of $G$ on both sides,
 i.e.,
 \begin{equation}\label{eqn-3-181}
 	H^0\big(\ol S,\,\Omega^1_{\ol S}\big)_{i}
 	= \big(\bar f'\big)^*H^0\big(\ol B',\,\Omega^1_{\ol B'}\big)_{i},
 	\qquad \forall~1\leq i\leq p-1.
 \end{equation}
 According to the definition of $i_m$, it follows that
 \begin{equation}\label{eqn-3-180}
 \dim H^0\big(\ol S,\,\Omega^1_{\ol S}\big)_{i_m}=\rank F^{1,0}_{\olb,i_m}>0;\qquad
 \dim H^0\big(\ol S,\,\Omega^1_{\ol S}\big)_{i}=\rank F^{1,0}_{\olb,i}=0,~\,~\forall~i>i_m.
 \end{equation}
 Hence
 \begin{equation}\label{eqn-3-179}
 \dim H^0\big(\olb',\,\Omega^1_{\olb'}\big)_{i_m}>0;\qquad
 \dim H^0\big(\olb',\,\Omega^1_{\olb'}\big)_{i}=0,\quad\forall~i>i_m.
 \end{equation}

 \vspace{0.5mm} (i). We first prove \eqref{eqn-3-126} by contradiction.
 Assume that
 $$i_m < j_0:=\left\{\begin{aligned}
 &p-1,&\quad&\text{if~}\beta>p,\\
 &p-1-\left[\frac{p}{\beta-1}\right], &&\text{if~}\beta\leq p,
 \end{aligned}\right.$$
 Then by \eqref{eqn-3-179}, one has
 \begin{equation}\label{eqn-3-24}
 H^0\big(\ol B',\,\Omega^1_{\ol B'}\big)_{i}=0,
 \qquad \text{for any~} p-1\geq i\geq j_0.
 \end{equation}
 Let $\{x_1,\cdots,x_\beta\}\subseteq \bbp^1$ be the branch locus of the induced quotient map
 $\pi:\,\olb' \to \olb'/G\cong \bbp^1$,
 and assume that $\pi$ is defined by
 $$\mathcal L_{\pi}^{\otimes p} \equiv \mathcal O_{\bbp^1}\Big(\sum_{j=1}^{\beta} r_jx_j\Big),
 \qquad\text{where~}1\leq r_1\leq \cdots \leq r_{\beta}\leq p-1.$$
 Then by a formula of Hurwitz-Chevalley-Weil (cf. \cite[Proposition\,5.9]{moonen-oort-13}),
 one has
 \begin{equation}\label{eqn-3-66}
 \dim H^0\big(\ol B',\,\Omega^1_{\ol B'}\big)_{i}=-1+\sum_{j=1}^{\beta}
 \left\{\frac{-ir_j}{p}\right\}.
 \end{equation}
 By \eqref{eqn-3-24}, we get
 \begin{equation*}
 H(k):=\sum_{j=1}^{\beta} \left\{\frac{kr_j}{p}\right\}=1,
 \qquad \text{for any~} 1\leq k\leq p-j_0.
 \end{equation*}
 This contradicts \autoref{lem-3-101} below.

 \vspace{0.5mm} (ii).
 We next prove $\beta\geq 4$ and \eqref{eqn-3-127-2}.
 By \autoref{prop-3-13} and \eqref{eqn-3-181}, it follows that
 $$\dim H^0\big(\olb',\,\Omega^1_{\olb'}\big)_{i}=\dim H^0\big(\ol S,\,\Omega^1_{\ol S}\big)_{i}=\rank F^{1,0}_{\olb,i},
 \qquad \forall~p-i_m\leq i\leq i_m.$$
 According to \eqref{eqn-3-102} together with \eqref{eqn-3-180}, one obtains
 \begin{equation}\label{eqn-3-182}
 \rank A^{1,0}_{\olb,p-i_m}=\rank A^{0,1}_{\olb,i_m}=\rank A^{1,0}_{\olb,i_m}\leq \rank E^{1,0}_{\olb,i_m}-1.
 \end{equation}
 Combining these with \eqref{eqn-3-66} and \eqref{eqn-3-35}, we obtain
 \begin{equation}\label{eqn-3-205}
 	\beta-2=\dim H^0\big(\olb',\,\Omega^1_{\olb'}\big)_{i_m}+\dim H^0\big(\olb',\,\Omega^1_{\olb'}\big)_{p-i_m}
 	=\rank F^{1,0}_{\olb,i_m}+\rank F^{1,0}_{\olb,p-i_m}.
 \end{equation}
 By the definition of $i_m$ with \eqref{eqn-3-35}, one has $\rank F^{1,0}_{\olb,p-i_m}\geq \rank F^{1,0}_{\olb,i_m}\geq 1$.
 From this with \eqref{eqn-3-205}, it follows that $\beta\geq 4$.
 Moreover,
 $$\begin{aligned}
 \rank F^{1,0}_{\olb,i_m}+\rank F^{1,0}_{\olb,p-i_m}
 =\,&\rank E^{1,0}_{\olb,i_m}+\rank E^{1,0}_{\olb,p-i_m}
 -2\,\rank A^{1,0}_{\olb,i_m}\\
 \geq\,& \rank E^{1,0}_{\olb,i_m}+\rank E^{1,0}_{\olb,p-i_m}
 -2\left(\rank E^{1,0}_{\olb,i_m}-1\right)\\
 =\,&\frac{2g}{p-1}+2-2\,\rank E^{1,0}_{\olb,i_m}.
 \end{aligned}$$
 This together with \eqref{eqn-3-205} proves \eqref{eqn-3-127-2}.

 \vspace{0.5mm} (iii).
 Finally, we prove \eqref{eqn-3-127-1}.
 Let $\ol F$ be a general fiber of $\bar f$, and $\ol\Gamma=\ol F/G\cong \bbp^1$ the quotient.
 Then one has the following commutative diagram, where $\bar\varphi':\,\ols/G \to \olb'/G$ is the induced fibration.
 $$\xymatrix{
 	\ol F \ar[rrr]^-{\bar f'|_{\ol F}} \ar[d]_-{\Pi|_{\ol F}}&&&\ol B' \ar[d]^-{\pi'}  \\
 	\ol\Gamma\ar[rrr]^-{\bar\varphi'|_{\ol \Gamma}}\cong \bbp^1 &&& \ol B'/G\cong \bbp^1
 }$$
 By assumption, $\Pi|_{\ol F}$ (resp. $\pi'$) is branched over $\alpha:=\frac{2g}{p-1}+2$ (resp. $\beta$) points.

 If $\deg\big(\bar f'|_{\ol F}\big)\geq p$,
 then by the Rieman-Hurwitz formula for the map $\bar f'|_{\ol F}$, one obtains
 $$(p-1)\alpha-2p=2g-2\geq \deg\big(\bar f'|_{\ol F}\big)\cdot \big(2g(\olb')-2\big)\geq p\big((p-1)\beta-2p\big)
 \geq 2(p-1)\beta-4p+2.$$
 Hence $\alpha\geq 2(\beta-1)$;
 and if the equality holds, then $p=3$, $\alpha=6$ and $g=4$, which contradicts the assumption that $g\geq 8$.

 If $\deg\big(\bar f'|_{\ol F}\big)< p$,
 then the inverse of the branch points of $\pi'$ in $\ol\Gamma\cong \bbp^1$ is contained in that of $\Pi|_{\ol F}$.
 Let $R_0$ be the ramification locus of $\bar\varphi'|_{\ol \Gamma}$.
 Then by the Rieman-Hurwitz formula, one has
 \begin{equation}\label{eqn-3-152}
 \deg(\bar\varphi'|_{\ol \Gamma})\cdot \beta-\alpha \leq
 \deg(R_0)=2\deg(\bar\varphi'|_{\ol \Gamma})-2.
 \end{equation}
 Since $\bar f$ is non-isotrivial,
 one has $\deg\big(\bar\varphi'|_{\ol \Gamma}\big)=\deg\big(\bar f'|_{\ol F}\big)\geq 2$.
 Hence $\alpha\geq 2(\beta-1)$.
 Moreover, if $\alpha= 2(\beta-1)$, then
 $\bar\varphi'|_{\ol \Gamma}$ is a double cover branched exactly over two of the branch points of $\pi'$.
 It follows that the branch loci of $\Pi|_{\ol F}$ are invariant when $\ol F$ runs in the family $\bar f$,
 and hence any two smooth fibers of $\bar f$ are isomorphic to each other.
 This contradicts the non-isotriviality of $\bar f$.
 Thus $\alpha>2(\beta-1)$.
 This proves \eqref{eqn-3-127-1}.
 \end{proof}

 {\vspace{1mm} \it Step III.~}
 In the last step, we derive a contradiction and hence complete
 the proof of \autoref{thm-curve-prime} for the case when $C$ is compact.

 First, from \autoref{lem-3-71} and \autoref{prop-3-13}, it follows that $p\geq 5$ and $\beta\leq p$;
 in fact, if $p=3$ or $\beta>p$,
 then $i_m=p-1$ by Step I and \eqref{eqn-3-126}, and hence $g<8$ by \autoref{lem-3-71} and \autoref{prop-3-13}.

 According to \eqref{eqn-3-126}, \eqref{eqn-3-127-2},\eqref{eqn-3-9} and \eqref{eqn-3-35}, one obtains that
 \begin{equation}\label{eqn-3-68}
 \left\{\begin{aligned}
 \alpha_0+4-\beta&~\leq  \frac{2\big([\frac{p}{\beta-1}]+1\big)\alpha_0}{p},&\quad&\text{if~}p\,|\,\alpha_0;\\
 \alpha_0+3-\beta&~\leq 2\Big[\frac{\big([\frac{p}{\beta-1}]+1\big)\alpha_0}{p}\Big], &&\text{if~}p{\not|}~\alpha_0.
 \end{aligned}\right.
 \end{equation}
 Note that $[x]\leq x$ for any $x\in \mathbb Q$.
 Combining this with \eqref{eqn-3-9} and \eqref{eqn-3-127-1}, we get
 $$\left\{\begin{aligned}
 \Big(1-\frac{4}{p}\Big)\beta&~\leq 1+\frac{2}{\beta-1}-\frac{2}{p},&\quad&\text{if~}p\,|\,\alpha_0;\\
 \Big(1-\frac{4}{p}\Big)(\beta-1)&~\leq 2, &&\text{if~}p{\not|}~\alpha_0.
 \end{aligned}\right.$$
 Note also that $\beta\geq 4$ by \autoref{lem-3-28}. 
 Hence the above inequalities give a contradiction if $p>11$.
 If $p=11$ or $7$, one verifies case-by-case that there is also a contradiction
 by \eqref{eqn-3-68}, \eqref{eqn-3-127-1}, \eqref{eqn-3-9} and \eqref{eqn-3-35}.

 Finally, we consider the case when $p=5$.
 Again by \eqref{eqn-3-68}, \eqref{eqn-3-127-1}, \eqref{eqn-3-9} together with \eqref{eqn-3-35},
 one obtains that $g=14$, $\alpha_0=8$, $\beta=5$ and $i_m=3$.
 In particular, \eqref{eqn-3-127-2} is an equality,
 which implies that \eqref{eqn-3-182} is also an equality, i.e.,
 $$\rank A^{1,0}_{\olb,2}=\rank A^{1,0}_{\olb,3}=\rank E^{1,0}_{\olb,3}-1=2, \quad \text{by \eqref{eqn-3-35}}.$$
 Note also that
 $$\rank A^{1,0}_{\olb,1}=\rank A^{1,0}_{\olb,4}=\rank E^{1,0}_{\olb,4}=1.$$
 Hence $\rank A^{1,0}_{\olb}=\sum\limits_{i=1}^{4} \rank A^{1,0}_{\olb,i}=6$,
 which is a contradiction to the bound \eqref{eqn-3-101-2}.
 This completes the proof.
 \end{proof}

 We end this subsection by proving the following technical lemma which has been used in the proof of \autoref{lem-3-28}.
 \begin{lemma}\label{lem-3-101}
 	Let $p\geq 3$ be any prime number, and $1\leq r_1\leq \cdots \leq r_{\beta}\leq p-1$ be a sequence of integers such that $p~\big|~\Big(\sum\limits_{j=1}^{\beta} r_j\Big)$.
 	Let $1\leq \theta \leq p-1$ be an integer such that
 	\begin{equation}\label{eqn-3-204}
 	H(k)=1,
 	\quad\forall~ 1\leq k\leq \theta;
 	\qquad\text{~where~}H(k):=\sum_{j=1}^{\beta} \left(\frac{kr_j}{p}-\left[\frac{kr_j}{p}\right]\right)=\sum_{j=1}^{\beta}\left\{\frac{kr_j}{p}\right\}.
 	\end{equation}
 	Then $\beta\leq p$ and $\theta\leq \left[\frac{p}{\beta-1}\right]$.
 \end{lemma}
 \begin{proof} The case when $\theta=1$ is clear, and we may assume that $\theta\geq2$.

 	Taking $k=1$ in \eqref{eqn-3-204}, we get immediately that $\beta \leq \sum\limits_{j=1}^{\beta}r_j=p$;
    and from the equality $H(2)=1$, we obtain that $r_{\beta}>\frac p2$ and $r_j< \frac p2$ for $1\leq j \leq \beta-1$.
 	In the following we deduce a contradiction under the assumption  $\theta>[\frac p{\beta-1}]$.
 	
 	\vspace{1.5mm}
 	(Step 1) 	First of all, we show that
 	\begin{equation}\label{eqn-3-210}
 	r_j=1,\qquad \forall~1\leq j \leq \beta-2.
 	\end{equation}
    We set
 	$$\delta=\Big[\frac p{\beta-1}\Big], ~t_1=\Big[\frac p{r_{\beta-1}}\Big],~
 	  t_2=\Big[\frac p{r_{\beta-2}}\Big], \text{~and~}~t_2'=\Big[\frac p{2r_{\beta-2}}\Big].$$
 	By assumption, $\delta+1\leq \theta$.
 	It is clear that $2t_2' \leq t_2 \leq 2t_2'+1$, and $t_1\leq t_2$.

 	Moreover,  $t_1 \leq \delta$; otherwise, $\frac{(\delta+1)r_j}p<1$ for any $1\leq j \leq \beta-1$, and it implies
 	$$1=H(\delta+1)\geq \sum\limits_{j=1}^{\beta-1} \frac{(\delta+1)r_j}p \geq \frac{\beta-1}p \cdot (\delta+1)> 1,$$ which is a contradiction.
 	Thus $t_1 \leq \min \{\delta,t_2\}$, from which together with \eqref{eqn-3-204} it follows that
 	$$1=H(t_1)> \frac{t_1 r_{\beta-2}}p+\frac{t_1 r_{\beta-1}}p >\frac{t_1 r_{\beta-2}}p+ \frac 12,
 	\quad\Longrightarrow\quad \frac{p}{2r_{\beta-2}}>t_1.$$
 	Hence $t_1\leq  t_2'$.

 	We claim also that there exists some $t_0$ with $t_2'< t_0<t_2+1$ such that $\big\{\frac{t_0 r_{\beta-1}}p \big\}> \frac12$.
 	In fact, if such $t_0$ does not exist, then by induction one has
 	$\big[\frac{t r_{\beta-1}}p\big]=\big[\frac{(t_2'+1) r_{\beta-1}}p\big]$ for any $t_2'< t<t_2+1$,
 	since $\frac{r_{\beta-1}}p  < \frac12$.
 	Hence
 	$$\begin{aligned}
 		\frac 12 \geq \Big\{\frac{t_2r_{\beta-1}}p\Big\}&\,=\Big\{\frac{(t_2'+1)r_{\beta-1}}p\Big\}+\frac{(t_2-t_2'-1)r_{\beta-1}}p\\
 		&\,> \frac{(t_2-t_2'-1)r_{\beta-1}}p \geq \frac{(t_2'-1)r_{\beta-1}}p \geq \frac{(t_1-1)r_{\beta-1}}p.
 	\end{aligned}$$
 	Note that $t_1=\big[\frac p{r_{\beta-1}}\big] \geq 2$.
 	From the above inequality it follows that $t_2=2t_2'=2t_1=4$,
 	in which case one computes easily that $r_{\beta}> \frac p2$, $r_{\beta-1}> \frac p3$ and $r_{\beta-2}> \frac p5$.
 	This contradicts the fact that $\sum\limits_{j=1}^{\beta}r_j=p.$

 	Now since
 	$$H(t_0)>\frac {t_0r_{\beta-2}}p+\Big\{\frac{t_0r_{\beta-1}}p\Big\}>\frac 12+\frac12=1,$$
 	one obtains that $t_2\geq t_0>\delta$.
 	Because
 	$$1=H(\delta+1)> \sum_{j=1}^{\beta-2}\frac{(\delta+1)r_j}p,$$
 	it follows that $r_j=1$ for any $1\leq j \leq \beta-2$.
 	
 	\vspace{1.5mm}
 	(Step 2) We show that
 	\begin{equation}\label{eqn-3-211}
 	\epsilon +t_1\geq \delta+1, \qquad \text{~where~} \epsilon=\big[\frac p{2(\beta-2)}\big].
 	\end{equation}
 	Indeed, completely similar to the estimation in (Step 1), one can show that there exists some $\tilde t_0$
 	with $\epsilon+1 \leq  \tilde t_0 \leq \epsilon+t_1+1$ such that
 	$\big\{\frac{\tilde t_0 r_{\beta-1}}p\big\} >\frac 12$.
 	For such $\tilde t_0$, we have
 	$$H(\tilde t_0)=\frac{(\beta-2)\tilde t_0}p+\left(\Big\{\frac{\tilde t_0 r_{\beta-1}}p\Big\}+
 	\Big\{\frac{\tilde t_0 r_{\beta}}p\Big\}\right)> \frac{(\beta-2)(\epsilon+1)}p +\frac 12>\frac12+\frac 12=1.$$
 	Hence we obtain $\epsilon+t_1 \geq \tilde t_0-1> \theta-1 > \delta$, i.e., $\epsilon +t_1\geq \delta+1$ as required.
 	
	\vspace{1.5mm}
	(Step 3) We proceed to show that
	\begin{equation}\label{eqn-3-212}
		p \geq t_1(t_1+1)(\beta-2)+2t_1+1.
	\end{equation}
	In fact, since $t_1\leq \delta<\theta$, by \eqref{eqn-3-204} with $k=t_1$ and using \eqref{eqn-3-210} one obtains that
 	$$1=H(t_1)>\frac{(\beta-2)t_1}p+\frac{t_1 r_{\beta-1}}p=\frac{(\beta-2)t_1+p-\eta}p,$$
 	where we write $\eta=p-t_1 r_{\beta-1}$. Thus $\eta>(\beta-2)t_1$, i.e. $\eta\geq (\beta-2)t_1+1$. Therefore,
 	\begin{equation*}
 	p=t_1 r_{\beta-1}+\eta \geq t_1 (\eta+1)+\eta \geq t_1(t_1+1)(\beta-2)+2t_1+1.
 	\end{equation*}
 	
 	\vspace{1.5mm}
 	(Step 4) Finally, we derive a contradiction.~
 	Clearly, we may assume that $\beta\geq 3$.
 	Moreover, if $\beta=3$, then $\delta=\big[\frac{p}{2}\big]=\frac {p-1}2$. Since $\theta > \delta$, one has
 	$H(\delta)+H(\delta+1)=2$ by \eqref{eqn-3-204}; on the other hand,  direct computation gives us
 	$H(\delta)+H(\delta+1)=\beta=3$. Hence we may assume that $\beta\geq 4$.
 	Combining \eqref{eqn-3-211} and \eqref{eqn-3-212}, we obtain
 	\begin{equation}\label{eqn-3-213}
 	\frac p{(t_1+1)(\beta-2)+\frac{2t_1+1}{t_1}}\geq t_1
 	\geq \Big[\frac p{\beta-1}\Big]-\Big[\frac p{2(\beta-2)}\Big]+1.
 	\end{equation}
 	Note that $\big[\frac p{\beta-1}\big]-\big[\frac p{2(\beta-2)}\big]+1 > \frac p{\beta-1}- \frac p{2(\beta-2)}$.
 	Hence
 	$$\frac{p}{\beta-1}- \frac p{2(\beta-2)} < \frac{p}{(t_1+1)(\beta-2)+\frac{2t_1+1}{t_1}}<\frac{p}{(t_1+1)(\beta-2)}.$$
 	Since $r_{\beta-1}<\frac{p}{2}$, or equivalently $t_1\geq 2$, one derives immediately a contradiction if $\beta> 6$.
 	For the cases when $4\leq \beta \leq 6$, one can derive a contradiction case-by-case according to \eqref{eqn-3-213}.
 	This completes the proof.
%
%
%
 \end{proof}

 \subsection{Non-existence of Shimura curves contained generically in $\mathcal{TS}_{g,n}$}\label{sec-non-general} In this subsection we prove \autoref{thm-curve} for $n$-superelliptic curves by induction on the number of prime factors of $n$.
 Before entering the details, we first explain the main idea of the induction process.

 %
 %

 Let $C$ be any smooth curve contained generically in $\mathcal{TS}_{g,n}$,
 and $\bar f:\,\ols \to \olb$ be the family of semi-stable $n$-superelliptic curves representing $C$
 as in \autoref{sec-represent}.
 As before, we choose and fix an $n$-superelliptic automorphism group
 for the general fiber of $\bar f$ if it has more than one such automorphism group.
 After a possible base change, we may assume that
 the group $G\cong \mathbb Z/n\mathbb Z$ admits an action on $\ols$
 which reduces to the superelliptic automorphism group on the general fiber of $\bar f$.
 Let $n_1\geq 2$ be any number dividing $n$,
 and consider the quotient family $\ols/H_1 \to \olb$,
 where $H_1 \leqslant G$ is the unique subgroup of order $\frac{n}{n_1}$.
 Resolving the singularities of $\ols/H_1$ and contracting the exceptional curves,
 one obtains a new family $\bar f_1:\,\ols_1 \to \olb$.
 By construction, $\bar f_1$ is also semi-stable,
 and there is a rational cover $\ol\Pi_{n,n_1}$ with the following diagram.
 $$\xymatrix{
 	\ols \ar@{-->}[rr]^-{\ol\Pi_{n,n_1}} \ar[dr]_-{\bar f} && \ols_1 \ar[dl]^-{\bar f_1}\\
 	&\olb&}$$
 If the general fiber of $\bar f$ is defined by $y^n=F(x)$,
 then the general fiber of $\bar f_1$ is given by $y^{n_1}=F(x)$,
 which admits a cyclic cover $\pi_1$ to $\bbp^1$
 with covering group $G_1\cong \mathbb Z/n_1\mathbb Z$,
 branch locus $R_1$, and local monodromy $a_1$ around $R_1$.
 Here $R_1$ and $a_1$ are given by
 \begin{equation}\label{eqn-3-96}
 \left\{\begin{aligned}
 &R_1=\{x_1,\cdots,x_{\alpha_0}\},~\text{~and~}~a_1=(1,\cdots,1), &\qquad&\text{if $n_1\,|\,\alpha_0$};\\
 &R_1=\{x_1,\cdots,x_{\alpha_0},\infty\},\text{~and~} a_1=(1,\cdots,1,a_{\infty,1}), && \text{if $n_1{\not|}~\alpha_0$},
 \end{aligned}\right.\end{equation}
 where $\{x_1,\cdots,x_{\alpha_0}\}$ are the set of roots of $F(x)$,
 and $a_{\infty,1}=n_1\left(\big[\frac{\alpha_0}{n_1}\big]+1\right)-\alpha_0$.
 In the case when $n_1{\not|}~\alpha_0$, the ramification index of $\pi_1$
 at $\infty$ is $r_{\infty,1}=\frac{n_1}{\gcd(n_1,\alpha_0)}$.
 By Hurwitz formula, the genus $g_1$ of a general fiber of $\bar f_1$ is given by the following formula.
 \begin{equation}\label{eqn-3-9(1)}
 g_1=\left\{\begin{aligned}
 &\frac{(n_1-1)(\alpha_0-2)}{2},&~&\text{if $n_1\,|\,\alpha_0$},\\
 &\frac{(n_1-1)(\alpha_0-2)+\frac{r_{\infty,1}-1}{r_{\infty,1}}\cdot n_1}{2},
 &&\text{if $n_1{\not|}~\alpha_0$.}
 \end{aligned}\right.
 \end{equation}
 The relative Jacobian of the family $\bar f_1$ induces a map from $B$
 to $\mathcal{TS}_{g_1,n_1}\subseteq \mathcal{A}_{g_1}$,
 which factorizes clearly through $C$:
 \begin{equation}\label{eqn-3-123}
 B \lra C \overset{\rho_{n,n_1}}\lra \mathcal{TS}_{g_1,n_1}\subseteq \mathcal{A}_{g_1}.
 \end{equation}
 We denote the image by $\rho_{n,n_1}(C)$.
 By definition, one obtains
 \begin{lemma}\label{lem-3-15}
 	Let $C\subseteq \mathcal{TS}_{g,n}$ be a Shimura curve.
 	Then the image $\rho_{n,n_1}(C)\subseteq \mathcal{TS}_{g_1,n_1}$ is either a special point,
 	or it is a Shimura curve.
 	Moreover, if $\rho_{n,n_1}(C)$ is a Shimura curve,
 	then $C$ is compact iff $\rho_{n,n_1}(C)$ is compact.
 \end{lemma}
 \begin{proof}
 	This follows directly from the definition of Shimura subvarieties  and the characterization of Shimura curves in $\Acalg$.
 \end{proof}

 \begin{remark}
 	The image $\rho_{n,n_1}$ might depend on the choices
 	of the $n$-superelliptic automorphism group on the general fiber of $\bar f$.
 	In other words, it is not clear whether there is a well-defined map
 	from $\mathcal{TS}_{g,n}$ to $\mathcal{TS}_{g_1,n_1}$.
 	In this paper, when talking about the map $\rho_{n,n_1}$,
 	it is understood that  an $n$-superelliptic automorphism group on the general fiber of $\bar f$
 	has already been chosed and fixed.
 \end{remark}

 The principle of the induction process is the following.
 \begin{proposition}\label{prop-3-6}
 	Let $C$ be a Shimura curve contained generically in $\mathcal{TS}_{g,n}$ with $g\geq n$.
 	If $n$ is not prime, then
 	$\rho_{n,n'}(C)$ is a Shimura curve
 	contained generically in $\mathcal{TS}_{g',n'}$,
 	where
 	$$n'=\max\left\{n_1~\big|~n_1<n\text{~and~}n_1~|~n\right\},$$
 	and $g'$ is determined by the formula \eqref{eqn-3-9(1)}.
 \end{proposition}
 The above proposition will be postponed until \autoref{sec-technical}.
 In the following, we prove \autoref{thm-curve} based on this principle.

 \begin{proof}[Proof of \autoref{thm-curve}]
 Let $C$ be a Shimura curve contained generically in $\mathcal{TS}_{g,n}$ with $g\geq 8$,
 and $\bar f:\,\ols \to \olb$ be the family of semi-stable
 $n$-superelliptic curves representing $C$ as in \autoref{sec-represent}.
 Assume that the general fiber of $\bar f$ is given by $y^n=F(x)$,
 where $F(x)$ is a separable polynomial in $x$ with $\deg(F)=\alpha_0$.
 By \autoref{prop-moonen}, we may assume that $g\geq n$, or equivalently $\alpha_0\geq 4$ by the Riemann-Hurwitz formula \eqref{eqn-3-9}.
 %
 Moreover, we claim that
 \begin{lemma}\label{prop-3-7}
 If $g\geq 8$ and $\alpha_0=4$, there does not exist any Shimura curve contained generically
 in $\mathcal{TS}_{g,n}$.
 \end{lemma}

 The proof of the above lemma will be postponed until the end of this subsection.
 Let's first complete the proof of \autoref{thm-curve}.
 We prove by induction on the number of prime factors in $n$.
 By \cite[Theorem\,E]{lu-zuo-14}, \autoref{thm-curve-prime} and \autoref{prop-3-7},
 we may assume that $n$ is not prime and $\alpha_0\geq 5$.
 Let
 $$n'=\max\left\{n_1~\big|~n_1<n\text{~and~}n_1~|~n\right\}.$$
 Since $n$ is not a prime, it follows that $n'\geq 2$.
 Consider the image $\rho_{n,n'}(C)\subseteq \mathcal{TS}_{g',n'}$, where $\rho_{n,n'}$ is defined in \eqref{eqn-3-123}.
 If
 \begin{equation}\label{eqn-3-95}
 g'\geq 8,
 \end{equation}
 then $\rho_{n,n'}(C)$ is not a Shimura curve by induction.
 Combining this with \autoref{prop-3-6},
 it follows that $C$ is   not a Shimura curve either.
 This gives a contradiction.
 By \eqref{eqn-3-9(1)}, it is easy to verify that the above condition\,\eqref{eqn-3-95} is satisfied
 unless $(n,\alpha_0)$ belongs to the following list (note that $\alpha_0\geq 5$ and $g\geq 8$ by our assumption).
 \begin{equation}\label{eqn-3-21}
 \left\{\begin{aligned}
 \text{Case (a):~}~&~\text{$n=4$ and $7\leq \alpha_0\leq 16$;}\\
 \text{Case (b):~}~&~\text{$n=6$ and $5\leq \alpha_0\leq 9$;}\\
 \text{Case (c):~}~&~\text{$n=8$ and $5\leq \alpha_0\leq 6$;}\\
 \text{Case (d):~}~&~\text{$n=9$ and $5\leq \alpha_0\leq 9$;}\\
 \text{Case (e):~}~&~\text{$n=10$, $15$ or $25$, and $\alpha_0=5$.}
 \end{aligned}\right.
 \end{equation}
 To complete the proof,
 it suffices to prove the non-existence of Shimura curves in the above cases.

 We again prove by contradiction.
 Assume that such a Shimura curve $C$ exists and let $\bar f$ be as above.
 By a possible base change, we may assume that
 the group $G\cong \mathbb Z/n\mathbb Z$ acts on $\ols$ as before,
 and hence there is an induced action of $G$
 on the associated logarithmic Higgs bundle as well as its subbundles with induced
 eigenspace decompositions as in \eqref{eqn-3-38} and \eqref{eqn-3-33}.
 Moreover, the rank of each eigenspace
 $E^{1,0}_{\olb,i}$ can be computed by \eqref{eqn-3-35}.
 We claim that
 \begin{lemma}\label{claim-3-6}
 Let $(n,\alpha_0)$ be as in \eqref{eqn-3-21}. Then the following statements hold.

 {\rm (i)}. If $3~|~n$, then $F^{1,0}_{\olb,2n/3}=0$.

 {\rm (ii)}.
 If $\rank E^{1,0}_{\olb,i}=1$, then $F^{1,0}_{\olb,i}=0$.

 {\rm (iii)}.
 Let
 $$i_m=\max\big\{i~\big|~F^{1,0}_{\olb,i}\neq 0\big\}.$$
 Then $i_m> n/2$, and hence after a suitable finite \'etale base
 change, there exists a unique morphism $\bar f':\,\ol S \to \ol B'$ such that
 \begin{equation}\label{eqn-3-106}
 \left\{\begin{aligned}
 \rank F^{1,0}_{\olb,i}&\,=\dim H^0\big(\ol S,\,\Omega^1_{\ol S}\big)_{i},
 \quad\forall~n-i_m\leq i \leq i_m;\\
 \bigoplus_{i= n-i_m}^{i_m} H^0\big(\ol S,\,\Omega^1_{\ol S}\big)_{i}
 &\,\subseteq \big(\bar f'\big)^*H^0\big(\ol B',\,\Omega^1_{\ol B'}\big).
 \end{aligned}\right.
 \end{equation}

 {\rm (iv)}.
 Let $i_m$ be as above.
 If $\gcd(n,i_m)=1$, then the curve $C$ is compact, and $G\cong \mathbb Z/n\mathbb Z$ induces a faithful action
 on $\olb'$ (here $\olb'$ is from (iii) above) such that $\ol B'/G\cong \bbp^1$ and that
 \begin{equation}\label{eqn-3-43}
 H^0\big(\ol S,\,\Omega^1_{\ol S}\big)_{i}
 =\big(\bar f'\big)^*H^0\big(\ol B',\,\Omega^1_{\ol B'}\big)_{i},\quad\text{~for any $1\leq i \leq n-1$ with $\gcd(i,n)=1$}.
 \end{equation}
 In particular,
 \begin{equation}\label{eqn-3-184}
 \begin{aligned}
 	g(\olb') \geq~& \frac{\varphi(n)}{2}\cdot
 	\left(\dim H^0\big(\ol S,\,\Omega^1_{\ol S}\big)_{i_m}+ \dim H^0\big(\ol S,\,\Omega^1_{\ol S}\big)_{n-i_m}\right)\\
 	=~&\frac{\varphi(n)}{2}\cdot
 	\left(\rank F^{1,0}_{\olb,i_m}+\rank F^{1,0}_{\olb,n-i_m}\right),
 \end{aligned}
 \end{equation}
 where $\varphi(n)$is the Euler phi function,
i.e. the number of non-negative integers less than $n$ which are relatively prime to $n$.
 \end{lemma}
 \noindent
 We will prove the above lemma at the end of this subsection.
 Based on the above lemma,
 let's first derive contradictions case-by-case with $(n,\alpha_0)$ in \eqref{eqn-3-21},
 and hence prove \autoref{thm-curve}.

 {\vspace{1mm} \it Case} (a).
 In this case, $i_m=3$ by \autoref{claim-3-6}\,(iii).
 Hence by \eqref{eqn-3-106}, one has
 \begin{equation*}
 g(\olb')\geq \rank F^{1,0}_{\olb}.
 \end{equation*}
 This gives a contradiction to \autoref{lem-3-71} as $g\geq 8$.


 {\vspace{1mm} \it Case} (b). In this case, one has $i_m=5$ by \autoref{claim-3-6}\,(i) and (iii).
 Hence similar to the above case, there is also a contradiction to \autoref{lem-3-71}.

 {\vspace{1mm} \it Case} (c). From \autoref{claim-3-6} it follows that
 $\alpha_0=6$, $i_m=5$ and $C$ is compact,
 and that \eqref{eqn-3-43} and \eqref{eqn-3-184} hold.
 Note also that $F^{1,0}_{\olb,5}\subsetneq E^{1,0}_{\olb,5}$ and hence $\rank F^{1,0}_{\olb,5}=1$;
 otherwise, one has $\rank F^{1,0}_{\olb,5}= \rank E^{1,0}_{\olb,5}=2$,
 and hence both $E^{1,0}_{\olb,3}$ and $E^{1,0}_{\olb,5}$ are flat,
 from which together with \eqref{eqn-3-184} it follows that
 $$g(\olb') \geq \frac{\varphi(8)}{2}\cdot \left(\rank F^{1,0}_{\olb,5}+\rank F^{1,0}_{\olb,3}\right)
 =\frac{4}{2}\cdot \left(\rank E^{1,0}_{\olb,5}+\rank E^{1,0}_{\olb,3}\right)=10.$$
 This contradicts \eqref{eqn-3-131} as $g=17$ by \eqref{eqn-3-9} in this case.
 Thus,
 $$\left\{\begin{aligned}
 &\rank A^{1,0}_{\ol B,7}=\rank E^{1,0}_{\ol B,7}=0;\\
 &\rank A^{1,0}_{\ol B,6}=\rank E^{1,0}_{\ol B,6}=1, &&\text{by \autoref{claim-3-6}\,(ii);}\\
 &\rank A^{1,0}_{\ol B,5}=\rank E^{1,0}_{\ol B,5}-1=1, &&\text{by the above arguments};\\
 &\rank A^{1,0}_{\ol B,4}=\rank E^{1,0}_{\ol B,4}=2, &&\text{by \autoref{lem-3-19}\,(i) and \autoref{cor-3-4}.}
 \end{aligned}\right.$$
 Hence
 $$\rank A^{1,0}_{\ol B}=2\sum_{i=5}^7 \rank A^{1,0}_{\ol B,7}+ \rank A^{1,0}_{\ol B,4}=6.$$
 This is a contradiction to \eqref{eqn-3-101-2}.

 {\vspace{1mm} \it Case} (d).
 In this case, we can derive a contradiction similarly as the above case.
 Indeed, Similar as above, by \autoref{claim-3-6}
 one obtains that $i_m=5$ and $C$ is compact.
 If $\rank F^{1,0}_{\olb,5}=\rank E^{1,0}_{\olb,5}$, then one obtains a contradiction to \eqref{eqn-3-131};
 if $\rank F^{1,0}_{\olb,5}\leq \rank E^{1,0}_{\olb,5}-1$, i.e., $\rank A^{1,0}_{\olb,5}\geq 1$,
 then
 $$\rank A^{1,0}_{\ol B}=2\sum_{i=5}^8 \rank A^{1,0}_{\ol B,7}\geq 2+2\sum_{i=6}^8 \rank A^{1,0}_{\ol B,7}
 =2+2\sum_{i=6}^8 \rank E^{1,0}_{\ol B,7},$$
 which contradicts \eqref{eqn-3-101-2}.

 {\vspace{1mm} \it Case} (e).
 By \autoref{claim-3-6} (ii) and (iii), one checks easily that
 $n\neq 10$, and that $i_m=8$ (resp. $i_m=13$ or $14$) when $n=15$ (resp. $n=25$).
 In the later two cases, by \autoref{claim-3-6}, after a suitable further \'etale base change, there is a unique fibration
 $\bar f':\,\ols \to \olb'$ and $G$ acts faithfully on $\olb'$ with $\olb'/G\cong \bbp^1$.
 Let $\ol F$ be a general fiber of $\bar f$, and $\ol\Gamma=\ol F/G\cong \bbp^1$ the quotient.
 Then similar to the proof of \autoref{lem-3-28}, one has the following commutative diagram.
 $$\xymatrix{
 	\ol F \ar[rrr]^-{\bar f'|_{\ol F}} \ar[d]_-{\Pi|_{\ol F}}&&&\ol B' \ar[d]^-{\pi'}  \\
 	\ol\Gamma\ar[rrr]^-{\bar\varphi'|_{\ol \Gamma}}\cong \bbp^1 &&& \ol B'/G\cong \bbp^1
 }$$
 %
 %
 Let $\beta$ be the number of branch points of the cover $\pi':\,\olb'\to \bbp^1$.
 Then similar to the proof of \autoref{lem-3-28}, one proves that $\beta\geq 4$ and $\alpha>2(\beta-1)$,
 where $\alpha=\alpha_0+1=6$ is the number of the branch points of $\Pi|_{\ol F}$.
 This gives a contradiction.
 \end{proof}

 \vspace{1.5mm}
 In the rest of this subsection
 we  prove \autoref{prop-3-7} and \autoref{claim-3-6},
 for which one needs the following proposition generalizing \autoref{prop-3-13}. Its proof will be given in later \autoref{sec-pf-3-14}.
 \begin{proposition}\label{prop-3-14}
 	Let $\bar f:\,\ol S \to \ol B$ be the family of semi-stable curves representing
 	a Shimura curve $C$ contained generically in $\mathcal{TS}_{g,n}$ with $g\geq 8$.
 	Assume that the group $G\cong \mathbb Z/n\mathbb Z$
 	acts on $\ol S$ whose restriction on the general fiber is
 	the $n$-superelliptic automorphism group (up to base change).
 	\begin{enumerate}
 	\item If $\alpha_0=4$, then $F^{1,0}_{\olb,i}=0$ for any $i\geq n/2$.
 	\item If $(n,\alpha_0)$ belongs to the list in \eqref{eqn-3-21}, then
 	      $F^{1,0}_{\olb,i}=0$ for any $i>n/2$ with $\rank E^{1,0}_{\olb,i}=1$.
 	\end{enumerate}
 \end{proposition}

 \vspace{1mm}
 \begin{proof}[Proof of \autoref{prop-3-7}]
 Since $g\geq 8$, it follows that $n\geq 7$.
 As before, after a suitable finite base change, we may assume that
 the group $G\cong \mathbb Z/n\mathbb Z$ acts on $\ols$,
 and hence there is an induced action of $G$
 on the associated logarithmic Higgs bundle as well as its subbundles with induced
 eigenspace decompositions as in \eqref{eqn-3-38} and \eqref{eqn-3-33}.
 Since $\alpha_0=4$,
 according to \autoref{prop-3-14}, it follows that $F_{\olb,i}^{1,0}= 0$ for any $i\geq n/2$.
 Moreover, we claim that
 \begin{claim}\label{claim-3-4}
 	If $F_{\olb,i}^{1,0}= 0$ for any $i\geq n/2$, then
 	\begin{enumerate}
 		\item[(1)\,] $n=9$, and hence $g=12$;
 		
 		\item[(2)\,] after a suitable finite \'etale base change,
 		there exists an irregular fibration $\bar f':\,\ols \to \olb'$
 		such that the group $G\cong \mathbb Z/9\mathbb Z$ induces a faithful action on $\olb'$ with $\olb'/G\cong \bbp^1$;
 		
 		\item[(3)\,] the general fiber $\ol F$ admits a cyclic cover of degree $18$ to $\bbp^1$ branched over exactly $4$ points.
 	\end{enumerate}
 \end{claim}
 Assume  the  claim above for the moment,
it then  follows that $\bar f$ is a universal family of cyclic covers of $\bbp^1$ with group $\wt G\cong \mathbb{Z}/18\mathbb{Z}$,
 since the general fiber $\ol F$ admits a cyclic cover of degree $18$ to $\bbp^1$ branched over exactly $4$ points.
 Hence according to \cite[Theorem\,3.6]{moonen-10},
 $C$ is not a Shimura curve since $g=12$.
 Therefore, it suffices to prove the above claim.

 {\vspace{1mm}\noindent\it Proof of \autoref{claim-3-4}.}
 (1). First of all, we show that $C$ is compact.
 In fact, since $E^{1,0}_{\olb,n-1}=0$ by \eqref{eqn-3-35}, one obtains $E^{1,0}_{\olb,1}=F^{1,0}_{\olb,1}$ by \autoref{lem-3-42}.
 If $C$ is  non-compact,
 then from \cite[Corollary\,4.4]{viehweg-zuo-04} it follows that
 the Higgs subbundle
 $$\left(E^{1,0}_{\olb}\oplus E^{0,1}_{\olb},~\theta_{\olb}\right)_1
 =\left(F^{1,0}_{\olb}\oplus F^{0,1}_{\olb},~0\right)_1$$
 is a trivial Higgs subbundle after a possible \'etale base change.
 In other words, the corresponding local subsystem $\mathbb V_{B,\,1}=\mathbb V_{B,\,1}^{tr}$
 is trivial.
Combine this  with \autoref{lem-3-12},
 it follows that $\mathbb V_{B,\,i}$ is also trivial for any $1\leq i\leq n-1$ and $\gcd(i,n)=1$.
 On the other hand, by \eqref{eqn-3-35}, it is easy to verify that there exists
 $i_0>n/2$ such that $\gcd(i_0,n)=1$ and $\rank E^{1,0}_{\olb,i_0}=1.$
 Hence $F_{\olb,i_0}^{1,0}=E^{1,0}_{\olb,i_0} \neq 0$,
 which is a contradiction to the assumption.

 Next we prove that $n$ is odd;
 otherwise, since $C$ is compact and $\rank E^{1,0}_{\olb,n/2}=1$,
 it follows from \autoref{cor-3-4} that $A^{1,0}_{\olb,n/2}=0$,
 i.e., $F^{1,0}_{\olb,n/2}=E^{1,0}_{\olb,n/2}\neq 0$, which contradicts the assumption.

 Finally, since $F^{1,0}_{\olb,i}=0$ for all $i>n/2$,
 it follows that $A^{1,0}_{\ol B,i}=E^{1,0}_{\olb,i}$ for all $i>n/2$.
 As $n$ is odd, by \eqref{eqn-3-102} and \eqref{eqn-3-35}, one obtains
 \begin{equation*}
 \rank A^{1,0}_{\ol B}=
 2\sum_{i=(n+1)/2}^{n-1}\rank E^{1,0}_{\olb,i}=
 \left\{\begin{aligned}
 &(n-1)/2, &\quad&\text{if~}n=4k+1 \text{~for some~}k>0;\\
 &(n+1)/2, &\quad&\text{if~}n=4k+3 \text{~for some~}k>0.
 \end{aligned}\right.
 \end{equation*}
 Combining this with \eqref{eqn-3-9} and \eqref{eqn-3-101-2},
 we obtain that $n=9$, and hence $g=12$.

 (2).
 Since $F_{\olb,i}^{1,0}= 0$ for any $i>n/2=9/2$,
 by \eqref{eqn-3-35} and \eqref{eqn-3-102}, one obtains that
 \begin{equation*}
 \rank F_{\olb,1}^{1,0}=\rank F_{\olb,2}^{1,0}=3,\quad
 \rank F_{\olb,3}^{1,0}=\rank F_{\olb,4}^{1,0}=1.
 \end{equation*}
 Hence by \cite[\S\,4.2]{deligne-71},
 it follows that after a suitable \'etale base change,
 both $F_{\olb,3}^{1,0}$ and $F_{\olb,4}^{1,0}$ become trivial.
 In other words, one has
 $$\dim H^0\big(\ols,\Omega_{\ols}^{1}\big)_3=\dim H^0\big(\ols,\Omega_{\ols}^{1}\big)_4=1.$$

 Next, we claim that $H^0\big(\ols,\Omega_{\ols}^{2}\big)_7=0$ by Hurwitz-Chevalley-Weil's formula (cf. \cite[Proposition\,5.9]{moonen-oort-13});
 indeed, one has $H^0\big(\ol F,\omega_{\ol F}\big)_7=0$ by Hurwitz-Chevalley-Weil's formula,
 from which it follows that
 $j^*(\gamma)=0$ for any $\gamma\in H^0\big(\ols,\Omega_{\ols}^{2}\big)_7$, where
 $$j^*:~ H^0\big(\ols,\Omega_{\ols}^{2}\big)_7= H^0\big(\ols,\omega_{\ols}\big)_7 \lra H^0\big(\ol F,\omega_{\ol F}\big)_7=0$$
 is the canonical pulling-back and $j:\,\ol F \hookrightarrow \ols$ is the embedding of a general fiber of $\bar f$ into $\ols$.
 Since $\ol F$ is general,
 it follows that $H^0\big(\ols,\Omega_{\ols}^{2}\big)_7=0$ as required.

 Now let $\omega\in H^0\big(\ols,\Omega_{\ols}^{1}\big)_3$ and $\eta \in H^0\big(\ols,\Omega_{\ols}^{1}\big)_4$
 be two non-zero one-forms.
 Since $\omega\wedge\eta \in H^0\big(\ols,\Omega_{\ols}^{2}\big)_7$ by construction,
 it follows that $\omega\wedge\eta=0$.
 Hence by Castelnuovo-de Franchis lemma (cf. \cite[Theorem\,IV-5.1]{bhpv-04}), there exists an irregular fibration $\bar f':\,\ol S \to \ol B'$ such that
 \begin{equation}\label{eqn-3-185}
 	H^0\big(\ol S,\,\Omega^1_{\ol S}\big)_{3} \oplus H^0\big(\ol S,\,\Omega^1_{\ol S}\big)_{4}
 	\subseteq  \big(\bar f'\big)^*H^0\big(\ol B',\,\Omega^1_{\ol B'}\big).
 \end{equation}
 It is clear that such a fibration is unique, and hence the group $G$
 induces an action on $\olb'$.
 Moreover, the induced action is faithful; otherwise
 $$\big(\bar f'\big)^*H^0\big(\ol B',\,\Omega^1_{\ol B'}\big)\subseteq H^0\big(\ol S,\,\Omega^1_{\ol S}\big)^{G_0}
 =\bigoplus_{m_0\,|\,i} H^0\big(\ol S,\,\Omega^1_{\ol S}\big)_i,$$
 which contradicts \eqref{eqn-3-185},
 where $G_0 \leq G$ is the kernel of the action of $G$ on $\olb'$ and $m_0=|G_0|$.
 Finally, since $\ols/G$ is ruled, it follows that $\olb'/G\cong \bbp^1$.
 %
 %
 %
 %
 %
 %

 (3).
 Let $\ol F$ be a general fiber of $\bar f$,
 and consider the restricted map $\bar f'|_{\ol F}:\,\ol F \to \olb'$.
 Since $G\cong \mathbb Z/9\mathbb Z$ acts faithfully on both $\ol F$ and $\olb'$,
 whose quetients are both isomorphic to $\bbp^1$.
 And similar to the proof of \autoref{lem-3-28}, one has the following commutative diagram.
 $$\xymatrix{
 	\ol F \ar[rrr]^-{\bar f'|_{\ol F}} \ar[d]_-{\Pi|_{\ol F}}&&&\ol B' \ar[d]^-{\pi'}  \\
 	\ol\Gamma\ar[rrr]^-{\bar\varphi'|_{\ol \Gamma}}\cong \bbp^1 &&& \ol B'/G\cong \bbp^1
 }$$
 %
 %
 We claim that $\deg(\bar f'|_{\ol F})=2$.
 Indeed, note that the cover $\Pi|_{\ol F}$ has exactly $\alpha=5$ branch points,
 and one checks easily that $\pi'$ admits $\beta\geq 3$ branch points since $g(\olb')>0$.
 Because $\bar f$ is not isotrivial,
 similar to the proof of \eqref{eqn-3-127-1}, one shows that
 $$3=\alpha-2>\deg(\bar f'|_{\ol F})\cdot (\beta-2)\geq \deg(\bar f'|_{\ol F})>1.$$
 Thus $\deg(\bar f'|_{\ol F})=2$ as required.

 Since $\deg(\bar f'|_{\ol F})=2$,
 it induces an involution $\tau_0$ on $\ol F$, such that $\olb'=\ol F/\langle\tau_0\rangle$.
 Let $\wt G \subseteq \Aut(\ol F)$ be the subgroup generated by $G$ and $\tau_0$.
 As $\bar f'|_{\ol F}$ is equivariant with respect to $G$,
 it follows that $\tau_0$ commutes with $G$.
 Hence $\wt G$ is cyclic group of order $|\wt G|=18$.
 Moreover, by considering the composition map
 $\pi'\circ \big(\bar f'|_{\ol F}\big):\,\ol F \to \bbp^1$, one checks easily that
 $\pi'\circ \big(\bar f'|_{\ol F}\big)$ is a cyclic cover branched over exactly $4$ points.
 This completes the proof.
 \end{proof}

 \begin{proof}[Proof of \autoref{claim-3-6}]
 (i).
 We prove by contradiction.
 Assume that $\rank F^{1,0}_{\olb,2n/3}>0$.
 As $3~|~n$ and $(n,\alpha_0)$ belongs to \eqref{eqn-3-21}, it follows that $5\leq \alpha_0\leq 9$.
 According to \autoref{prop-3-6},
 the image $\rho_{n,3}(C)\subseteq \mathcal{TS}_{g_1,3}$ is still a Shimura curve,
 where $\rho_{n,3}$ is defined in \eqref{eqn-3-123}. Moreover, by \eqref{eqn-3-9(1)} one gets
 $$g_1=\left\{\begin{aligned}
 &\alpha_0-1, &\quad&\text{if~}3{\not|}~\alpha_0;\\
 &\alpha_0-2, && \text{if~}3\,|\,\alpha_0.
 \end{aligned}\right.$$
 Let $\bar f_1:\,\ols_1 \to \olb$
 be the family of $3$-superelliptic curves associated to $\rho_{n,3}(C)\subseteq \mathcal{TS}_{g_1,3}$,
 and denote by  $\left(\wt E^{1,0}_{\olb}\oplus \wt E^{0,1}_{\olb},~\wt \theta_{\olb}\right)$
 the corresponding Higgs bundle associated to $\bar f_1$.
 Then it follows from \autoref{lem-3-18}, whose proof is given later in Section 4, that we have the isomorphism
 $$\wt E_{\olb,2}^{1,0} \cong  E_{\olb,2n/3}^{1,0}$$
 Hence $\rank \wt F^{1,0}_{\olb,2} >0$ by our hypothesis, where $\wt F^{1,0}_{\olb}\subseteq \wt E^{1,0}_{\olb}$ is the flat part.
 Therefore, by \autoref{prop-3-13} for $\bar f_1$,
 after a suitable \'etale base change,
 there exists a fibration $\bar f':\,\ols_1 \to \olb'$ different from $\bar f_1$ such that
 $$\begin{aligned}
 g(\olb')\geq \rank \wt F^{1,0}_{\olb,1}+\rank \wt F^{1,0}_{\olb,2}
 =~&\left(\rank \wt E^{1,0}_{\olb,1}-\rank \wt E^{1,0}_{\olb,2}\right)+2\,rank \wt F^{1,0}_{\olb,2}\\
 \geq~&\left(\rank \wt E^{1,0}_{\olb,1}-\rank \wt E^{1,0}_{\olb,2}\right)+2.
 \end{aligned}$$
 This is a contradiction to \eqref{eqn-3-131} since $5\leq \alpha_0\leq 9$.

 \vspace{1mm} (ii).
 This follows from \autoref{prop-3-14}.

 \vspace{1mm} (iii).
 By \autoref{prop-3-13}, it suffices to prove that
 \begin{equation}\label{eqn-3-207}
 	i_m>n/2.
 \end{equation}
 We divide the proof into two cases.

 Consider first the case when $C$ is non-compact.
 In this case, the proof is similar to that of \autoref{lem-3-72}.
 By \cite[Corollary\,4.4]{viehweg-zuo-04},
 we may assume that the unitary local subsystem $\mathbb{V}_{B}^u\subseteq \mathbb{V}_{B} \otimes \mathbb{C}$
 is trivial after a suitable finite base change, i.e., $\mathbb{V}_{B}^u=\mathbb{V}_{B}^{tr}$.
 We proceed along the possible values of $n$:
 \begin{list}{}
 	{\setlength{\labelwidth}{6mm}
 		\setlength{\leftmargin}{8mm}
 		\setlength{\itemsep}{2.5mm}}
 	\item[(1).] Using \autoref{lem-3-12} and \autoref{lem-3-42}, one proves easily that
 	$\rank F^{1,0}_{\olb,n/2+1}>0$ if $n=8$ or $10$;
 	and that $\rank F^{1,0}_{\olb,(n+1)/2}>0$ if $n=9$, $15$ or $25$.
 	
 	\item[(2).] If $n=6$, then by \autoref{lem-3-15} and \autoref{prop-3-6}, $\rho_{6,3}(C)$ is again a non-compact Shimura curve,
 	where $\rho_{6,3}$ is defined in \eqref{eqn-3-123}.
 	Hence by \autoref{lem-3-18}, it suffices to prove $\rank \wt F^{1,0}_{\olb,2}>0$,
 	where $\wt F^{1,0}_{\olb}\oplus \wt F^{0,1}_{\olb}$ is denoted to be the flat subbundle
 	associated to the new family representing $\rho_{6,3}(C)$.
 	This has already been discussed in \autoref{lem-3-72}:
 	suppose that $\rank \wt F^{1,0}_{\olb,2}=0$.
 	Then using \eqref{eqn-3-176} and \eqref{eqn-3-35}, one derives a contradiction to \eqref{eqn-3-101-1}.
 	
 	\item[(3).] If $n=4$, we will derive a contradiction when $\rank  F^{1,0}_{\olb,3}=0$.~
 	By \autoref{lem-3-15} and \autoref{prop-3-6}, $\rho_{4,2}(C)$ is again a non-compact Shimura curve,
 	where $\rho_{4,2}$ is defined in \eqref{eqn-3-123}.
 	Moreover, $\rho_{4,2}(C)$ is contained generically in the hyperelliptic Torelli locus.
 	Hence according to the proof of \cite[Theorem\,1.4\,(ii)]{lu-zuo-14}, one has
 	$\rank \wt F^{1,0}_{\olb}\leq 1$, where $\wt F^{1,0}_{\olb}\oplus \wt F^{0,1}_{\olb}$ is denoted to be the flat subbundle
 	associated to the new family representing $\rho_{4,2}(C)$.
 	Thus $\rank F^{1,0}_{\olb,2}=\rank \wt F^{1,0}_{\olb}\leq 1$ by \autoref{lem-3-18}.
 	Since we assume that $\rank  F^{1,0}_{\olb,3}=0$, by \eqref{eqn-3-102} and \eqref{eqn-3-35} one has
 	$$\rank F^{1,0}_{\olb}=\rank F^{1,0}_{\olb,1}+\rank F^{1,0}_{\olb,2}\leq \big(\rank E^{1,0}_{\olb,1}-\rank E^{1,0}_{\olb,3}\big)+1.$$
 	Equivalently, one has
 	$$\rank A^{1,0}_{\olb}\geq 2\,\rank E^{1,0}_{\olb,3}+\rank E^{1,0}_{\olb,2}-1.$$
 	It is clear that $q_{\bar f}=\rank F^{1,0}_{\olb}\neq 0$.
 	According to \eqref{eqn-3-35}, the above bound on $\rank A^{1,0}_{\olb}$ gives a contradiction to \eqref{eqn-3-101-1}.
 \end{list}

 Consider next the case when $C$ is compact.~
 In this case, we prove \eqref{eqn-3-207} by contradiction.
 Assume that $F^{1,0}_{\olb,i}=0$ for all $i>n/2$.
 Then $A^{1,0}_{\ol B,i}=E^{1,0}_{\ol B,i}$ for any $i>n/2$.
 Combing this with \eqref{eqn-3-102}, one obtains
 \begin{equation*}
 \rank A^{1,0}_{\ol B}=\left\{
 \begin{aligned}
 &2\sum_{i=(n+1)/2}^{n-1}\rank E^{1,0}_{\ol B,i}, &&\text{if~}2{\not|}~n;\\
 &\rank A^{1,0}_{\ol B,n/2}+2\sum_{i=(n+2)/2}^{n-1}\rank E^{1,0}_{\ol B,i}, &&\text{if~}2\,|\,n.
 \end{aligned}\right.
 \end{equation*}
 We claim that $\rho_{n,2}(C)$ is still a Shimura curve if $2\,|\,n$.
 In fact, if $\rho_{n,2}(C)$ were not a Shimura curve, it follows easily
 from \autoref{prop-3-6} and \autoref{lem-3-19} that $n=6$ and $\alpha_0=8$.
 We remark here that we apply \autoref{prop-3-6} twice
 and use the fact that $\rho_{4,2}\circ \rho_{8,4}=\rho_{8,2}$ when excluding the case when $n=8$.
 For the case when $n=6$ and $\alpha_0=8$,
 by \autoref{lem-3-51} and \autoref{lem-3-52} together with \autoref{lem-3-19}\,(v), one obtains a contradiction.
 Thus we may assume that $\rho_{n,2}(C)$ is still a Shimura curve when $2\,|\,n$.
 Combining the above equation with \autoref{cor-3-4} and \eqref{eqn-3-9}, we obtain
 a contradiction to \eqref{eqn-3-101-2}.
 In fact, let $\nu=\frac{4(g-1)}{\rank A_{\olb}^{1,0}}$. Then
 \noindent{\begin{center}\label{stand-notation}
 \renewcommand{\arraystretch}{1.2}
 \begin{tabular}{|c|c|c|c|c|c|c|c|c|}\hline
 &$n=25$&$n=15$&$n=10$&$n=9$&$n=8$&$n=6$&$n=4$ \\\hline
 $\alpha_0=5$ &$\nu=10$         & $\nu=10$&$\nu=10$& $\nu=\frac{15}{2}$ & $\nu\leq \frac{26}{3}$& $\nu\leq 9$&---\\\hline
 $\alpha_0=6$ & ---             &  ---    &  ---   & $\nu=9~$\, &  $\nu\leq 8$& $\nu\leq 9$ &---\\\hline
 $\alpha_0=7$ & ---             &  ---    &  ---   & $\nu=\frac{23}{3}$&  ---   & $\nu\leq 7$ & $\nu\leq 8$\\\hline
 $\alpha_0=8$ & ---             &  ---   &  ---   & $\nu=9~$\, &  ---   & $\nu\leq 8$& $\nu\leq 8$ \\\hline
 $\alpha_0=9$ & ---             &  ---    &  ---   & $\nu=9~$\, &  ---   & $\nu\leq 9$& $\nu\leq \frac{22}{3}$\\\hline
 $\alpha_0=10$& ---             &  ---    &  ---   &  ---   &  ---   &  ---   & $\nu\leq 8$\\\hline
 $\alpha_0=11$& ---             &  ---    &  ---   &  ---   &  ---   &  ---   & $\nu\leq 7$\\\hline
 $\alpha_0=12$& ---             &  ---    &  ---   &  ---   &  ---   &  ---   &$\nu\leq 7$\\\hline
 $\alpha_0=13$& ---             &  ---    &  ---   &  ---   &  ---   &  ---   &$\nu\leq \frac{34}{5}$\\\hline
 $\alpha_0=14$& ---             &  ---    &  ---   &  ---   &  ---   &  ---   &$\nu\leq \frac{36}{5}$\\\hline
 $\alpha_0=15$& ---             &  ---   &  ---   &  ---   &  ---   &  ---   &$\nu\leq 8$\\\hline
 $\alpha_0=16$& ---             &  ---    &  ---   &  ---   &  ---   &  ---   &$\nu\leq 8$\\\hline
 \end{tabular}
 \end{center}}
 \noindent
 This contradicts \eqref{eqn-3-101-2}.

 \vspace{1mm} (iv).
 From the uniqueness of the fibration as in (iii), it follows that $G$ admits an induced action on $\olb'$.
 Since $\gcd(n,i_m)=1$, it follows from \eqref{eqn-3-106} that this induced action is faithful.
 Moreover, $\olb'/G\cong \bbp^1$ as the quotient $\ols/G$ is ruled.

 The equality \eqref{eqn-3-43} follows from a similar argument as the proof of \autoref{claim-3-31}.
 Indeed, by \eqref{eqn-3-106} one has
 $$H^0\big(\ol S,\,\Omega^1_{\ol S}\big)_{i_m} \oplus H^1\big(\ol S,\,\mathcal{O}_{\ols}\big)_{i_m}=
 \left(H^1(\ols,\,\mathbb Q) \otimes \mathbb{C}\right)_{i_m}
 \subseteq  \big(\bar f'\big)^*\left(H^1(\olb',\,\mathbb Q)\otimes \mathbb{C}\right),$$
 and the eigen-subspaces $\left(H^1(\ols,\,\mathbb Q) \otimes \mathbb{C}\right)_{i}$'s for $1\leq i\leq p-1$ with $\gcd(n,i)=1$
 are permuted by this action of the arithmetic Galois subgroup ${\rm Gal}\big(\mathbb{Q}(\xi_n)/\mathbb{Q}\big)$,
 where $\xi_n$ is a primitive $n$-th root of the unit.
 Hence
 $$\bigoplus_{1\leq i\leq n \atop \gcd(i,n)=1}\left(H^1(\ols,\,\mathbb Q) \otimes \mathbb{C}\right)_{i}
 \subseteq \big(\bar f'\big)^*\left(H^1(\olb',\,\mathbb Q)\otimes \mathbb{C}\right).$$
 By taking the $(1,0)$-part, we proves \eqref{eqn-3-43}.

 The inequality \eqref{eqn-3-184} follows immediately from \eqref{eqn-3-43} together with \eqref{eqn-3-106},
 by noting also that for any $1\leq \{i,j\}\leq n-1$ with $\gcd(n,i)=\gcd(n,j)=1$ one has
 $$\dim H^0\big(\ol B',\,\Omega^1_{\ol B'}\big)_{i}+\dim H^0\big(\ol B',\,\Omega^1_{\ol B'}\big)_{n-i}
 =\dim H^0\big(\ol B',\,\Omega^1_{\ol B'}\big)_{j}+\dim H^0\big(\ol B',\,\Omega^1_{\ol B'}\big)_{n-j}.$$

 Finally, we prove by contradiction that $C$ is compact.
 Assume that $C$ is non-compact.
 If $n=4$ or $6$, then $i_m=n-1$ since $i_m>n/2$ and $\gcd(n,i_m)=1$.
 Hence by \eqref{eqn-3-106} one has
 $g(\olb')\geq \rank F^{1,0}_{\olb}$,
 This gives a contradiction to \autoref{lem-3-71} as $g\geq 8$.
 For the rest cases in \eqref{eqn-3-21},
 one checks by \eqref{eqn-3-35} easily that $\rank E^{0,1}_{\olb,1}=0$,
 combining which together with \autoref{lem-3-42} and \cite[Corollary\,4.4]{viehweg-zuo-04},
 one obtains that the Higgs subbundle
 $$\left(E^{1,0}_{\olb} \oplus E^{0,1}_{\olb},~\theta\right)_1
 =\left(F^{1,0}_{\olb}\oplus F^{0,1}_{\olb},~0\right)_1$$
 is a trivial Higgs subbundle after a possible \'etale base change.
 In other words, the corresponding local subsystem $\mathbb V_{\olb,\,1}$ is trivial.
 By a similar argument as in \autoref{lem-3-72},
 one shows that $\mathbb V_{\olb,\,i}$ is also trivial for any $1\leq i\leq n-1$ and $\gcd(i,n)=1$.
 Hence
 $$g(\olb')\geq \sum_{1\leq i\leq n \atop \gcd(i,n)=1} \rank E^{1,0}_{\olb,i}=\frac{\varphi(n)}{2}\cdot \rank E^{1,0}_{\olb,1}.$$
 Since $C$ is non-compact, $\Delta_{nc}\neq \emptyset$.
 Similar to the proof of \autoref{lem-3-71}, one derives also a contradiction to \eqref{eqn-3-131}
 when restricting $\bar f'$ to a fiber $\ol F$ over $\Delta_{nc}$.
 This completes the proof.
 \end{proof}

 \section{Family of superelliptic curves}\label{sec-superelliptic-family}
 In this section we prove some technical properties of semi-stable families of superelliptic curves,
 which are used in the last section,
 mainly including:
 \begin{enumerate}
 	\item an upper bound on the ample part in the Hodge bundle, i.e., the proof of \autoref{prop-3-11};
 	\item the existence of a fibration structure, i.e., the proof of \autoref{prop-3-13};
 	\item transitivity of Shimura curves, i.e., the proof of \autoref{prop-3-6}.
 \end{enumerate}

 In \autoref{sec-4-pre} we recall some basic facts and notations on families of curves.
 In \autoref{subsec-3-2} we investigate the invariants of a superelliptic family.
 In \autoref{subsec-3-4} we study the behavior of the flat part contained in the associated logarithmic Higgs bundle.
 In \autoref{subsec-3-x} we study the irregular superelliptic family.
 Finally in Sections\,\ref{sec-pf-3-11}-\ref{sec-pf-3-14} we prove the technical results used in the last section.

%

 \subsection{Preliminaries}\label{sec-4-pre}
 In the subsection, we collect some generalitie on families of curves.
 We refer to \cite{bhpv-04} for more details.

 Recall that a semi-stable  (resp. stable) curve is  a complete connected reduced nodal  curve such that
 each rational component intersects with the other components at $\geq 2$  (resp. 3) points.
 A semi-stable (resp. stable)  family of curves is a
 flat projective morphism  $\bar f:\,\ol S\to \ol B$
 from a projective surface $\ol S$ to a smooth projective curve $\ol B$
 with connected fibres such that all the singular fibres of $\bar f$ are
 semi-stable (resp. stable) curves. Moreover, $\bar f$ is said to be
 hyperelliptic if a general fibre of $\bar f$ is a hyperelliptic curve;
 and to be isotrivial if  all its smooth fibres are isomorphic to each other.
 From now on, we  assume
 that $\bar f:\,\ol S \to \ol B$ is a semi-stable family  of curves of genus
 $g\geq 2$ with singular fibres $\Upsilon\to\Delta$ and     $\ol S$ is
 smooth.

 Denote by $\omega_{\ol S/\ol B}=\omega_{\ol S}\otimes \bar f^*\omega_{\ol B}^{\vee}$ the
 relative canonical sheaf of $\bar f$. Let $\chi(\mathcal O_{\ol S})$
 be the Euler characteristic of the structure sheaf,
 and $\chit(\cdot)$ be the topological Euler characteristic.
 Consider the following relative invariants:
 \begin{equation}\label{defofrelativeinv}
 \left\{
 \begin{aligned}
 &\omega_{\ol S/\ol B}^2=\omega_{\ol S}^2-8(g-1)\big(g(\olb)-1\big),\\
 &\delta(\bar f)=\chit(\ol S)-4(g-1)\big(g(\olb)-1\big)=\sum_{\ol F\in \Upsilon}\delta(\ol F),\\
 &\deg \bar f_*\omega_{\ol S/\ol B}=\chi(\mathcal O_{\ol S})-(g-1)\big(g(\olb)-1\big),
 \end{aligned}\right.
 \end{equation}
 where $\delta(\ol F)$ is the number of nodes contained in the fiber $\ol F$.
 All the invariants in \eqref{defofrelativeinv} are nonnegative and
 satisfy the Noether's formula:
 \begin{equation}\label{formulanoether}
 12\deg \bar f_*\omega_{\ol S/\ol B}=\omega_{\ol S/\ol B}^2+\delta(\bar f).
 \end{equation}
 Since $\bar f$ is semi-stable, we also have the identity
 \begin{equation}\label{eqnomega=Omega}
 \omega_{\ol S/\ol B}\cong \Omega^1_{\ol S/\ol B}(\log\Upsilon).
 \end{equation}
 Recall also that
 $$\delta(\ol F)=\sum_{i=0}^{[g/2]} \delta_i(\ol F),$$
 where $\delta_i(\ol F)$ is the number of nodes of type $i$ contained in $\ol F$.
 Here we say a singular
 point $q$ of $\ol F$ to be  of type $i\in \big[1,  [g/2]\big]$ (resp. 0) if the
 partial normalization of $\ol F$ at $q$ consists of two connected
 components of arithmetic genera $i$ and $g-i$ (resp. is connected).
 By definition, a singular fiber $\ol F$ has a compact Jacobian if and only if $\delta_0(\ol F)=0$.
 Define $\delta_i(\bar f)=\sum\limits_{\ol F\in \Upsilon} \delta_i(\ol F)$,
 $\delta_h(\ol F)=\sum\limits_{i=2}^{[g/2]} \delta_i(\ol F)$,
 and $\delta_h(\bar f)=\sum\limits_{\ol F\in \Upsilon} \delta_h(\ol F)$.
 Then
 \begin{equation}\label{formulaofdelta_f}
 \left\{\begin{aligned}
 \delta(\ol F)&\,=\sum_{i=0}^{[g/2]}\delta_i(\ol F)=\delta_0(\ol F)+\delta_1(\ol F)+\delta_h(\ol F);\\
 \delta(\bar f)&\,=\sum_{i=0}^{[g/2]}\delta_i(\bar f)=\delta_0(\bar f)+\delta_1(\bar f)+\delta_h(\bar f).
 \end{aligned}\right.
 \end{equation}

 \subsection{Invariants for a family of superelliptic curves}\label{subsec-3-2} In this section we define the local invariants for families of superelliptic curves, and we show that the relative invariants  \eqref{defofrelativeinv} of such a family can be expressed as
 suitable combinations of the local invariants.\vspace{1mm}

 Recall from  \autoref{def-superelliptic} that, an $n$-superelliptic curve $\ol F$ has an affine equation of the form
 $$y^n=F(x)$$ with $F$ a separable polynomial.
 We always denote by $\alpha_0=\deg(F)$.
 As we have seen in \eqref{eqn-3-9}, the genus $g=g(\ol F)$ is determined by $n$ and $\alpha_0$.
 Moreover, there is an induced $n$-superelliptic cover $\pi:\,\ol F\to \bbp^1$,
 which is a cyclic cover with covering group $G\cong \mathbb Z/n\mathbb Z$,
 branch locus $R$, and local monodromy $a$ around $R$ as in \autoref{rems-3-1}\,(i).
Let $\alpha$ be given by \eqref{eqn-3-45}.
 Then $n\,|\,\alpha$ if and only if $n\,|\,\alpha_0$ or $n\,|\,(\alpha_0+1)$. Hence from \autoref{rems-3-1}\,(i) we see that
 \begin{equation}\label{eqn-3-80}
 \text{the cover $\pi$ is branched over $\infty$ with local monodromy $a_{\infty}\neq 1$ if and only if $n{\not|}~\alpha$.}
 \end{equation}


 {\vspace{1mm}}We first study local families of semi-stable superelliptic curves,
  i.e., restriction of a family $\bar f:S\ra T$ to an open subset $T_0\subset T$
  with $T_0\isom\{t\in\Cbb:|t|<1\}$ (for the analytic topology).
 \begin{lemma}\label{lem-3-2}
 	Let $\bar f_0:\,S_0 \to T_0$
 	be a local family of semi-stable superelliptic curves, with $T_0$ identified with  the open unit disk in $\Cbb$.
 	Assume that $\ol F_0=\bar f_0^{-1}(0)$ is the unique singular fiber.
 	Then after a suitable base change, the following statements hold:
 	\begin{enumerate}
 		\item the group $G\cong \mathbb Z/n\mathbb Z$ acts on $S_0$
 		such that it reduces to an $n$-superelliptic automorphism group
 		on the general fiber when taking the restriction,
 		and that the quotient map
 		$\Pi_0:\,S_0 \to S_0/G$ branches over $\alpha$ disjoint sections $\{D_i\}_{i=1}^{\alpha}$ of $\bar \varphi_0$
 		plus certain nodes in $\ol \Gamma_0=\bar \varphi_0^{-1}(0)$, where $\alpha$ is given in \eqref{eqn-3-45},
 		and $\bar \varphi_0:\,S_0/G \to T_0$ is the induced family;
 		\item the fiber $\ol \Gamma_0$ with $\ol \Gamma_0\cap \{D_i\}_{i=1}^{\alpha}$ as its marked points
 		is a semi-stable $\alpha$-pointed rational curve;
 		\item there are at least two marked points on $\ol \Gamma_0'$,
 		where $\ol \Gamma_0'$ is any one of the two connected components of $\ol \Gamma_0'\setminus\{y\}$ and $y$ is any node in $\ol \Gamma_0$.
 	\end{enumerate}
 \end{lemma}
 \begin{proof}
 	The first two statements are clear.
 	For the third one, first it is clear that $\ol \Gamma_0'$ contains at least one marked point.
 	Hence it suffices to derive a contradiction if $\ol \Gamma_0'$ contains exactly one marked point.
 	If this indeed occurs, then there is an irreducible component $C_0\subseteq \ol \Gamma_0'\subseteq \ol \Gamma_0$ such that
 	$C_0$ intersects $\ol \Gamma_0\setminus C_0$ at only one point and that there is a unique marked point on $C_0$.
 	It follows that $E_0$ is still a smooth rational curve with one intersection point
 	with $\ol F_0\setminus E_0$, i.e., $E_0$ is a $(-1)$-curve,
 	where $E_0\subseteq \ol F_0$ is any component in the inverse image of $C_0$.
 	This implies that the fiber $\ol F_0$ is not semi-stable, which is absurd.
 \end{proof}

 \begin{definition}\label{def-3-1}
 	Let $\bar f_0:\,S_0 \to T_0\triangleq \big\{t\,\big|\,|t|<1\big\}$
 	be a local family of semi-stable $n$-superelliptic curves
 	with $\ol F_0=\bar f_0^{-1}(0)$ as the unique singular fiber. The index of a node $x\in{\ol F}_0$, denoted as ${\rm index}(x)$, is defined as follows:
 	
 	(i). Assume first that $\bar f_0$ satisfies the three statements in \autoref{lem-3-2}.
 	Let $\ell=|G\cdot x|$ be the number of points in the $G$-orbit of $x$,
 	and $y\in \ol \Gamma_0=\ol F_0/G$ be the image of $x$.
 	Assume that $\ol \Gamma_0\setminus \{y\}=\ol \Gamma_0'\cup \ol \Gamma_0''$
 	such that $\ol \Gamma_0'$ (resp. $\ol \Gamma_0''$) contains $\gamma$
 	(resp. $\alpha-\gamma$) marked points with $\alpha-\gamma\geq \gamma$.
 	In the case when $n{\not|}~\alpha$,
 	the section $D_{\alpha}$ is assumed to be the one whose local monodromy is equal to $a_{\infty}$,
 	where $a_{\infty}$ \big($\neq 1$ by \eqref{eqn-3-80}\big) is the local monodromy of $\infty$ as in \eqref{eqn-3-46};
 	and we always assume that $p_{\alpha}:=D_{\alpha}\cap \ol \Gamma_0\in \ol \Gamma_0'$ if $\gamma=\alpha/2$.
 	Then we define ${\rm index}(x)$ to be the triple $(\gamma,\,\ell,\,k)$ with
 	$$
 	k=\left\{\begin{aligned}
 	&0, &\quad&\text{if $n\,|\,\alpha$, or if $n{\not|}~\alpha$ and $p_{\alpha} \in \ol \Gamma_0''$};\\
 	&1, &&\text{if $n{\not|}~\alpha$ and $p_{\alpha} \in \ol \Gamma_0'$}.
 	\end{aligned}\right.$$
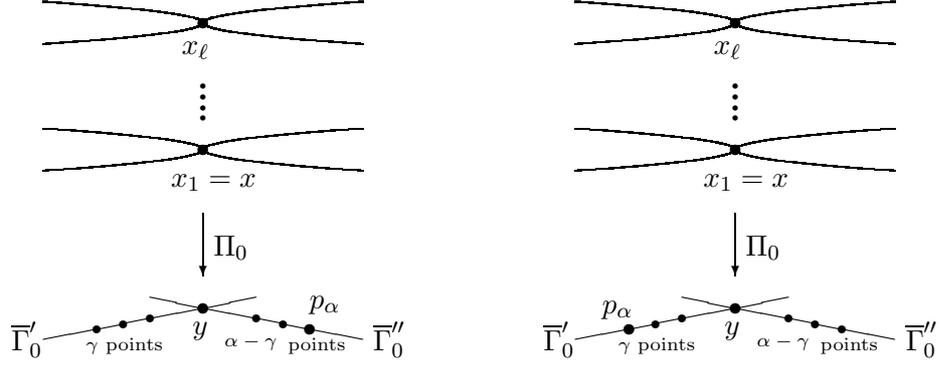
  \begin{center}
  	\setlength{\unitlength}{1.4mm}
  	\begin{picture}(80,37)
  	
  	\qbezier(30,31)(0,33)(30,35)
  	\qbezier(0,31)(30,33)(0,35)
  	
  	\multiput(15,24)(0,1){4}{\circle*{0.5}}
  	
  	\qbezier(30,19)(0,21)(30,23)
  	\qbezier(0,19)(30,21)(0,23)
  	
  	\put(0,3){\line(5,1){20}}
  	\put(30,3){\line(-5,1){20}}
  	
  	\put(15,15){\vector(0,-1){6}}
  	
  	\put(15,6){\circle*{1}}
  	\put(15,21){\circle*{1}}
  	\put(15,33){\circle*{1}}
  	
  	\put(14,3.5){$y$}
  	\put(12,17.5){$x_1=x$}
  	\put(13,30){$x_{\ell}$}
  	
  	\put(31,2){$\ol\Gamma_0''$}
  	\put(-3,2){$\ol\Gamma_0'$}
  	
  	\multiput(5,4)(2.5,0.5){3}{\circle*{0.8}}
  	\put(25,4){\circle*{1}}
  	\multiput(25,4)(-2.5,0.5){3}{\circle*{0.8}}
  	\put(25,6){$p_{\alpha}$}
  	
  	\put(4,2){\tiny $\gamma$ points}
  	\put(17,2.4){\tiny $\alpha-\gamma$}
  	\put(23,2){\tiny points}
  	
  	\put(16,11){$\Pi_0$}
  	
  	\put(-7,-5){typical case when ${\rm index}(x)=(\gamma,\,\ell,\,0)$;}

  	\qbezier(80,31)(50,33)(80,35)
  	\qbezier(50,31)(80,33)(50,35)
  	
  	\multiput(65,24)(0,1){4}{\circle*{0.5}}
  	
  	\qbezier(80,19)(50,21)(80,23)
  	\qbezier(50,19)(80,21)(50,23)
  	
  	\put(50,3){\line(5,1){20}}
  	\put(80,3){\line(-5,1){20}}
  	
  	\put(65,15){\vector(0,-1){6}}
  	
  	\put(65,6){\circle*{1}}
  	\put(65,21){\circle*{1}}
  	\put(65,33){\circle*{1}}
  	
  	\put(64,3.5){$y$}
  	\put(62,17.5){$x_1=x$}
  	\put(63,30){$x_{\ell}$}
  	
  	\put(81,2){$\ol\Gamma_0''$}
  	\put(47,2){$\ol\Gamma_0'$}
  	
  	\multiput(55,4)(2.5,0.5){3}{\circle*{0.8}}
  	\put(55,4){\circle*{1}}
  	\multiput(75,4)(-2.5,0.5){3}{\circle*{0.8}}
  	\put(52.5,5.5){$p_{\alpha}$}
  	
  	\put(54,2){\tiny $\gamma$ points}
  	\put(67,2.4){\tiny $\alpha-\gamma$}
  	\put(73,2){\tiny points}
  	
  	\put(66,11){$\Pi_0$}

  	\put(43,-5){typical case when ${\rm index}(x)=(\gamma,\,\ell,\,1)$.}
  	\end{picture}
  \end{center}
  \vspace{8mm}

 	
 	(ii). In the general case, we choose any base change $T_0'\ra T_0$ such that the family $\bar f_0':\,S_0'\to T_0'$ obtained by pulling-back
 	satisfies the three statements in \autoref{lem-3-2} and let $x'\in \ol F_0'$ be any node over $x$.
 	Then we define ${\rm index}(x):={\rm index}(x')$.
 	One checks that the definition is independent on the choices of the base change and the node $x'$.
 \end{definition}

 \begin{definition}\label{def-3-4}
 	(i). For any singular fiber $\ol F$ in a semi-stable family $\bar f:\, \ol S \to \ol B$
 	of $n$-superelliptic curves of genus $g\geq 2$,
 	we define the singularity index $s_{\gamma,\ell}(\ol F)$ (resp. $s_{\gamma,\ell}'(\ol F)$) of $\ol F$ to be the number of nodes in $\ol F$
 	with index equal to $(\gamma,\ell,0)$ (resp. $(\gamma,\ell,1)$).
 	
 	(ii). For a semi-stable family $\bar f:\, \ol S \to \ol B$ of $n$-superelliptic curves of genus $g\geq 2$,
 	we define the singularity indices of $\bar f$ as
 	$$s_{\gamma,\ell}(\bar f)=\sum s_{\gamma,\ell}(\ol F);\quad\qquad
 	s_{\gamma,\ell}'(\bar f)=\sum s_{\gamma,\ell}'(\ol F).$$
 	Here the sum takes over all singular fibers (which are finitely many) of $\bar f$.
 	If there is no confusion, we simply denote by $s_{\gamma,\ell}=s_{\gamma,\ell}(\bar f)$
 	and $s_{\gamma,\ell}'=s_{\gamma,\ell}'(\bar f)$.
 \end{definition}

 \begin{remarks}\label{rem-3-1}
 	(i). By the definition together with \autoref{lem-3-2},
 	the index $(\gamma,\ell,k)$ of a node $x$ satisfies
 	that $2\leq \gamma\leq [\alpha/2]$, and that
 	\begin{equation}\label{eqn-3-44}
 	\ell=\left\{\begin{aligned}
 	&\gcd(\gamma,n),&\quad&\text{if~}k=0;\\
 	&\gcd(\alpha-\gamma,\,n),&&\text{if~}k=1.
 	\end{aligned}\right.
 	\end{equation}
 	In other words,
 	 $$s_{\gamma,\ell}=s_{\gamma,\ell}'=0,\qquad\text{~if~}\gamma=1\text{~or~}\gamma>[\alpha/2];\qquad\qquad$$
 	$$s_{\gamma,\ell}=0,~\text{~if~}\ell\neq \gcd(\gamma,n);
 	\qquad s_{\gamma,\ell}'=0,~\text{~if~}\ell\neq \gcd(\alpha-\gamma,n).$$
 	The equality \eqref{eqn-3-44} needs some explanations:
 	By \autoref{def-3-1}, we may assume that the singular fiber
 	$\ol F_0$ admits a cyclic cover $\pi:\, \ol F_0 \to \ol \Gamma_0$ with covering group $G$,
 	where $\ol \Gamma_0$ is a singular rational curve.
 	Let $y=\pi(x)$, and $\ol \Gamma_0\setminus\{y\}=\ol \Gamma_0'\cup \ol \Gamma_0''$
 	with $\ol \Gamma_0'$ (resp. $\ol \Gamma_0''$) containing $\gamma$
 	(resp. $\alpha-\gamma\geq \gamma$) marked points.
 	Then
 	 \begin{list}{}
 	 	{\setlength{\labelwidth}{3mm}
 	 		\setlength{\leftmargin}{5mm}
 	 		\setlength{\itemsep}{2.5mm}}
 	 	\item[$\bullet$] if $k=0$: the local monodromy of each marked point on $\ol \Gamma_0'$ is equal to 1, and thus
 	 		 \begin{list}{}
 	 		 	{\setlength{\labelwidth}{1mm}
 	 		 		\setlength{\leftmargin}{3mm}
 	 		 		\setlength{\itemsep}{2.5mm}}
 	 		 	\item[$\cdot$] if $n\,|\,\gamma$: $\pi$ is not branched over $y$;
 	 		 	\item[$\cdot$] if $n{\not|~}\gamma$: the local monodromy around $y\in\ol\Gamma_0'$ for the restricted cyclic cover $\pi^\inv({\ol\Gamma_0'})\ra{\ol\Gamma'_0}$ is equal to $$a'_y=n([\frac{\gamma}{n}]+1)-\gamma$$
 	 		 \end{list}
 	 	\item[$\bullet$] if $k=1$: the local monodromy of each marked point on $\ol \Gamma_0''$ is equal to 1, and thus
 	 	\begin{list}{}
 	 		{\setlength{\labelwidth}{1mm}
 	 			\setlength{\leftmargin}{3mm}
 	 			\setlength{\itemsep}{2.5mm}}
 	 		\item[$\cdot$] if $n\,|\,(\alpha-\gamma)$: $\pi$ is not branched over $y$;
 	 		\item[$\cdot$] if $n{\not|~}(\alpha-\gamma)$: the local monodromy around $y\in\ol\Gamma''_0$ for the restricted cyclic cover is $$a_y''=n([\frac{\alpha-\gamma}{n}]+1)-(\alpha-\gamma).$$
 	 	\end{list}
 	 \end{list}
We thus obtain the value of $\ell$:
 	
 	$$\ell=|G\cdot x|=|\pi^{-1}(y)|=\left\{\begin{aligned}
 	&n=\gcd(\gamma,\,n), &&\text{if $k=0$ and $n\,|\,\gamma$};\\
 	&\gcd(a_y',\,n)=\gcd(\gamma,\,n), &\quad&\text{if $k=0$ and $n{\not|}~\gamma$};\\
 	&n=\gcd(\alpha-\gamma,\,n), &&\text{if $k=1$ and $n\,|\,(\alpha-\gamma)$};\\
 	&\gcd(a_y'',\,n)=\gcd(\alpha-\gamma,\,n), &\quad&\text{if $k=1$ and $n{\not|}~(\alpha-\gamma)$}.
 	\end{aligned}\right.$$
 	
 	(ii). For any singular fiber $\ol F$, we have
 	\begin{equation}\label{eqn-3-54}
 	s_{\gamma,\ell}'(\ol F)=0,~\forall~\gamma,\ell,\qquad\text{if~}n\,|\,\alpha,
 	\end{equation}
 	and
 	$$\begin{aligned}
 	&\quad\sum_{\gamma=2}^{[\alpha/2]}\sum_{\ell >1}\big(s_{\gamma,\ell}(\ol F)+s_{\gamma,\ell}'(\ol F)\big)=\delta_0(\ol F);\\[1mm]
 	&\left\{\begin{aligned}
 	&s_{\gamma,1}(\ol F)+s_{\gamma,1}'(\ol F)=\delta_{i(\gamma)}(\ol F), &\quad&\text{if~}r_{\infty}=n,\\[1mm]
 	&s_{\gamma,1}(\ol F)=\delta_{i(\gamma)}(\ol F),
 	\quad s_{\gamma,1}'(\ol F)=\delta_{j(\gamma)}(\ol F), &\quad&\text{if~}r_{\infty} \neq n,
 	\end{aligned}\right.
 	\end{aligned}$$
 	where $r_{\infty}=\frac{n}{\gcd(n,\alpha_0)}$ is the ramification index of the cover $\pi$ at $\infty$ as in \autoref{rems-3-1}\,(i),
 	and
 	$$\begin{aligned}
 	i(\gamma)&~=\frac{(n-1)(\gamma-1)}{2}\\
 	j(\gamma)&~=\frac{(n-1)(\gamma-2)+\frac{r_{\infty}-1}{r_{\infty}}\cdot n}{2}=g-\frac{(n-1)(\alpha-\gamma-1)}{2}.
 	\end{aligned}$$
 \end{remarks}

 \begin{theorem}\label{invariantstheorem}
 	Let $\olf:\,\ols\to \olb$ be a family of semi-stable $n$-superelliptic curves of genus $g\geq2$ over a projective curve,
 	and let $b_{\gamma}=\frac{(n^2-1)\gamma(\alpha-\gamma)}{\alpha-1}-n^2.$
 	Then
 	\begin{eqnarray}
 	\omega_{\ols/\olb}^2&\hspace{-2mm}=\hspace{-2mm}&
 	\left\{\begin{aligned}
 	&\sum_{\gamma,\,\ell} b_{\gamma}\cdot \frac{s_{\gamma,\ell}}{\ell^2}, &&\text{if $n\,|\,\alpha$},\\[1mm]
 	&\hspace{-2mm}\begin{aligned}
 	&~\sum_{\gamma,\,\ell} \left(b_{\gamma}
 	-\frac{(n^2-r_{\infty}^2)\gamma(\gamma-1)}{r_{\infty}^2(\alpha-1)(\alpha-2)}\right)
 	\cdot \frac{s_{\gamma,\ell}}{\ell^2}\\
 	&+\sum_{\gamma,\,\ell} \left(b_{\gamma}
 	 -\frac{(n^2-r_{\infty}^2)(\alpha-\gamma)(\alpha-\gamma-1)}{r_{\infty}^2(\alpha-1)(\alpha-2)}\right)
 	\cdot \frac{s_{\gamma,\ell}'}{\ell^2},
 	\end{aligned}&&\text{if $n{\not|}~\alpha$},
 	\end{aligned}\right.\label{K_f^2formula}\\[2mm]
 	\qquad \deg \olf_*\omega_{\ols/\olb}&\hspace{-2mm}=\hspace{-2mm}&
 	\left\{\begin{aligned}
 	&\frac{1}{12}\sum_{\gamma,\,\ell} (b_{\gamma}+\ell^2)
 	\cdot \frac{s_{\gamma,\ell}}{\ell^2}, &&\text{if $n\,|\,\alpha$},\\[1mm]
 	&\hspace{-2mm}\begin{aligned}
 	&~\frac{1}{12}\sum_{\gamma,\,\ell} \left(b_{\gamma}+\ell^2
 	-\frac{(n^2-r_{\infty}^2)\gamma(\gamma-1)}{r_{\infty}^2(\alpha-1)(\alpha-2)}\right)
 	\cdot \frac{s_{\gamma,\ell}}{\ell^2}\\
 	&+\frac{1}{12}\sum_{\gamma,\,\ell} \left(b_{\gamma}+\ell^2
 	 -\frac{(n^2-r_{\infty}^2)(\alpha-\gamma)(\alpha-\gamma-1)}{r_{\infty}^2(\alpha-1)(\alpha-2)}\right)
 	\cdot \frac{s_{\gamma,\ell}'}{\ell^2},
 	\end{aligned}&&\text{if $n{\not|}~\alpha$},
 	\end{aligned}\right.\label{chi_fformula}\\[2mm]
 	\delta(\bar f)&\hspace{-2mm}=\hspace{-2mm}&
 	\left\{\begin{aligned}
 	&\sum_{\gamma,\,\ell} s_{\gamma,\ell}, &~&\text{if $n\,|\,\alpha$},\\[1mm]
 	&\sum_{\gamma,\,\ell} \big(s_{\gamma,\ell}+s_{\gamma,\ell}'\big), &&\text{if $n{\not|}~\alpha$}.\\
 	\end{aligned}\right.\label{e_fformula}
 	\end{eqnarray}
 \end{theorem}

 %

 \begin{proof}
 	First, \eqref{e_fformula} follows directly from \autoref{def-3-4}; see also \autoref{rem-3-1}\,(ii).
 	Note also that \eqref{chi_fformula} follows from \eqref{formulanoether}, \eqref{K_f^2formula} and \eqref{e_fformula}.
 	So it suffices to prove \eqref{K_f^2formula}.
 	
 	
 	
 	Since any finite base change only modifies the two sides of \eqref{K_f^2formula} by a common multiple, we may, up to a finite base change, assume that there exists an action of $G\cong \mathbb Z/n\mathbb Z$
 	on $\ols$, such that $\ol Y=\ols/G$ is ruled over $\olb$,
 	and that the quotient $\ol \Pi:\, \ols \to \oly$
 	is branched over $\alpha$ disjoint sections
 	$\big\{D_i\big\}_{i=1}^{\alpha}$ of the induced family $\bar\varphi:\,\oly\to\olb$
 	and possibly some of the nodes in fibers of $\bar\varphi$ (cf. \autoref{lem-3-2}).
 	Moreover, if $n{\not|}~\alpha$, one may further assume that $D_{\alpha}$
 	is the section whose restriction to a general fiber
 	$\ol \Gamma\cong \bbp^1$ of $\bar\varphi$ corresponds to
 	the branch point $\infty$ of the restricted cover $\ol\Pi|_{\ol F}:\,\ol F \to \ol \Gamma\cong \bbp^1$.
 	In this case, the ramification index of $D_{\alpha}$ is $r_{\infty}=\frac{n}{\gcd(a_{\infty},n)}$.
 	Let's consider the following commutative diagram.
 	$$\xymatrix{
 		\ols \ar[rr]^-{\ol \Pi} \ar[d]_-{\bar f}
 		&& \ol Y=\ols/G \ar[d]^-{\bar \varphi}\\
 		\olb \ar@{=}[rr] && \olb
 	}$$

 	For any node $y_j$ in some fiber $\ol \Gamma_0$ of $\bar\varphi$,
 	if the local equation of $\oly$ around $y_j$ is given by $xy-t^{m_j}$,
 	then we call $m_j$ the multiplicity of $y_j$.
 	Let $\ell_j$ be the number of points in $\ols$ over $y_j$,
 	and $\ol \Gamma_0\setminus \{y_j\}=\ol \Gamma_0'\cup \ol \Gamma_0''$
 	with $\ol \Gamma_0'$ (resp. $\ol \Gamma_0''$) containing $\gamma_j$ (resp. $\alpha-\gamma_j\geq \gamma_j$) marked points.
 	If $n{\not|}~\alpha$ and $\gamma_j=\alpha/2$, then we assume that $D_{\alpha}\cap \ol \Gamma_0\in \ol \Gamma_0'$.
 	Then we define the index of $y_j$ to be $(\gamma_j,\,\ell_j,\,k_j)$, where
 	$$
 	k_j=\left\{\begin{aligned}
 	&0, &\quad&\text{if $n\,|\,\alpha$, or if $n{\not|}~\alpha$ and $D_{\alpha}\cap \ol \Gamma_0\in \ol \Gamma_0''$};\\
 	&1, &&\text{if $n{\not|}~\alpha$ and $D_{\alpha}\cap \ol \Gamma_0\in \ol \Gamma_0'$}.
 	\end{aligned}\right.$$
 	
 	With the notations introduced above, we claim that
 	\begin{claim}\label{claim-3-1}
 		(i). If $n\,|\,\alpha$, then
 		$$(\alpha-1)\sum_{i=1}^{\alpha}D_{i}^2=-\sum_{y_j} m_j\gamma_j(\alpha-\gamma_j).$$
 		
 		(ii). If $n{\not|}~\alpha$, then
 		$$\begin{aligned}
 		(\alpha-2)\sum\limits_{i=1}^{\alpha-1}D_i^2&\,=-\sum_{y_j} m_j(\gamma_j-k_j)(\alpha+k_j-1-\gamma_j);\\
 		(\alpha-1)\sum_{i=1}^{\alpha}D_{i}^2&\,=-\sum_{y_j} m_j\gamma_j(\alpha-\gamma_j).
 		\end{aligned}$$
 	\end{claim}
 	\begin{proof}[Proof of \autoref{claim-3-1}]
 		The proof is similar to that of \cite[Lemma\,4.8]{cornalba-harris-88}.
 		As an illustration, we prove (ii) here.
 		
 		Note that both of the sides are invariant if we resolve the singularities on $\ol Y$,
 		and hence we may assume that $\ol Y$ is smooth.
 		Moreover, we may contract $\oly$ to a $\mathbb P^1$ bundle $\varphi:\,Y\to \olb$ over $\olb$
 		such that the order of the singularities of $R_0=\sum D_{i,0}$ is at most $[\alpha/2]$,
 		and that $D_{\alpha,0}$ does not pass through any singularity of order equal to $\alpha/2$ if $n{\not|}~\alpha$,
 		where $D_{i,0}\subseteq Y$ is the image of $D_i$.
 		Note that in the $\mathbb P^1$ bundle $Y$, one has
 		$(D_{i,0}-D_{j,0})^2=0$, i.e., $D_{i,0}^2+D_{j,0}^2=2D_{i,0}\cdot D_{j,0}$. Hence
 		$$\begin{aligned}
 		(\alpha-2)\sum_{i=1}^{\alpha-1}D_{i,0}^2&\,=2\sum_{1\leq i<j\leq \alpha-1}D_{i,0}\cdot D_{j,0};\\
 		(\alpha-1)\sum_{i=1}^{\alpha}D_{i,0}^2&\,=2\sum_{1\leq i<j\leq \alpha}D_{i,0}\cdot D_{j,0}.
 		\end{aligned}$$
 		Now blowing up a point with exactly $\gamma$ of these sections passing through it will create
 		a node $y$ of index $(\gamma,\ell,k)$ in the fibers,
 		where $k=0$ if $D_{\alpha,0}$ does not pass through this point,
 		and $k=1$ if $D_{\alpha,0}$ passes through this point;
 		at the same time, the left hand side (resp. the right hand side) of the first equality above
 		decreases by $(\gamma-k)(\alpha-2)$ \big(resp. $(\gamma-k)(\gamma-k-1)$\big),
 		and the left hand side (resp. the right hand side) of the second equality above
 		decreases by $\gamma(\alpha-1)$ \big(resp. $\gamma(\gamma-1)$\big).
 		Thus
 		$$\begin{aligned}
 		(\alpha-2)\sum\limits_{i=1}^{\alpha-1}D_i^2&
 		\,=2\sum_{1\leq i<j\leq \alpha-1}D_{i}\cdot D_{j}-\sum_{y_j} m_j(\gamma_j-k_j)(\alpha+k_j-1-\gamma_j);\\
 		(\alpha-1)\sum_{i=1}^{\alpha}D_{i}^2&\,=2\sum_{1\leq i<j\leq \alpha}D_{i}\cdot D_{j}-\sum_{y_j} m_j\gamma_j(\alpha-\gamma_j).
 		\end{aligned}$$
 		Note that in $\ol Y$, these sections $\big\{D_i\big\}_{i=1}^{\alpha}$ are disjoint with each other,
 		i.e., $D_i\cdot D_j=0$ for any $i\neq j$.
 		Hence we complete the proof of the claim.
 	\end{proof}
 	
 	Come back to the proof of \eqref{K_f^2formula}.
 	Let $\xi_{\gamma,\ell}$ (resp. $\xi_{\gamma,\ell}'$) be the number of the nodes in fibers of $\bar\varphi$
 	with index $(\gamma,\ell,0)$ (resp. $(\gamma,\ell,1)$), counted according to their multiplicities.
 	Then
 	\begin{equation}\label{eqn-3-55}
 	s_{\gamma,\ell}=\ell\cdot \frac{\xi_{\gamma,\ell}}{n/\ell}=\frac{\ell^2}{n}\cdot \xi_{\gamma,\ell},
 	\qquad\quad s_{\gamma,\ell}'=\ell\cdot \frac{\xi_{\gamma,\ell}'}{n/\ell}=\frac{\ell^2}{n}\cdot \xi_{\gamma,\ell}'.
 	\end{equation}
 	By the definitions, we have (see also \autoref{rem-3-1}\,(ii))
 	\begin{equation}\label{eqn-3-52}
 	\text{$\xi_{\gamma,\ell}'=s_{\gamma,\ell}'=0$,\qquad if $n\,|\,\alpha$.}
 	\end{equation}
 	According to \autoref{claim-3-1}, we obtain that if $n\,|\,\alpha$, then
 	\renewcommand{\theequation}{\arabic{section}-\arabic{equation}a}
 	\begin{equation}\label{eqn-3-49-a}
 	\quad (\alpha-1)\sum_{i=1}^{\alpha}D_{i}^2=-\sum_{\gamma,\,\ell} \gamma(\alpha-\gamma)\xi_{\gamma,\ell}
 	=-\sum_{\gamma,\,\ell} \frac{n\gamma(\alpha-\gamma)}{\ell^2}\cdot s_{\gamma,\ell};
 	\end{equation}
 	and that if $n{\not|}~\alpha$, then
 	\addtocounter{equation}{-1}
 	\renewcommand{\theequation}{\arabic{section}-\arabic{equation}b}
 	\begin{equation}\label{eqn-3-49-b}
 	\left\{\begin{aligned}
 	(\alpha-2)\sum\limits_{i=1}^{\alpha-1}D_i^2
 	&\,=-\sum_{\gamma,\,\ell} \Big(\gamma(\alpha-1-\gamma)\xi_{\gamma,\ell}+(\gamma-1)(\alpha-\gamma)\xi_{\gamma,\ell}'\Big)\\
 	&\,=-\sum_{\gamma,\,\ell}\bigg(\frac{n\gamma(\alpha-1-\gamma)}{\ell^2}\cdot s_{\gamma,\ell}
 	+\frac{n(\gamma-1)(\alpha-\gamma)}{\ell^2}\cdot s_{\gamma,\ell}'\bigg),\\
 	(\alpha-1)\sum_{i=1}^{\alpha}D_{i}^2
 	&\,=-\sum_{\gamma,\,\ell} \gamma(\alpha-\gamma)(\xi_{\gamma,\ell}+\xi_{\gamma,\ell}')
 	=-\sum_{\gamma,\,\ell} \frac{n\gamma(\alpha-\gamma)}{\ell^2}\cdot (s_{\gamma,\ell}+s_{\gamma,\ell}').
 	\end{aligned}\right.
 	\end{equation}
 	\renewcommand{\theequation}{\arabic{section}-\arabic{equation}}
 	By the construction, we have
 	\begin{equation}\label{eqn-3-50}
 	\omega_{\oly/\olb}^2=-\sum_{\gamma,\,\ell} (\xi_{\gamma,\ell}+\xi_{\gamma,\ell}')
 	=-\sum_{\gamma,\,\ell} \frac{n(s_{\gamma,\ell}+s_{\gamma,\ell}')}{\ell^2};\qquad
 	\omega_{\oly/\olb}\cdot D_i=-D_i^2,\,~\,\forall~1\leq i\leq \alpha.
 	\end{equation}
 	By the Riemann-Hurwitz formula, one has
 	$$\omega_{\ols/\olb}^2=
 	\left\{\begin{aligned}
 	&\ol\Pi^*\Big(\omega_{\oly/\olb}+\frac{n-1}{n}\sum_{i=1}^{\alpha}D_i\Big), &\quad&\text{if $n\,|\,\alpha$};\\
 	 &\ol\Pi^*\Big(\omega_{\oly/\olb}+\frac{n-1}{n}\sum_{i=1}^{\alpha-1}D_i+\frac{r_{\infty}-1}{r_{\infty}}D_{\alpha}\Big),
 	&&\text{if $n{\not|}~\alpha$}.
 	\end{aligned}\right.$$
 	Hence if $n\,|\,\alpha$, then
 	\renewcommand{\theequation}{\arabic{section}-\arabic{equation}a}
 	\begin{equation}\label{eqn-3-51-a}
 	\omega_{\ols/\olb}^2=n\Big(\omega_{\oly/\olb}+\frac{n-1}{n}\sum_{i=1}^{\alpha}D_i\Big)^2
 	=n\Big(\omega_{\oly/\olb}^2-\frac{n^2-1}{n^2}\sum_{i=1}^{\alpha}D_i^2\Big),
 	\end{equation}
 	and if $n{\not|}~\alpha$, then
 	\addtocounter{equation}{-1}
 	\renewcommand{\theequation}{\arabic{section}-\arabic{equation}b}
 	\begin{equation}\label{eqn-3-51-b}
 	\begin{aligned}
 	 \omega_{\ols/\olb}^2&\,=n\Big(\omega_{\oly/\olb}+\frac{n-1}{n}\sum_{i=1}^{\alpha-1}D_i+\frac{r_{\infty}-1}{r_{\infty}}D_{\alpha}\Big)^2\\
 	 &\,=n\Big(\omega_{\oly/\olb}^2-\frac{n^2-1}{n^2}\sum_{i=1}^{\alpha-1}D_i^2-\frac{r_{\infty}^2-1}{r_{\infty}^2}D_{\alpha}^2\Big).
 	\end{aligned}
 	\end{equation}
 	\renewcommand{\theequation}{\arabic{section}-\arabic{equation}}
 	Combining the equalities \eqref{eqn-3-51-a}, \eqref{eqn-3-51-b},
 	\eqref{eqn-3-50}, \eqref{eqn-3-49-a}, \eqref{eqn-3-49-b}, and \eqref{eqn-3-52}
 	all together, we complete the proof of \eqref{K_f^2formula} and \autoref{invariantstheorem}.
 \end{proof}

 We conclude this subsection by proving the non-existence of hyperelliptic fibers in a family of $n$-superelliptic curves,
 which will be used to prove \autoref{prop-3-11} in \autoref{sec-pf-3-11}.

 \begin{lemma}\label{lem-3-23}
 	Let $\olf:\,\ols\to \olb$ be a family of semi-stable $n$-superelliptic curves
 	as in \autoref{invariantstheorem} with $\alpha\geq 5$.
 	Assume moreover that the general fiber of $\bar f$ is not hyperelliptic.
 	Then $\bar f$ admits no (smooth or singular) hyperelliptic fiber with a compact Jacobian.
 \end{lemma}
 \begin{proof}
 	We divide the proof into three steps.
 	
 	{\vspace{1.5mm}\noindent \bf Step I:}
 	We show that $\bar f$ admits no smooth hyperelliptic fiber.
 	
 	This is clear by the Castelnuovo-Severi inequality (cf. \cite{accola-06}).
 	Indeed, if there was such a smooth hyperelliptic fiber $\ol F$,
 	then $\ol F$ admits two different maps of degree $2$ and $n$ respectively onto $\bbp^1$
 	without common non-trivial factorization.
 	Hence $g\leq n-1$ by the Castelnuovo-Severi inequality (cf. \cite{accola-06}).
 	This is a contradiction to \eqref{eqn-3-9} once $\alpha\geq 5$.
 	
 	{\vspace{1.5mm}\noindent \bf Step II:}
 	Let $\ol F$ be a singular hyperelliptic fiber of $\olf$ with a compact Jacobian.
 	For any irreducible component $D\subseteq \ol F$,
 	we show that
 	\begin{enumerate}
 		\item[$\bullet$] either $g(D)=0$ and $D^2=-2$;
 		\item[$\bullet$] or $1\leq g(D)\leq \frac{n-1}{2}$ and $D^2=-1$.
 	\end{enumerate}
 	
 	Let $\ol F^{\#}$ be the stable model of $\ol F$.
 	Then it suffices to show $1\leq g(D)\leq \frac{n-1}{2}$ and $D^2=-1$ for any component $D\subseteq \ol F^{\#}$.
 	Note that $\ol F^{\#}$ admits two automorphisms $\tau$ and $\iota$,
 	where $\tau$ is of order $n$ and $\iota$ is the hyperelliptic involution.
 	Both the quotients $\ol F^{\#}/\langle\tau\rangle$ and $\ol F^{\#}/\langle\iota\rangle$ are trees of rational curves.
 	By \cite[Lemma\,5.7]{lu-zuo-14}, any component $D\subseteq \ol F^{\#}$ is not rational, i.e., $g(D)\geq 1$.
 	Hence both $\tau$ and $\iota$ act non-trivially on $D$.
 	It is clear that the hyperelliptic involution $\iota$ keeps $D$ invariant and acts faithfully on $D$.
 	Without loss of generality, we may assume that $\tau$ also keeps $D$ invariant and hence acts faithfully on $D$;
 	otherwise, we replace $\tau$ by a suitable power $\tau^{k_0}$
 	(the only difference is that the order of $\tau$ is smaller after this replacment).
 	For the sake of notations, we still denote by $\tau$ and $\iota$ their restriction on $D$.
 	Let
 	$$\Sigma_D=\{x\in D~|~x\text{~is a node of $\ol F^{\#}$}\}.$$
 	It is clear that $|\Sigma_D|=-D^2$,
 	Moreover, as $\ol F^{\#}$ admits a compact Jacobian,
 	$x$ is fixed by both $\tau$ and $\iota$ for any node $x\in \ol F^{\#}$;
 	in fact, let $\ol F'$ and $\ol F''$ be the two connected components of $\ol F^{\#}\setminus\{x\}$.
 	Then both $\ol F'$ and $\ol F''$ are kept invariant under $\tau$ (and $\iota$),
 	and $x$ is the unique intersection of $\ol F'$ and $\ol F''$ since $\ol F^{\#}$ admits a compact Jacobian.
 	Hence $x=\ol F'\cap \ol F''$ is fixed by both $\tau$ and $\iota$.
 	Thus $\tau$ and $\iota$ are commutative with each other as automorphisms of $D$,
 	since they admit at least one common fixed point on $D$ (i.e., any point in $\Sigma_D$).
 	It follows that $\tau$ induces a non-trivial automorphism $\tau'$ on $D/\langle\iota\rangle\cong \bbp^1$.
 	Let $|\tau|$ be the order of $\tau$, and
 	$$D_{\tau}=\left\{x\in D~|~\text{$x$ is fixed by $\tau^k$ for some $1\leq k<|\tau|$}\right\}\subseteq D,$$
 	and $\mu=\big|D_{\tau}\big|$.
 	Then by Hurwitz formula, one has
 	$$2g(D)\leq (n-1)(\mu-2).$$
 	Since $g(D)\neq 0$, it follows that $\mu\geq 3$.
 	Let ${\rm Fix}(\tau')\subseteq \bbp^1$ be the fixed locus of $\tau'$.
 	Then
 	\begin{equation*}
 	\big|{\rm Fix}(\tau')\big|=2, \quad\text{~and~}\quad \pi\big(D_{\tau}\big)\subseteq {\rm Fix}(\tau'),
 	\end{equation*}
 	where $\pi:\, D\to D/\langle\iota\rangle\cong \bbp^1$ is the quotient map, which is of degree two.
 	Note that $\Sigma_D\subseteq D_{\tau}$
 	and $\pi^{-1}\big(\pi(x)\big)=x$ for any $x\in \Sigma_D$.
 	It follows that
 	$$3\leq \mu=\big|D_{\tau}\big|
 	\leq 2\,\big|{\rm Fix}(\tau')\setminus \pi(\Sigma_D)\big|+|\Sigma_D|=4-|\Sigma_D|.$$
 	Therefore $\mu=3$, and hence $g(D)\leq \frac{n-1}{2}$ and $D^2=-|\Sigma_D|=-1$ as required.
 	
 	{\vspace{1.5mm}\noindent \bf Step III:}
 	We show that $\bar f$ admits no singular hyperelliptic fiber with a compact Jacobian.
 	
 	If there were such a singular hyperelliptic fiber $\ol F$,
 	let $k$ (resp. $k'$) be the number of rational components (resp. irrational components) in $\ol F$.
 	Then by Step II, one has
 	$$\delta(\ol F)=\frac12 \sum_{\{D\neq D'\}\subseteq \ol F} D\cdot D'=-\frac12\sum_{D\subseteq \ol F} D^2=\frac12(2k+k').$$
 	On the other hand, since $\ol F$ admits a compact Jacobian, we get
 	$$\delta(\ol F)=k+k'-1.$$
 	Combining the above two equalities, we obtain $k'=2$.
 	Hence $g\leq k\cdot 0 + k'\cdot \frac{n-1}{2}=n-1$,
 	which contradicts the Hurwitz formula \eqref{eqn-3-9} since $\alpha\geq 5$.
 	This completes the proof.
 \end{proof}

 \subsection{Flat part in the Higgs bundle for a family of superelliptic curves}\label{subsec-3-4}
 In this subsection, we study the flat part
 in the logarithmic Higgs bundle $\big(E_{\ol B}^{1,0}\oplus E_{\ol B}^{0,1},\,\theta_{\ol B}\big)$
 associated to a family $\bar f:\,\ols \to \olb$ of semi-stable superelliptic curves.
 The main purpose is to prove certain triviality of the flat part
 and the existence of fibration structions on $\ols$ up to base change.
 We will freely use the notations introduced at the beginning of \autoref{subsec-3-2}.

 As the main technique is based on cyclic covers,
 we always assume in this subsection that the group $G\cong \mathbb Z/n\mathbb Z$
 admits an action on $\ol S$ whose restriction on the general fiber
 is an $n$-superelliptic automorphism group.
 As we see in \autoref{rems-3-1},
 one can achieve this by a suitable finite base change (not necessarily \'etale).
 Note that the Galois cover $\ol\Pi:\,\ol S \to \ol Y=\ols/G$
 induces a Galois cover $\Pi':\,S'\to \wt Y$ whose branch locus is a normal crossing divisor
 and with ${\rm Gal}(\Pi')\cong G$,
 where $\wt Y \to \ol Y$ is the minimal resolution of singularities.
 Let $\wt S \to S'$ be the minimal resolution of singularities.
 Then there is an induced birational contraction $\wt S \to \ols$ fitting into the following commutative diagram.
 \vspace{-4mm}\begin{figure}[H]
 	$$\xymatrix{
 		\wt S \ar[r]\ar[d]\ar@/_7mm/"3,1"_-{\tilde f} \ar@/^4mm/"1,4"^-{\wt \Pi}
 		&S'\ar[rr]_-{\Pi'}&&\wt Y \ar[d]\ar@/^7mm/"3,4"^-{\tilde \varphi}\\
 		\ols \ar[d]^-{\bar f} \ar[rrr]^-{\ol\Pi}
 		&&& \ol Y \ar[d]_-{\bar\varphi}\\
 		\olb \ar@{=}[rrr] &&& \olb
 	}$$
 	\caption{Induced Galois cover}\label{figure-2}
 \end{figure}\vspace{-6mm}

 Since $\Pi':\,S'\to \wt Y$ is a cyclic cover of degree $n$,
 we may assume it is defined by the following relation:
 $$\mathcal L^{n} \equiv \mathcal O_Y\big(\wt{\mathfrak R}\big),\qquad\text{where `$\equiv$' stands for linear equivalence}.$$
 Following \cite{esnault-viehweg-92}, if $\wt {\mathfrak R}=\sum a_jD_j$ is the decomposition into prime divisors, we define
 $$\mathcal L^{(i)}=\mathcal L^i \otimes \mathcal O_{\wt Y}\Big(-\sum \Big[\frac{ia_j}{n}\Big]D_j\Big).$$
 Restricting to a general fiber $\wt \Gamma$ of $\tilde\varphi$, one has
 \begin{equation}\label{eqn-3-20}
 \mathcal L^{(i)}|_{\wt \Gamma} \cong \left\{\begin{aligned}
 &\mathcal O_{\bbp^1}\Big(\frac{i\alpha_0}{n}\Big), &\quad&\text{if~}\frac{i\alpha_0}{n}\in \mathbb{Z};\\
 &\mathcal O_{\bbp^1}\Big(\Big[\frac{i\alpha_0}{n}\Big]+1\Big), &&\text{if~}\frac{i\alpha_0}{n} \not\in \mathbb{Z}.
 \end{aligned}\right.
 \end{equation}
 Let $\wt R\subseteq \wt Y$ be the reduced branch divisor of $\Pi'$, i.e., the support of $\wt {\mathfrak R}$.
 Then by \cite[Lemma\,1.7]{viehweg-82}, one has the inclusion
 \begin{equation}\label{eqn-3-12}
 \tau_1:~\wt\Pi_*\Omega^1_{\wt S}\hookrightarrow
 \Omega^1_{\wt Y} \bigoplus
 \left(\bigoplus_{j=1}^{n-1}\Omega^1_{\wt Y}(\log \wt R)\otimes {\mathcal L^{(j)}}^{-1}\right).
 \end{equation}

 %

 \begin{lemma}\label{lem-3-1}
 	For each $1\leq i\leq n-1$, there is a sheaf morphism
 	\begin{equation}\label{eqn-3-15}
 	\varrho_i:~\tilde \varphi^*F^{1,0}_{\olb,i} \lra \wt\Pi_*\Omega^1_{\wt S},
 	\end{equation}
 	such that the induced canonical morphism
 	\begin{equation}\label{eqn-3-13}
 	F^{1,0}_{\olb,i}=\tilde \varphi_*\tilde \varphi^*F^{1,0}_{\olb,i}
 	\lra \tilde \varphi_*\wt\Pi_*\Omega^1_{\wt S}
 	=\tilde f_*\Omega^1_{\wt S}=\bar f_*\Omega^1_{\ol S}
 	\lra \bar f_*\Omega^1_{\ol S/\ol B}(\log\Upsilon)=E^{1,0}_{\olb}
 	\end{equation}
 	coincides with the inclusion $F^{1,0}_{\olb,i} \hookrightarrow F^{1,0}_{\olb}
 	\hookrightarrow E^{1,0}_{\olb}$.
 	Moreover, we may choose $\varrho_i$ so that the image of $\varrho_i$ is contained in
 	$\Omega^1_{\wt Y}(\log \wt R)\otimes {\mathcal L^{(i)}}^{-1}$ under the inclusion \eqref{eqn-3-12}.
 \end{lemma}
 \begin{proof}
 	According to \cite[Corollary\,7.2]{lu-zuo-14}, there exists a sheaf morphism
 	\begin{equation}\label{eqn-3-5}
 	\bar\varrho:~\bar f^*F^{1,0}_{\olb}\lra \Omega^1_{\ol S},
 	\end{equation}
 	such that the induced canonical morphism
 	\begin{equation*}
 	F^{1,0}_{\olb}=\bar f_*\bar f^*F^{1,0}_{\olb}\lra \bar f_*\Omega^1_{\ol S}\lra \bar f_*\Omega^1_{\ol S}(\log\Upsilon)
 	\lra \bar f_*\Omega^1_{\ol S/\ol B}(\log\Upsilon)=E^{1,0}_{\olb}
 	\end{equation*}
 	coincides with the inclusion $F^{1,0}_{\olb}\hookrightarrow  E^{1,0}_{\olb}.$
 	Since $\tilde\rho^*\Omega^1_{\ol S} \subseteq \Omega^1_{\wt S}$,
 	by pulling back \eqref{eqn-3-5}, we obtain a sheaf morphism
 	$$\tilde f^*F^{1,0}_{\olb}=\tilde\rho^* \bar f^*F^{1,0}_{\olb} \lra \Omega^1_{\wt S},$$
 	which corresponds to an element
 	$$\tilde\eta\in H^0\big(\wt S, \Omega^1_{\wt S}\otimes \tilde f^*{\big(F^{1,0}_{\olb}\big)}^{\vee}\big).$$
 	By pushing-out, we also obtain an element
 	$$\wt\Pi_*(\tilde\eta)\in H^0\left(\wt Y, \wt\Pi_*\big(\Omega^1_{\wt S}\otimes
 	\tilde f^*{\big(F^{1,0}_{\olb}\big)}^{\vee}\big)\right)
 	=H^0\left(\wt Y, \wt\Pi_*\Omega^1_{\wt S}\otimes \tilde \varphi^*{\big(F^{1,0}_{\olb}\big)}^{\vee}\right).$$
 	Hence one gets a sheaf morphism
 	\begin{equation}\label{eqn-3-14}
 	\varrho:~\tilde\varphi^*F^{1,0}_{\olb} \lra \wt\Pi_*\Omega^1_{\wt S}.
 	\end{equation}
 	By restricting to $\tilde \varphi^*F^{1,0}_{\olb,i}$, we obtain $\varrho_i$
 	as in \eqref{eqn-3-15} such that the induced morphism \eqref{eqn-3-13}
 	coincides with the inclusion $F^{1,0}_{\olb,i} \hookrightarrow F^{1,0}_{\olb}
 	\hookrightarrow E^{1,0}_{\olb}$.
 	
 	Note that the group $G$ acts on both sides of \eqref{eqn-3-5}.
 	One may require that the morphism $\bar\varrho$ is equivariant with respect to $G$.
 	So it is with the morphism $\varrho$.
 	Combining \eqref{eqn-3-14} with \eqref{eqn-3-12},
 	we obtain a sheaf morphism
 	$$\tilde\varphi^*F^{1,0}_{\olb} \lra
 	\Omega^1_{\wt Y} \bigoplus
 	\left(\bigoplus_{j=1}^{n-1}\Omega^1_{\wt Y}(\log \wt R)\otimes {\mathcal L^{(j)}}^{-1}\right),$$
 	which is compatible with the $G$-actions on both sides.
 	Hence the the image of $\varrho_i$ is contained in
 	$\Omega^1_{\wt Y}(\log \wt R)\otimes {\mathcal L^{(i)}}^{-1}$.
 \end{proof}

 %
 %

 \begin{lemma}\label{lem-3-20}
 	Let $\wt R\subseteq \wt Y$ be the reduced branch divisor of $\Pi'$ as above,
 	and $\wt \Gamma$ be a general fiber of $\tilde\varphi$.
 	Assume that $\wt R$ contains at least one section of $\tilde\varphi$,
 	and that there exist $1\leq i_1\leq i_2 \leq n-1$ such that
 	$F^{1,0}_{\olb,i_1}\neq 0$, $F^{1,0}_{\olb,i_2}\neq 0$, and
 	\begin{equation}\label{eqn-3-90}
 	\wt \Gamma \cdot \bigg(\omega_{\wt Y}(\wt R) \otimes
 	\left({\mathcal L^{(i_1)}}^{-1} \otimes {\mathcal L^{(i_2)}}^{-1}\right)\bigg)<0.
 	\end{equation}
 	Then both $F^{1,0}_{\olb,i_1}$ and $F^{1,0}_{\olb,i_2}$ become trivial
 	after a   suitable finite \'etale base change.
 \end{lemma}
 \begin{proof}
 	We divide the proof into two steps.
 	
 	{\vspace{1.5mm}\noindent\bf Step I:}
 	We show that for any non-zero unitary subbundle
 	$\hat{\mathcal U}\subseteq F^{1,0}_{\olb,i}$ with $i=i_1$ or $i_2$,
 	the image $\varrho_i(\tilde\varphi^*\hat{\mathcal U})$ is an invertible subsheaf $\hat M$
 	such that $\hat M$ is numerically effective (nef), $\hat M^2 = 0$,
 	and $\hat M\cdot D = 0$ for any component $D\subseteq \wt R_h$,
 	where $\varrho_i$ is given in \eqref{eqn-3-15}.
 	
 	\vspace{0.5mm}
 	The proof of this step is similar to that of \cite[Lemma\,7.3]{lu-zuo-14},
 	and here we only do it for $i=i_1$, as the case for $i=i_2$ is completely parallel.
 	By \autoref{lem-3-1}, the image $\varrho_{i_1}(\tilde\varphi^*\hat{\mathcal U})$ is contained in
 	$\Omega^1_{\wt Y}(\log \wt R)\otimes {\mathcal L^{(i_1)}}^{-1}$,
 	and it is non-zero if $\hat{\mathcal U}\neq 0$.
 	Mimicking the proof of \cite[Lemma\,7.3]{lu-zuo-14},
 	it suffices to show that the image $\varrho_{i_1}(\tilde\varphi^*\hat{\mathcal U})$ is a subsheaf of rank one
 	if $\hat{\mathcal U}\neq 0$.
 	We prove this by contradiction. Assume that it is not the case.
 	By taking wedge-product, one obtains a non-zero map
 	$$\begin{aligned}
 	\tau\circ\varrho_{i_1}\wedge \tau\circ\varrho_{i_2}:~
 	\tilde\varphi^*\hat{\mathcal U} \otimes \tilde\varphi^*F^{1,0}_{\olb,i_2}
 	\quad\lra\quad& \wedge^2\Omega^1_{\wt Y}(\log \wt R)\otimes \left({\mathcal L^{(i_1)}}^{-1} \otimes {\mathcal L^{(i_2)}}^{-1}\right)\\
 	=\,&\,\omega_{\wt Y}(\wt R) \otimes \left({\mathcal L^{(i_1)}}^{-1} \otimes {\mathcal L^{(i_2)}}^{-1}\right).
 	\end{aligned}$$
 	Denote by $\mathcal C$ the image of the above map, we proceed to establish the semi-positivity of $\mathcal C$; here we recall that a locally free sheaf $\mathcal E$ on $\wt Y$ is called semi-positive, if for any morphism $\psi: Z \to \wt Y$ from a smooth complete curve $Z$,
 	the pulling-back $\psi^*\mathcal E$ has no quotient line bundle of negative degree.
 	 \begin{list}{}
 	 	{\setlength{\labelwidth}{3mm}
 	 		\setlength{\leftmargin}{5mm}
 	 		\setlength{\itemsep}{2.5mm}}
 	 	\item[$\bullet$] On the one hand,
 	 	for any morphism $\psi: Z \to \wt Y$ from a smooth complete curve $Z$,
 	 	$\psi^*\big(\tilde\varphi^*F^{1,0}_{\olb,i_1} \otimes \tilde\varphi^*F^{1,0}_{\olb,i_2}\big)$
 	 	is poly-stable of slope zero since it comes from a unitary representation (cf. \cite{narasimhan-sesdahri-65}),
 	 	which implies that $\tilde\varphi^*\hat{\mathcal U} \otimes \tilde\varphi^*F^{1,0}_{\olb,i_2}$ is semi-positive.
 	 	Therefore, as a quotient of $\tilde\varphi^*\hat{\mathcal U} \otimes \tilde\varphi^*F^{1,0}_{\olb,i_2}$,
 	 	$\mathcal C$ is also semi-positive.
 	 	
 	 	\item[$\bullet$] On the other hand, from \eqref{eqn-3-90} it follows that
 	 	$\omega_{\wt Y}(\wt R) \otimes \left({\mathcal L^{(i_1)}}^{-1} \otimes {\mathcal L^{(i_2)}}^{-1}\right)$
 	 	can not contain any non-zero semi-positive subsheaf.
 	 	It contradicts the semi-positivity of $\mathcal C$.
 	 \end{list}
%
%
 	Hence the image $\varrho_{i_1}(\tilde\varphi^*\hat{\mathcal U})$ is a subsheaf of rank one as required.
 	
 	{\vspace{1.5mm}\noindent\bf Step II:}
 	We show that both $F^{1,0}_{\olb,i_1}$ and $F^{1,0}_{\olb,i_2}$ become trivial
 	after a suitable finite \'etale base change.
 	
 	\vspace{0.5mm}
 	The proof of this step is similar to that of \cite[Proposition\,7.4]{lu-zuo-14}.
 	In fact, by \cite[\S\,4.2]{deligne-71}, 
 	it suffices to show that $F^{1,0}_{\olb,i}$ is a direct sum of line bundles
 	after a suitable unramified base change for $i=i_1$ or $i_2$.
 	By assumption, $\wt R_h$ contains at least one section of $\tilde\varphi$.
 	Let $D\subseteq \wt R_h$ be such a section, and
 	$$F^{1,0}_{\olb,i}=\bigoplus_j\, \mathcal U_{ij}$$
 	be the decomposition of $F^{1,0}_{\olb,i}$ into irreducible subbundles.
 	By Step I with the unitary subbundle $\mathcal U_{ij}\subseteq F^{1,0}_{\olb,i}$,
 	we obtain $M_{ij}\cdot D=0$,
 	i.e., $\deg\mathcal O_D(M_{ij})=0$, where $M_{ij}$ is the image $\varrho_i(\tilde\varphi^*\mathcal U_{ij})$.
 	As $D$ is a section, $D\cong \ol B$. Hence we may view $\mathcal O_D(M_{ij})$ as an invertible
 	sheaf on $\ol B$, which is a quotient of $\mathcal U_{ij}$
 	since $M_{ij}$ is a quotient of $\tilde\varphi^*\mathcal U_{ij}$.
 	As $\mathcal U_{ij}$ comes from a unitary local system, $\mathcal U_{ij}$ is poly-stable.
 	Thus $\mathcal U_{ij}=\mathcal O_D(M_{ij}) \oplus \mathcal U_{ij}'$.
 	Because $\mathcal U_{ij}$ is irreducible,
 	$\mathcal U_{ij}=\mathcal O_D(M_{ij})$ is a line bundle as required.
 \end{proof}

 \begin{remark}\label{rem-3-4}
 	Let $a_{\infty}$ be the local monodromy around $\infty$ for the induced cyclic cover to $\mathbb P^1$
 	on the general fiber of $\bar f$ as in \eqref{eqn-3-46}.
 	If $a_{\infty}\neq 1$, then $\wt R$ clearly contains a section of $\tilde\varphi$.
 	In general, as pointed out at the beginning in the proof of \autoref{invariantstheorem},
 	after a suitable finite base change (may not be \'etale),
 	$\wt R$ consists of distinguished sections of $\tilde\varphi$ (i.e., the inverse image of $\{D_i\}_{i=1}^{\alpha}$)
 	plus certain components in the fibers of $\tilde\varphi$ (i.e., the inverse image of certain nodes in fibers of $\bar \varphi$).
 	In particular, we can always achieve the assumption that $\wt R$ contains at least one section of $\tilde\varphi$
 	using base change (may not be \'etale).
 \end{remark}

 \begin{corollary}\label{cor-3-3}
 	Assume that $\wt R$ contains at least one section of $\tilde\varphi$,
 	and that $F^{1,0}_{\olb,i_0}\neq 0$ for some $i_0\geq n/2$.
 	Then after a suitable unramified base change,
 	$F^{1,0}_{\olb,i}$ is trivial for any $n-i_0\leq i \leq i_0$.
 \end{corollary}
 \begin{proof}
 	Note that
 	$$\wt \Gamma\cdot \omega_{\wt Y}(\wt R)=\left\{\begin{aligned}
 	&\alpha_0-2, &\quad&\text{if~}n\,|\,\alpha_0;\\[1mm]
 	&\alpha_0-1, &&\text{if~}n{\not|}~\alpha_0.
 	\end{aligned}\right.$$
 	Hence by \eqref{eqn-3-20}, one checks easily that
 	$$\wt \Gamma\cdot \bigg(\omega_{\wt Y}(\wt R) \otimes \left({\mathcal L^{(i)}}^{-1} \otimes {\mathcal L^{(i_0)}}^{-1}\right)\bigg)<0,\qquad \forall~n-i_0\leq i \leq i_0.$$
 	Hence by applying \autoref{lem-3-20} to the case when $i_1=i_0$ and $i_2=i$ with $n-i_0\leq i \leq i_0$ and $F^{1,0}_{\olb,i}\neq 0$,
 	we complete the proof.
 \end{proof}

 In the rest part of this subsection,
 we focus on the existence of a fibration structure on $\ols$.
 Since $G$ acts on $\ol S$,
 it induces an action on $H^0(\ol S,\,\Omega^1_{\ol S})$.
 Let
 $$H^0\big(\ol S,\,\Omega^1_{\ol S}\big)=\bigoplus_{i=0}^{n-1} H^0\big(\ol S,\,\Omega^1_{\ol S}\big)_i$$
 be the eigenspace decomposition.
 Then $H^0\big(\ol S,\,\Omega^1_{\ol S}\big)_0\cong \bar f^*H^0\big(\ol B,\,\Omega^1_{\ol B}\big)$,
 and according to \cite[Theorem\,3.1]{fujita-78},
 \begin{equation}\label{eqn-3-32}
 \dim H^0\big(\ol S,\,\Omega^1_{\ol S}\big)_i=\rank \big(F^{1,0}_{\olb,i}\big)^{tr},\quad\forall~1\leq i\leq n-1,
 \end{equation}
 where $\big(F^{1,0}_{\olb,i}\big)^{tr}\subseteq F^{1,0}_{\olb,i}$ is the trivial part contained in
 the flat bundle $F^{1,0}_{\olb,i}$ as in \autoref{lem-3-12}.

 \begin{lemma}\label{lem-3-21}
 	Assume that there exist $1\leq i_1\leq i_2 \leq n-1$, such that
 	\eqref{eqn-3-90} holds, and that $H^0\big(\ol S,\,\Omega^1_{\ol S}\big)_{i_1}\neq 0$ and $H^0\big(\ol S,\,\Omega^1_{\ol S}\big)_{i_2}\neq 0$.
 	Then there exists a unique fibration $\bar f':\,\ol S \to \ol B'$ such that
 	\begin{equation}\label{eqn-3-91}
 	H^0\big(\ol S,\,\Omega^1_{\ol S}\big)_{i}
 	\subseteq \big(\bar f'\big)^*H^0\big(\ol B',\,\Omega^1_{\ol B'}\big), \qquad\text{for $i=i_1,i_2$}.
 	\end{equation}
 \end{lemma}
 \begin{proof}
 	First, the uniqueness of such a fibration is clear.
 	It suffices to show the existence of such a fibration with the property \eqref{eqn-3-91}.
 	By Castelnuovo-de Franchis lemma (cf. \cite[Theorem\,IV-5.1]{bhpv-04}), it is enough to show that
 	\begin{equation}\label{eqn-3-69}
 	\omega_{i_1}\wedge\omega_{i_2}=0, \quad\forall~\omega_{i_1}\in H^0\big(\ol S,\,\Omega^1_{\ol S}\big)_{i_1}
 	\text{~and~}\omega_{i_2}\in H^0\big(\ol S,\,\Omega^1_{\ol S}\big)_{i_2}.
 	\end{equation}
 	As an easy case, if $i_1+i_2=n$, then
 	it is clear that
 	$$\omega_{i_1}\wedge \omega_{i_2} \in H^0\big(\ol S,\,\Omega^2_{\ol S}\big)^{G}
 	=\ol\Pi^*H^0\big(\ol Y,\,\Omega^2_{\ol Y}\big)=0,\quad
 	\text{since $\ol Y$ is ruled over $\ol B$}.$$
 	In general, we have to apply a more detailed description on the differential sheaves
 	associated to cyclic covers due to Esnault-Viehweg \cite{esnault-viehweg-92}.
 	First note that $H^0\big(\wt S,\,\Omega^k_{\wt S}\big)\cong H^0\big(\ol S,\,\Omega^k_{\ol S}\big)$
 	for $k=1,2$,
 	and this isomorphism is compatible with the action of the group $G$.
 	It suffices to prove \eqref{eqn-3-69} for $\wt S$.
 	There is an inclusion
 	$$\iota_k:~\wt\Pi_*\Omega^k_{\wt S}\hookrightarrow
 	\Omega^k_{\wt Y} \bigoplus
 	\left(\bigoplus_{i=1}^{n-1}\Omega^k_{\wt Y}(\log \wt R)\otimes {\mathcal L^{(i)}}^{-1}\right).$$
 	Taking global sections, we obtain an injection
 	$$\iota_k:~H^0\big(\wt S,\,\Omega^k_{\wt S}\big)=H^0\big(\wt Y,\,\Pi_*\Omega^k_{\wt S}\big)
 	\,\hookrightarrow\, H^0\big(\wt Y,\,\Omega^k_{\wt Y}\big) \bigoplus
 	\left(\bigoplus_{i=1}^{n-1} H^0\big(\wt Y,\,\Omega^k_{\wt Y}(\log \wt R)\otimes {\mathcal L^{(i)}}^{-1}\big)\right),$$
 	which is compatible with the action of the group $G$ on both sides,
 	i.e.,
 	$$\left\{\begin{aligned}
 	\iota_k\Big(H^0\big(\wt S,\,\Omega^k_{\wt S}\big)_0\Big) &\,\subseteq H^0\big(\wt Y,\,\Omega^k_{\wt Y}\big),\\
 	\iota_k\Big(H^0\big(\wt S,\,\Omega^k_{\wt S}\big)_i\Big)
 	&\,\subseteq H^0\big(\wt Y,\,\Omega^k_{\wt Y}(\log \wt R)\otimes {\mathcal L^{(i)}}^{-1}\big),
 	\quad \forall~1\leq i\leq n-1.
 	\end{aligned}\right.$$
 	By pulling back, we have the following injection map, which is just the inclusion map.
 	$$
 	\wt\Pi^*(\iota_k):~ H^0\big(\wt S,\,\Omega^k_{\wt S}\big)_i \hookrightarrow
 	H^0\big(\wt S,\,\Omega^k_{\wt S}(\log \wt R')\otimes \wt\Pi^*{\mathcal L^{(i)}}^{-1}\big),
 	\quad \forall~1\leq i\leq n-1,
 	$$
 	where $\wt R'$ is the support of the divisor $\wt\Pi^{-1}(\wt R)$.
 	Since the wedge-product operation is clearly commutative with the inclusion map,
 	it follows that for any two $1$-forms
 	$\omega_{i_1}\in H^0\big(\ol S,\,\Omega^1_{\ol S}\big)_{i_1}$
 	and $\omega_{i_2}\in H^0\big(\ol S,\,\Omega^1_{\ol S}\big)_{i_2}$, one has
 	$$\omega_{i_1}\wedge\omega_{i_2}\in
 	H^0\big(\wt S,\,\Omega^k_{\wt S}(\log \wt R')\otimes
 	\wt\Pi^*{\mathcal L^{(i_1)}}^{-1}\otimes \wt\Pi^*{\mathcal L^{(i_2)}}^{-1}\big).$$
 	Let $\wt F\subseteq \wt S$ be a general fiber of $\tilde f$, and $\wt \Gamma=\wt\Pi\big(\wt F\big)$.
 	Then by the assumption \eqref{eqn-3-90}, one obtains that
 	$$
 	\wt F\cdot \Big(\Omega^2_{\wt S}(\log \wt R')\otimes
 	\wt\Pi^*{\mathcal L^{(i_1)}}^{-1}\otimes \wt\Pi^*{\mathcal L^{(i_2)}}^{-1} \Big)
 	=n\,\wt \Gamma \cdot\Big(\omega_{\wt Y}(\wt R) \otimes {\mathcal L^{(i_1)}}^{-1} \otimes {\mathcal L^{(i_2)}}^{-1}\Big)<0.
 	$$
 	Hence
 	$$H^0\big(\wt Y,\,\Omega^2_{\wt Y}(\log \wt R)\otimes {\mathcal L^{(i_1)}}^{-1}\otimes {\mathcal L^{(i_2)}}^{-1}\big)=0.$$
 	This shows that $\omega_{i_1}\wedge\omega_{i_2}=0$ as required.
 \end{proof}

 \begin{corollary}\label{cor-3-11}
 	Assume that $H^0\big(\ol S,\,\Omega^1_{\ol S}\big)_{i_0}\neq 0$ for some $i_0\geq n/2$.
 	Then after a suitable base change, there exists a unique fibration $\bar f':\,\ol S \to \ol B'$ such that
 	\begin{equation}\label{eqn-3-42}
 	\bigoplus_{i= n-i_0}^{i_0} H^0\big(\ol S,\,\Omega^1_{\ol S}\big)_{i}
 	\subseteq \big(\bar f'\big)^*H^0\big(\ol B',\,\Omega^1_{\ol B'}\big).
 	\end{equation}
 \end{corollary}
 \begin{proof}
 	Since $H^0\big(\ol S,\,\Omega^1_{\ol S}\big)_{i_0}\neq 0$,
 	it follows from \eqref{eqn-3-32} that $F^{1,0}_{\olb,i_0}\neq 0$.
 	Hence by \autoref{cor-3-3} and \autoref{rem-3-4}, after a suitable base change,
 	$F^{1,0}_{\olb,i}$ is trivial for any $n-i_0\leq i \leq i_0$.
     Combining this with \eqref{eqn-3-35} and \eqref{eqn-3-102}, one obtains
 	$$\dim H^0\big(\ol S,\,\Omega^1_{\ol S}\big)_{n-i_0}=\rank F^{1,0}_{\olb,n-i_0}
 	\geq \rank F^{1,0}_{\olb,i_0}=\dim H^0\big(\ol S,\,\Omega^1_{\ol S}\big)_{i_0}>0.$$
 	According to the proof of \autoref{cor-3-3},
 	the assumption \eqref{eqn-3-90} holds if $i_1=i_0$ and $i_2=n-i_0$.
 	Hence by \autoref{lem-3-21}, there exists a unique fibration
 	$\bar f'_{n-i_0}:\,\ols \to \olb'_{n-i_0}$ such that
 	$$H^0\big(\ol S,\,\Omega^1_{\ol S}\big)_{i_0} \oplus H^0\big(\ol S,\,\Omega^1_{\ol S}\big)_{n-i_0}
 	\subseteq \big(\bar f'_{n-i_0}\big)^*H^0\big(\ol B'_{n-i_0},\,\Omega^1_{\ol B'_{n-i_0}}\big).$$
 	In fact, the same holds also if we replace $n-i_0$ by any $i$
 	satisfying that $n-i_0\leq i\leq i_0$ and that $H^0\big(\ol S,\,\Omega^1_{\ol S}\big)_{i}\neq 0$,
 	i.e., there exists a unique fibration
 	$\bar f'_i:\,\ols \to \olb'_i$ such that
 	$$H^0\big(\ol S,\,\Omega^1_{\ol S}\big)_{i_0} \oplus H^0\big(\ol S,\,\Omega^1_{\ol S}\big)_{i}
 	\subseteq \big(\bar f'_i\big)^*H^0\big(\ol B'_i,\,\Omega^1_{\ol B'_i}\big).$$
 	Since $H^0\big(\ol S,\,\Omega^1_{\ol S}\big)_{i_0}\neq 0$,
 	from the uniqueness of these fibrations $\bar f'_i$'s, it follows that
 	these $\bar f'_i$'s are in fact the same one.
 	We denote such a fibration by $\bar f':\,\ol S \to \ol B'$,
 	which is of course unique and the inclusion \eqref{eqn-3-42} holds.
 	This completes the proof.
 \end{proof}

 From the uniqueness of $\bar f'$ obtained in \autoref{lem-3-21},
 it follows that there is an induced map
 \begin{equation}\label{eqn-3-71}
 \iota:\,G\lra \Aut(\ol B').
 \end{equation}
 Since $G\cong \mathbb Z/n\mathbb Z$ is a cyclic group, ${\rm Ker}(\iota)\cong \mathbb Z/m\mathbb Z$ for some $m$ with $m\,|\,n$.

 \begin{lemma}\label{lem-3-17}
 	Assume that there exist $1\leq i_1\leq i_2 \leq n-1$, such that
 	\eqref{eqn-3-90} holds, and that $H^0\big(\ol S,\,\Omega^1_{\ol S}\big)_{i_1}\neq 0$ and $H^0\big(\ol S,\,\Omega^1_{\ol S}\big)_{i_2}\neq 0$.
 	Let $\bar f':\,\ol S \to \ol B'$ be the fibration obtained in \autoref{lem-3-21}, and $\iota$ be given in \eqref{eqn-3-71}.
 	Then
 	$$\big(\bar f'\big)^*H^0\big(\ol B',\,\Omega^1_{\ol B'}\big) \subseteq
 	\bigoplus_{m\,|\,i} H^0\big(\ol S,\,\Omega^1_{\ol S}\big)_{i},\quad\text{where $m$ is order of ${\rm Ker}(\iota)$.}$$
 \end{lemma}
 \begin{proof}
 	Let $\tau\in G$ be any generator of $G$.
 	Then $\tau^m$ is a generator of ${\rm Ker}(\iota)$ by construction.
 	Hence
 	$$\big(\bar f'\big)^*H^0\big(\ol B',\,\Omega^1_{\ol B'}\big) \subseteq
 	H^0\big(\ol S,\,\Omega^1_{\ol S}\big)^{\tau^m},$$
 	where
 	$$\begin{aligned}
 	H^0\big(\ol S,\,\Omega^1_{\ol S}\big)^{\tau^m}
 	\triangleq\,& \left\{\omega\in H^0\big(\ol S,\,\Omega^1_{\ol S}\big)~\big|~\big(\tau^m\big)^*\omega=\omega\right\}\\
 	=\,&\bigoplus_{m\,|\,i} H^0\big(\ol S,\,\Omega^1_{\ol S}\big)_{i}.
 	\end{aligned}$$
 	This completes the proof.
 \end{proof}

 \begin{corollary}\label{cor-3-8}
 	Assume that $H^0\big(\ol S,\,\Omega^1_{\ol S}\big)_{i_0}\neq 0$ for some $i_0\geq n/2$,
 	and let $\bar f':\,\ol S \to \ol B'$ be the fibration obtained in \autoref{cor-3-11}.
 	Assume moreover that $\gcd(i_0,n)=1$.
 	Then $G$ induces a faithful action on $\ol B'$, such that $\ol B'/G\cong \bbp^1$
 	and
 	\begin{equation}\label{eqn-3-43-1}
 	H^0\big(\ol S,\,\Omega^1_{\ol S}\big)_{i}
 	=\big(\bar f'\big)^*H^0\big(\ol B',\,\Omega^1_{\ol B'}\big)_{i},\quad\text{~for any $1\leq i \leq n-1$ with $\gcd(i,n)=1$}.
 	\end{equation}
 \end{corollary}
 \begin{proof}
 	Let $\iota:G\ra \Aut({\ol B}')$ be defined in \eqref{eqn-3-71}.
 	By \autoref{cor-3-11} and \autoref{lem-3-17},
 	one obtains that $\iota$ is an isomorphism, and $G$ acts faithfully on $\ol B'$.
 	Moreover, taking the quotients of the $G$-actions, we obtain the following commutative diagram.
 	$$\xymatrix{
 		\ol S \ar[rr]^-{\ol \Pi} \ar[d]_-{\bar f'}
 		&& \ol Y \ar[d]^-{\bar \varphi'}\\
 		\ol B' \ar[rr]^-{\pi} && \ol B'/G
 	}$$
 	Since $\ol Y$ is a ruled surface over $\ol B$, it follows that $\olb'/G\cong \bbp^1$.
 	It remains to show \eqref{eqn-3-43-1}.
 	
 	Consider the $\mathbb Q$-vector space $H^1\big(\ol S,\,\mathbb Q\big)$,
 	which admits a natural $G$-action.
 	Let
 	$$H^1\big(\ol S,\,\mathbb Q\big)\otimes \mathbb C
 	=H^1\big(\ol S,\,\mathbb C\big)=\bigoplus_{i=0}^{n-1} H^1\big(\ol S,\,\mathbb C\big)_i$$
 	be the eigenspace decomposition.
 	The morphism $\bar f'$ induces an inclusion
 	\begin{equation}\label{eqn-3-17}
 	\big(\bar f'\big)^*:~ H^1\big(\ol B',\,\mathbb Q\big)\otimes \mathbb C
 	=H^1\big(\ol B',\,\mathbb C\big)
 	\hookrightarrow H^1\big(\ol S,\,\mathbb C\big)=H^1\big(\ol S,\,\mathbb Q\big)\otimes \mathbb C,
 	\end{equation}
 	which is clearly compatible with the actions of $G$ on both sides.
 	By \eqref{eqn-3-42}, it follows that
 	$$\big(\bar f'\big)^* H^1\big(\ol B',\,\mathbb C\big)_i = H^1\big(\ol S,\,\mathbb C\big)_i,
 	\qquad\text{for $i\in \{i_0,p-i_0\}$}.$$
 	As a vector subspace of $H^1\big(\ol S,\,\mathbb C\big)$,
 	$\big(\bar f'\big)^*H^1\big(\ol B',\,\mathbb C\big)$ is defined over $\mathbb Q$.
 	Hence it follows from \eqref{eqn-3-17} that
 	\begin{equation}\label{eqn-3-18}
 	\big(\bar f'\big)^* H^1\big(\ol B',\,\mathbb C\big)
 	\supseteq \bigoplus_{1\leq i\leq n-1 \atop \gcd(i,n)=1} H^1\big(\ol S,\,\mathbb C\big)_i.
 	\end{equation}
 	Combining \eqref{eqn-3-17} with \eqref{eqn-3-18},
 	and taking the $(1,0)$-parts, we complete the proof of \eqref{eqn-3-43-1}.
 \end{proof}

 \subsection{Irregular family of superelliptic curves}\label{subsec-3-x}
 In this subsection,
 we consider irregular families of superelliptic curves,
 i.e., the relative irregularity of the family is positive.
 This will be used to obtain a better slope inequality for irregular families of superelliptic curves.
 The main purpose of this subsection is to prove
 \begin{proposition}\label{prop-3-5}
 	Let $\olf:\,\ols\to \olb$ be a family of semi-stable $n$-superelliptic curves
 	of genus $g\geq2$ as in \autoref{invariantstheorem}.
 	Assume that the relative irregularity $q_{\bar f}:=q(\ol S)-g(\ol B)>0$.
 	Then
 	\begin{equation}\label{eqn-3-53}
 	\left\{\begin{aligned}
 	\frac{4}{n}\cdot s_{2,n}'&\,\leq \sum_{\gamma,\ell}\frac{n\gamma(\alpha-\gamma)}{\ell^2(\alpha-1)}\cdot (s_{\gamma,\ell}+s_{\gamma,\ell}');\\[1mm]
 	\frac{1}{n}\cdot s_{2,n}'&\,\leq \sum_{\gamma,\,\ell} \left(\frac{n\gamma(\gamma-1)}{\ell^2(\alpha-1)(\alpha-2)}s_{\gamma,\ell}
 	+\frac{n(\alpha-\gamma)(\alpha-\gamma-1)}{\ell^2(\alpha-1)(\alpha-2)}s_{\gamma,\ell}'\right).
 	\end{aligned}\right.
 	\end{equation}
 \end{proposition}
 \begin{proof}
 	To prove \eqref{eqn-3-53}, we may assume that $s_{2,n}'>0$,
 	i.e., the singular fibers of $\bar f$ contains nodes with index $(2,n,1)$.
 	In particular, $n \,|\,(\alpha-2)$ by \eqref{eqn-3-44}, i.e., $\alpha_0=nk+1$ for some integer $k$ by \eqref{eqn-3-45}.
 	Hence for any general fiber $\ol F$ of $\bar f$,
 	\begin{equation}\label{eqn-3-206}
 	\text{the induced $n$-superelliptic cover $\pi:\,\ol F \to \bbp^1$ is totally ramified.}
 	\end{equation}
 	
 	Similar to the proof of \autoref{invariantstheorem}, we may assume that $G\cong \mathbb Z/n\mathbb Z$
 	admits an action on $\ol S$ such that $\ol \Pi:\,\ols \to \oly$ is branched
 	over $\alpha$ disjoint sections $\{D_i\}_{i=1}^{\alpha}$ and possibly over some nodes in fibers of $\bar\varphi$.
 	Let $\tilde\rho:\,\wt Y \to \ol Y$ be the resolution of the singularities of $\ol Y$,
 	and $\rho:\,\wt Y \to Y$ be a contraction to a $\bbp^1$-bundle over $\ol B$
 	such that the order of the singularities of $R_0=\sum\limits_{i=1}^{\alpha} D_{i,0}$ is at most $[\alpha/2]$,
 	and that $D_{\alpha,0}$ does not pass through any singularity of order equal to $\alpha/2$,
 	where $D_{i,0}\subseteq Y$ is the image $\rho\big(\tilde\rho^{-1}(D_i)\big)$.
 	$$\xymatrix{
 		\wt Y \ar[rrrr]^-{\rho_1} \ar[d]_-{\tilde\rho}\ar[drrrr]^-{\rho}
 		&&&& Y_1 \ar[d]^-{\rho_0}\\
 		\oly \ar[drr]_-{\bar\varphi} &&&& Y \ar[dll]^-{\varphi}\\
 		&&\ol B&&
 	}$$
 	Note that $\rho$ consists of a sequence of blowing-ups centered at singularities of $R_0$.
 	We reorder these blowing-ups as
 	$\rho:\,\wt Y \overset{\rho_1}\longrightarrow Y_1 \overset{\rho_0} \longrightarrow Y$
 	such that $\rho_0$ is the resolution of singularities of $\sum\limits_{i=1}^{\alpha-1} D_{i,0}$.
 	In other words, all nodes in the fibers of $\bar\varphi$ with index equal to $(2,n,1)$ (resp. not equal to $(2,n,1)$)
 	are created by $\rho_1$ (resp. $\rho_0$),
 	where the notion for the index of a node in the fibers of $\bar\varphi$ are introduced in the proof of \autoref{invariantstheorem}.
 	Let $D_{i,1}=\rho_1\big(\tilde\rho^{-1}(D_i)\big)$.
 	Then by \eqref{eqn-3-55},
 	$$\left\{\begin{aligned}
 	\sum_{i=1}^{\alpha}D_{i,1}^2&\,=\sum_{i=1}^{\alpha}\big(\tilde\rho^{-1}(D_i)\big)^2+4\xi_{2,n}'
 	=\sum_{i=1}^{\alpha}\big(\tilde\rho^{-1}(D_i)\big)^2+\frac{4}{n}\cdot s_{2,n}';\\
 	D_{\alpha,1}^2&\,=\big(\tilde\rho^{-1}(D_{\alpha})\big)^2+\xi_{2,n}'
 	=\big(\tilde\rho^{-1}(D_{\alpha})\big)^2+\frac{1}{n}\cdot s_{2,n}'
 	\end{aligned}\right.$$
 	Note also that the sections $\{D_i\}_{i=1}^{\alpha}$ do not pass through any singularity of $\ol Y$.
 	Hence $\{\tilde\rho^{-1}(D_i)\}_{i=1}^{\alpha}$ are still disjoint sections,
 	and by \eqref{eqn-3-49-b} one gets
 	$$\left\{\begin{aligned}
 	\sum_{i=1}^{\alpha}\big(\tilde\rho^{-1}(D_i)\big)^2&\,=\sum_{i=1}^{\alpha}D_{i}^2
 	=-\sum_{\gamma,\,\ell} \frac{n\gamma(\alpha-\gamma)}{\ell^2(\alpha-1)}\cdot (s_{\gamma,\ell}+s_{\gamma,\ell}');\\
 	\big(\tilde\rho^{-1}(D_{\alpha})\big)^2&\,=D_{\alpha}^2
 	=-\sum_{\gamma,\,\ell} \left(\frac{n\gamma(\gamma-1)}{\ell^2(\alpha-1)(\alpha-2)}s_{\gamma,\ell}
 	+\frac{n(\alpha-\gamma)(\alpha-\gamma-1)}{\ell^2(\alpha-1)(\alpha-2)}s_{\gamma,\ell}'\right).
 	\end{aligned}\right.$$
 	Therefore, it suffices to show that
 	$\sum\limits_{i=1}^{\alpha}D_{i,1}^2\leq 0$ and $D_{\alpha,1}^2\leq 0$, which follow directly from the next lemma.
 \end{proof}

 \begin{lemma}\label{lem-3-13}
 Keep the assumptions as in the proposition above. Then the divisor $R_1=\sum\limits_{i=1}^{\alpha}D_{i,1}$ is semi-negative definite.
 \end{lemma}
 \begin{proof}
 	Let $\psi:\,\wt S \to \ol S$ be the minimal blowing-up such that there exists a morphism
 	$\wt\Pi:\,\wt S \to \wt Y$ with the following commutative diagram.
 	$$\xymatrix{
 		\wt S \ar[rr]^-{\wt\Pi}\ar[d]_-{\psi}
 		&& \wt Y \ar[d]^-{\tilde \rho} \ar[rr]^-{\rho_1} && Y_1\\
 		\ol S\ar[rr]^-{\ol \Pi} && \ol Y
 	}$$
 	Let $\big(\rho_1\circ\wt\Pi\big)^{-1}(R_1)\subseteq \wt S$
 	be the total inverse image of $R_1$,
 	and $\wt R\subseteq \wt S$ be its support.
 	Then it suffices to show that the divisor $\wt R$ is semi-negative definite.
 	
 	By construction, the action of $G\cong \mathbb Z/n\mathbb Z$ on $\ol S$ lifts to $\wt S$.
 	Let $\Alb(\wt S)$ be the Albanese variety of $\wt S$,
 	and $\tau$ be any generator of $G$.
 	Then $\tilde f:=\bar f\circ \psi$ induces a map $\Alb(\tilde f):\, \Alb(\wt S)\to \Alb(\ol B)$
 	and $\tau$ has a natural action on $\Alb(\wt S)$.
 	Let
 	$$\Alb_0(\wt S)=\left\{x\in \Alb(\wt S)~\big|~\sum_{i=1}^{n}\tau^i(x)=e\right\},
 	\quad\text{where $e\in \Alb(\wt S)$ is the identity element}.$$
 	Then we claim that
 	\begin{claim}\label{claim-3-2}
 		$\Alb(\wt S)$ is isogenous to $\Alb_0(\wt S) \oplus \Alb(\tilde f)^{-1}\big(\Alb(\ol B)\big)$ and $\dim \Alb_0(\wt S)=q_{\bar f}$.
 	\end{claim}
 	\begin{proof}[Proof of \autoref{claim-3-2}]
 		Note that $\tau$ has a natural action on $H^0\big(\wt S, \Omega_{\wt S}^1\big)$ by pulling-back,
 		and the map $\tilde f$ induces an injection
 		$$\tilde f^*:~H^0\big(\ol B, \Omega_{\ol B}^1\big) \hookrightarrow H^0\big(\wt S, \Omega_{\wt S}^1\big),
 		~\text{~such that~}~
 		\tilde f^*H^0\big(\ol B, \Omega_{\ol B}^1\big)
 		=\left\{\omega \in H^0\big(\wt S, \Omega_{\wt S}^1\big)~\big|~\tau^*(\omega)=\omega\right\}.$$
 		To prove this claim, it suffices to show that
 		\begin{equation}\label{eqn-3-56}
 		H^0\big(\wt S, \Omega_{\wt S}^1\big)=\tilde f^*H^0\big(\ol B, \Omega_{\ol B}^1\big)\oplus W,
 		\quad\text{where~}W=\left\{\omega \in H^0\big(\wt S, \Omega_{\wt S}^1\big)~\big|~\sum_{i=1}^{n}(\tau^*\big)^i(\omega)=0\right\}.
 		\end{equation}
 		On the one hand, it is clear that $\tilde f^*H^0\big(\ol B, \Omega_{\ol B}^1\big)\cap W=0$.
 		On the other hand, for any $\omega\in H^0\big(\wt S, \Omega_{\wt S}^1\big)$,
 		let $\omega'=\frac{1}{n}\sum\limits_{i=1}^{n}(\tau^*\big)^i(\omega)$.
 		Then it is easy to verify that $\omega'\in \tilde f^*H^0\big(\ol B, \Omega_{\ol B}^1\big)$,
 		and that $\omega-\omega' \in W$.
 		Hence we obtain the decomposition $\omega=\omega'+(\omega-\omega')$ as required.
 	\end{proof}
 	
 	Denote by $J_0:\,\wt S \to \Alb_0(\wt S)$ the induced map.
 	Then we claim that
 	\begin{claim}\label{claim-3-3}
 		Let $\wt B_1 \subseteq \wt S$ be the strict inverse image of $R_1$.
 		Then $\wt B_1$ is contracted by $J_0$.
 	\end{claim}
 	\begin{proof}[Proof of \autoref{claim-3-3}]
 		Let $D\subseteq \wt B_1$ be any irreducible component, $\wt D$ its normalization,
 		$j:\,\wt D \to \wt S$ the induced map and $\vartheta=J_0\circ j:\,\wt D \to \Alb_0(\wt S)$ the composition.
 		We have to prove that $\vartheta(\wt D)$ is a point.
 		
 		We prove by contradiction. Assume that $\vartheta(\wt D)$ is one-dimensional.
 		Then the induced map
 		$$\vartheta^*:~H^0\left(\Alb_0(\wt S),\,\Omega_{\Alb_0(\wt S)}^1\right) \lra H^0\left(\wt D,\,\Omega_{\wt D}^1\right)$$
 		is non-zero. On the other hand, it is clear that $\vartheta^*$ factors through
 		$$H^0\left(\Alb_0(\wt S),\,\Omega_{\Alb_0(\wt S)}^1\right) \overset{J_0^*}\lra H^0\left(\wt S,\,\Omega_{\wt S}^1\right)
 		\overset{j^*}\lra H^0\left(\wt D,\,\Omega_{\wt D}^1\right).$$
 		By \eqref{eqn-3-56}, there is a decomposition of the form $H^0\left(\wt S,\,\Omega_{\wt S}^1\right)=\tilde{f}^*H^0({\ol B},\Omega^1_{\ol B})\oplus W$, and $W$ contains the image of $J_0^*$ according to the proof of \autoref{claim-3-2}.
 		To deduce a contradiction, it suffices to prove that the restriction
 		$$j^*\big|_{W}:~W \lra H^0\left(\wt D,\,\Omega_{\wt D}^1\right)$$
 		is zero.
 		
 		In fact, let $p\in D$ be an arbitrary smooth point of $D$.
 		Note that $\tau$ fixes $D$ due to \eqref{eqn-3-206}.
 		Hence locally around $p$, there exists local coordinate $(x,y)$ such that $C$ is defined by $y=0$
 		and the action of $\tau$ is given by $\tau(x,y)=(x,\epsilon y)$, where $\epsilon$ is a primitive $n$-th root of $1$.
 		For any $1$-form $$\omega=\zeta(x,y)dx+\eta(x,y)dy \in H^0\left(\wt S,\,\Omega_{\wt S}^1\right),$$
 		one has
 		$$\omega\in W ~\Longleftrightarrow~ \sum_{i=1}^n\zeta(x,\epsilon^iy)=0,~\,~\sum_{i=1}^n\epsilon^i\eta(x,\epsilon^iy)=0.$$
 		Hence if $\omega \in W$, $y$ should divide the function $\zeta(x,y)$, i.e.,
 		$\zeta(x,y)=y\cdot \tilde \zeta(x,y)$ for some function $\tilde \zeta(x,y)$.
 		In other words,
 		$j^*\omega\big|_{j^{-1}(p)}=0$ for any $\omega \in W$.
 		It follows that $j^*\omega=0$ for any $\omega \in W$ since $p$ is arbitrary.
 		This completes the proof.
 	\end{proof}
 	
 	Come back to the proof of \autoref{lem-3-13}.
 	By construction, the divisor $\wt R$ consists of $\wt B_1$ plus some rational curves.
 	According to \autoref{claim-3-3}, $\wt R$ is contracted by the map $J_0:\,\wt S \to \Alb_0(\wt S)$.
 	Note that the image $J_0(\wt S)$
 	generates the abelian variety $\Alb_0(\wt S)$ with $\dim \Alb_0(\wt S)=q_{\bar f}$ by \autoref{claim-3-2}.
 	Since $q_{\bar f}>0$, it follows that $J_0(\wt S)$ is not a point, i.e., $\dim J_0(\wt S)\geq 1$.
 	Therefore by the Hodge index theorem, $\wt R$ is semi-negative definite,
 	and hence $R_1$ is also semi-negative definite.
 	This completes the proof of \autoref{lem-3-13}.
 \end{proof}

 %
 %
 %

 \subsection{Proof of \autoref{prop-3-11}}\label{sec-pf-3-11}
 Since $g\geq n$, it follows that $\alpha\geq 5$ by the Riemann-Hurwitz formula \eqref{eqn-3-9}.
 According to \autoref{lem-3-23}, the family $\bar f$ admits no hyperelliptic fiber with compact Jacobian.
 Hence the Arakelov type equality \eqref{eqn-3-191} holds for $E^{1,0}_{\olb}$.
 Therefore, our conclusion follows from the next lemma.
 \begin{lemma}\label{prop-3-2}
 	Let $\olf: \ols\to \olb$ be a family of semi-stable $n$-superelliptic curves
 	as in \autoref{invariantstheorem}.
 	and let $\Upsilon_{nc} \to \Delta_{nc}$ be the singular fibers with non-compact Jacobians.
 	
 	%
 	%
 	
 	{\rm(i)}. 
 	If $\Delta_{nc}=\emptyset$ and $g\geq n$, then
 	\begin{equation}\label{eqn-3-19}
 	\deg \olf_*\omega_{\ols/\olb}\leq
 	\frac{2g-2}{\lambda_{n,c}}\cdot\deg\Omega^1_{\ol B},
 	\end{equation}
 	where $\lambda_{n,c}$ is defined in \eqref{eqn-3-30}.
 	
 	{\rm(ii)}.
 	Assume that $\Delta_{nc}\not=\emptyset$, $g\geq 4$ and $q_{\bar f}>0$.
 	If either $n=3$ or $4$, then
 	\begin{equation}\label{eqn-3-156}
 	\deg \olf_*\omega_{\ols/\olb} < \frac{2g-2}{\lambda_{n,nc}}\cdot\deg\Omega^1_{\ol B}(\log\Delta_{nc}),
 	\end{equation}
 	where $\lambda_{3,nc}$ and $\lambda_{4,nc}$
 	are defined in \eqref{eqn-3-154} and \eqref{eqn-3-155} respectively.
 \end{lemma}

 The above lemma follows directly from the Miyaoka-Yau type inequality \cite[Theorem\,4.1]{lu-zuo-14} together the following improved
 slope inequalities for a family of semi-stable superelliptic curves.

 \begin{lemma}\label{prop-3-4}
 	Let $\olf: \ols\to \olb$ be a family of semi-stable $n$-superelliptic curves
 	as before.
 	
 	{\rm(i)}. If $\Delta_{nc}=\emptyset$ and $g\geq n$, then
 	\begin{equation}\label{eqn-3-57}
 	\omega_{\ols/\olb}^2\geq \lambda_{n,c} \cdot \deg \olf_*\omega_{\ols/\olb}+2\delta_1(\bar f)+3\delta_h(\bar f),
 	\end{equation}
 	where $\lambda_{n,c}$ is defined in \eqref{eqn-3-30}.
 	
 	{\rm(ii)}. Assume that $\Delta_{nc}\not=\emptyset$, $g\geq 4$ and $q_{\bar f}>0$.
 	If either $n=3$ or $4$, then
 	\begin{equation}\label{eqn-3-132}
 	\omega_{\ols/\olb}^2 \geq \lambda_{n,nc} \cdot \deg \olf_*\omega_{\ols/\olb}+2\delta_1(\bar f)+3\delta_h(\bar f),
 	\end{equation}
 	where $\lambda_{3,nc}$ and $\lambda_{4,nc}$ are defined in \eqref{eqn-3-154} and \eqref{eqn-3-155} respectively.
 \end{lemma}

 To complete the proof of \autoref{prop-3-11}, it suffices to prove the above improved slope inequalities.
 \begin{proof}[Proof of \autoref{prop-3-4}]
 	We use the notations introduced in \autoref{subsec-3-2}.
 	
 	(i).
 	Since $\Delta_{nc}=\emptyset$, one gets that $\delta_0(\bar f)=0$,
 	from which and \autoref{rem-3-1}, it follows that $s_{\gamma,\ell}=s_{\gamma,\ell}'=0$ for any $\ell>1$,
 	and
 	$$s_{\gamma,1}=0,~\text{~if~}\gcd(\gamma,n)\neq 1, \quad\text{and}\quad
 	s_{\gamma,1}'=0,~\text{~if~}\gcd(\alpha-\gamma,n)\neq 1.$$
 	Moreover,
 	$$\begin{aligned}
 	\delta_1(\bar f)&\,=\left\{\begin{aligned}
 	&s_{2,1},&\quad&\text{if~}n=3;\\
 	&s_{2,1}', &&\text{if~}n=4\text{~and~}\alpha=4k+3\text{~with~}k\geq 1;\\
 	&0, &&\text{otherwise.}
 	\end{aligned}\right.\\
 	\delta_h(\bar f)&\,=\sum_{\gamma=2}^{[\alpha/2]}(s_{\gamma,1}+s_{\gamma,1}')-\delta_1(\bar f).
 	\end{aligned}$$
 	Hence by \autoref{invariantstheorem} we obtain
 	$$\begin{aligned}
 	&\omega_{\ols/\olb}^2- \lambda_{n,c} \cdot \deg \olf_*\omega_{\ols/\olb}\\
 	=~&\left\{\begin{aligned}
 	&\sum_{\gamma=2}^{[\alpha/2]}\Big(z_{\gamma}-\frac{\lambda_{n,c}}{12}(z_{\gamma}+1)\Big) s_{\gamma,1}, &~&\text{if~}n\,|\,\alpha_0,\\
 	&\sum_{\gamma=2}^{[\alpha/2]}\Big(z_{\gamma}-\frac{\lambda_{n,c}}{12}(z_{\gamma}+1)\Big) s_{\gamma,1}
 	+\sum_{\gamma=2}^{[\alpha/2]}\Big(z_{\gamma}'-\frac{\lambda_{n,c}}{12}(z_{\gamma}'+1)\Big) s_{\gamma,1}', &~&\text{if~}n{\not|}~\alpha_0,
 	\end{aligned}\right.\\
 	\geq~& 2\delta_1(\bar f)+3\delta_h(\bar f),
 	\end{aligned}$$
 	where
 	$$\begin{aligned}
 	z_{\gamma}&\,=\frac{(n^2-1)\gamma(\alpha-\gamma)}{\alpha-1}-n^2
 	 -\frac{(n^2-d^2)\gamma(\gamma-1)}{d^2(\alpha-1)(\alpha-2)}~\text{~with~}d=\left\{\begin{aligned}
 	&n,&&\text{if~}n\,|\,\alpha_0,\\
 	&r_{\infty},&&\text{if~}n{\not|}~\alpha_0,
 	\end{aligned}\right.\\
 	z_{\gamma}'&\,=\frac{(n^2-1)\gamma(\alpha-\gamma)}{\alpha-1}-n^2
 	-\frac{(n^2-d^2)(\alpha-\gamma)(\alpha-\gamma-1)}{d^2(\alpha-1)(\alpha-2)}.
 	\end{aligned}$$
 	This completes the proof.
 	
 	(ii).
 	We prove here only for the case when $n=3$;
 	the case when $n=4$ can be proven similarly.
 	Since $n=3$, it follows from \eqref{eqn-3-45} that $\alpha=3k$ or $3k+2$ for some integer $k$.
 	According to \eqref{eqn-3-54} and \eqref{eqn-3-44},
 	when $\alpha=3k$, we have
 	$$s_{\gamma,\ell}'=0; \qquad s_{\gamma,1}=0,~ \text{~if~}3\,|\,\gamma;
 	\qquad s_{\gamma,3}=0,~ \text{~if~}3{\not|}~\gamma;
 	$$
 	and when $3k+2$, we have
 	\begin{equation*}\left\{\begin{aligned}
 	&s_{\gamma,1}=0,&~& \text{if~}3\,|\,\gamma;  &\qquad&s_{\gamma,1}'=0,&~& \text{if~}3\,|\,(\gamma+1);\\
 	&s_{\gamma,3}=0,&& \text{if~}3{\not|}~\gamma;  &&s_{\gamma,3}'=0,&& \text{if~}3{\not|}~(\gamma+1).
 	\end{aligned}\right.
 	\end{equation*}
 	Hence
 	\begin{equation}\label{eqn-3-61}
 	\delta_0(\bar f)=\sum_{\gamma=2}^{[\alpha/2]}(s_{\gamma,3}+s_{\gamma,3}'),\quad
 	\delta_1(\bar f)=s_{2,1},\quad \delta_h(\bar f)=\sum_{\gamma=3}^{[\alpha/2]}(s_{\gamma,1}+s_{\gamma,1}'),
 	~\text{~and~}~ s_{2,3}=0.
 	\end{equation}
 	
 	If $\alpha=3k$, then $s_{2,3}'=0$;
 	moreover, since $g\geq 4$,  one has $k\geq 2$ by \eqref{eqn-3-9}.
 	Hence by \autoref{invariantstheorem} one obtains that
 	(where $\lambda_1=\frac{15\alpha-63}{2(\alpha-3)}$)
 	$$\begin{aligned}
 	\omega_{\ols/\olb}^2-\lambda_1\cdot \deg \olf_*\omega_{\ols/\olb}
 	=&~\sum_{\gamma=2}^{[\alpha/2]}
 	\left(\Big(\frac{8\gamma(\alpha-\gamma)}{\alpha-1}-9\Big)
 	-\frac{\lambda_1}{12}\cdot \Big(\frac{8\gamma(\alpha-\gamma)}{\alpha-1}-8\Big)\right)(s_{\gamma,1}+s_{\gamma,1}')\\
 	&+\sum_{\gamma=3}^{[\alpha/2]}
 	\left(\Big(\frac{8\gamma(\alpha-\gamma)}{9(\alpha-1)}-1\Big)
 	 -\frac{\lambda_1}{12}\cdot\frac{8\gamma(\alpha-\gamma)}{9(\alpha-1)}\right)(s_{\gamma,3}+s_{\gamma,3}')\\
 	\geq&\,2\delta_1(\bar f)+3\delta_h(\bar f).
 	\end{aligned}$$
 	
 	If $\alpha=3k+2$,
 	then one again has $k\geq 2$.
 	If moveover $s_{2,3}'=0$, then one can show again that
 	$$\omega_{\ols/\olb}^2\geq \frac{15\alpha-63}{2(\alpha-3)}\cdot \deg \olf_*\omega_{\ols/\olb}
 	+2\delta_1(\bar f)+3\delta_h(\bar f).$$
 	Hence we may assume that $s_{2,3}'\neq 0$.
 	In this case, we have (where $\lambda_2=\frac{6\alpha-18}{\alpha-2}$)
 	\begin{equation}\label{eqn-3-59}
 	\begin{aligned}
 	&~\omega_{\ols/\olb}^2-\lambda_2\cdot \deg \olf_*\omega_{\ols/\olb}\\
 	=&~\sum_{\gamma=2}^{[\alpha/2]}
 	\left(\Big(\frac{8\gamma(\alpha-\gamma)}{\alpha-1}-9\Big)
 	-\frac{\lambda_2}{12}\cdot \Big(\frac{8\gamma(\alpha-\gamma)}{\alpha-1}-8\Big)\right)(s_{\gamma,1}+s_{\gamma,1}')\\
 	&-\frac19s_{2,3}'+\sum_{\gamma=3}^{[\alpha/2]}
 	\left(\Big(\frac{8\gamma(\alpha-\gamma)}{9(\alpha-1)}-1\Big)
 	 -\frac{\lambda_2}{12}\cdot\frac{8\gamma(\alpha-\gamma)}{9(\alpha-1)}\right)(s_{\gamma,3}+s_{\gamma,3}').
 	\end{aligned}
 	\end{equation}
 	Note that $s_{2,1}'=0$ in this case.
 	Hence by \eqref{eqn-3-53} for $n=3$ one obtains that
 	\begin{equation*}
 	\left\{\begin{aligned}
 	s_{2,3}'&\,\leq \sum_{\gamma=2}^{[\alpha/2]}\frac{9\gamma(\alpha-\gamma)}{2\alpha} \Big(s_{\gamma,1}+\frac{s_{\gamma,3}}{9}\Big)
 	+\sum_{\gamma=3}^{[\alpha/2]}\frac{9\gamma(\alpha-\gamma)}{2\alpha} \Big(s_{\gamma,1}'+\frac{s_{\gamma,3}'}{9}\Big);\\[1mm]
 	s_{2,3}'&\,\leq \sum_{\gamma=2}^{[\alpha/2]} \frac{9\gamma(\gamma-1)}{2(\alpha-2)}\Big(s_{\gamma,1}+\frac{s_{\gamma,3}}{9}\Big)
 	 +\sum_{\gamma=3}^{[\alpha/2]}\frac{9(\alpha-\gamma)(\alpha-\gamma-1)}{2(\alpha-2)}\Big(s_{\gamma,1}'+\frac{s_{\gamma,3}'}{9}\Big).
 	\end{aligned}\right.
 	\end{equation*}
 	Let $x=\frac{\alpha(\alpha-7)}{(\alpha-1)(\alpha-4)}$.
 	Then $0< x<1$ since $\alpha\geq 8$.
 	Hence
 	\begin{equation}\label{eqn-3-60}
 	\begin{aligned}
 	s_{2,3}'\leq&~ \sum_{\gamma=2}^{[\alpha/2]}\left(\frac{9\gamma(\alpha-\gamma)}{2\alpha}\cdot x
 	+\frac{9\gamma(\gamma-1)}{2(\alpha-2)}\cdot (1-x)\right) \Big(s_{\gamma,1}+\frac{s_{\gamma,3}}{9}\Big)\\
 	& +\sum_{\gamma=3}^{[\alpha/2]}\left(\frac{9\gamma(\alpha-\gamma)}{2\alpha}\cdot x
 	+\frac{9(\alpha-\gamma)(\alpha-\gamma-1)}{2(\alpha-2)}\cdot (1-x)\right)\Big(s_{\gamma,1}'+\frac{s_{\gamma,3}'}{9}\Big)\\
 	 =&~\sum_{\gamma=2}^{[\alpha/2]}\frac{9\gamma\big(\alpha^2-(\gamma+8)\alpha+10\gamma+4\big)}{2(\alpha-2)(\alpha-4)} \Big(s_{\gamma,1}+\frac{s_{\gamma,3}}{9}\Big)\\
 	& +\sum_{\gamma=3}^{[\alpha/2]}\frac{9(\alpha-\gamma)\big((\gamma+2)\alpha-10\gamma+4\big)}{2(\alpha-2)(\alpha-4)}\Big(s_{\gamma,1}'+\frac{s_{\gamma,3}'}{9}\Big)
 	\end{aligned}
 	\end{equation}
 	Combining \eqref{eqn-3-60} together with \eqref{eqn-3-59} and \eqref{eqn-3-61}, we obtain that
 	if $s_{2,3}'> 0$, then
 	$$\omega_{\ols/\olb}^2\geq \frac{6\alpha-18}{\alpha-2}\cdot \deg \olf_*\omega_{\ols/\olb}
 	+2\delta_1(\bar f)+3\delta_h(\bar f).$$
 	This completes the proof of \eqref{eqn-3-132} for the case when $n=3$. 	
 \end{proof}


 \subsection{Proof of \autoref{prop-3-13}}\label{sec-pf-3-13}
 It follows directly from \autoref{cor-3-3} and \autoref{cor-3-11}.
 \qed

 \subsection{Proof of \autoref{prop-3-6}}\label{sec-technical}
 In this subsection, we prove \autoref{prop-3-6},
 which is used in the induction process in the proof of \autoref{thm-curve}.
 We use the notations introduced at the beginning of \autoref{sec-non-general}.
 Before entering the proof, let's first do some preparations.


 \begin{lemma}\label{lem-3-18}
 	Let $\left(E^{1,0}_{\olb}\oplus E^{0,1}_{\olb},~\theta_{\olb}\right)$
 	and $\left(\wt E^{1,0}_{\olb}\oplus \wt E^{0,1}_{\olb},~\wt \theta_{\olb}\right)$
 	be the corresponding logarithmic Higgs bundles associated to the families $\bar f$ and $\bar f_1$ respectively.
 	Then the Galois group $G\cong \mathbb Z/n\mathbb Z$ (resp. $G_1\cong \mathbb Z/n_1\mathbb Z$)
 	admits a natural action on the logarithmic Higgs bundle associated to $\bar f$ (resp. $\bar f_1$),
 	and the eigenspaces satisfy that (where $m_1=\frac{n}{n_1}$)
 	\begin{equation}\label{eqn-3-72}
 	\left(\wt E^{1,0}_{\olb}\oplus \wt E^{0,1}_{\olb},~\wt \theta_{\olb}\right)_i
 	\cong \left(E^{1,0}_{\olb}\oplus E^{0,1}_{\olb},~\theta_{\olb}\right)_{im_1},
 	\qquad \forall~1\leq i\leq n_1-1,
 	\end{equation}
 \end{lemma}
 \begin{proof}
 	Let $\tau\in G$ be any generator, $H \leqslant G$ be the subgroup generated by $\tau^{m_1}$, and $G_1=G/H$.
 	Then
 	$$\bigoplus_{i=1}^{n_1-1}\left(E^{1,0}_{\olb}\oplus E^{0,1}_{\olb},~\theta_{\olb}\right)_{im_1}
 	=\left(E^{1,0}_{\olb}\oplus E^{0,1}_{\olb},~\theta_{\olb}\right)^H.$$
 	On the other hand,
 	by construction, $\ols_1$ is birational to the quotient $\ols/H$.
 	It follows that
 	$$\left(\wt E^{1,0}_{\olb}\oplus \wt E^{0,1}_{\olb},~\wt \theta_{\olb}\right)
 	\cong \left(E^{1,0}_{\olb}\oplus E^{0,1}_{\olb},~\theta_{\olb}\right)^H.$$
 	Hence
 	$$\left(\wt E^{1,0}_{\olb}\oplus \wt E^{0,1}_{\olb},~\wt \theta_{\olb}\right)
 	\cong \bigoplus_{i=1}^{n_1-1}\left(E^{1,0}_{\olb}\oplus E^{0,1}_{\olb},~\theta_{\olb}\right)_{im_1}.$$
 	Note that the group $G_1$ acts naturally on both sides,
 	and that the above isomorphism is clearly equivariant with respect to the actions of $G_1$.
 	This proves \eqref{eqn-3-72}.
 \end{proof}

 \begin{lemma}\label{lem-3-19}
 	Notations as above.
 	Assume that $C$ is a Shimura curve but $\rho_{n,n_1}(C)$ is not a Shimura curve.
 	Then
 	\begin{enumerate}
 		\item[(i)] $n_1\,|\,\alpha_0$ and $n\not |~\alpha_0$;
 		
 		\item[(ii)] up to a suitable finite \'etale base change, $\bar f:\,\ols\to \olb$ is birational to the resolution of
 		a degree-$n$ cyclic cover of a $\mathbb P^1$-bundle $\varphi:\,Y \to \olb$
 		branched exactly over $\alpha_0+1$ disjoint sections,
 		denoted as $D_1,\cdots,D_{\alpha_0+1}$, such that the local monodromy is 1 around $D_1,\cdots,D_{\alpha_0}$, and equals $a_\infty=n(1-\{\frac{\alpha_0}{n}\})$ around $D_{\alpha_0+1}$;
 		
 		\item[(iii)] up to a suitable finite \'etale base change,
 		the Higgs subbundle
 		$$\bigoplus_{i=1}^{n_1-1}\left(E^{1,0}_{\olb}\oplus E^{0,1}_{\olb},~\theta_{\olb}\right)_{im_1}$$
 		is a trivial Higgs subbundle, where $m_1=\frac{n}{n_1}$;
 		
 		\item[(iv)] $F_{\olb,i}^{1,0}=0$ for any $i\geq m_1$ and $m_1{\not|}~i$;
 		
 		\item[(v)] if moreover
 		the general fiber of $\bar f$ admits a unique $n$-superelliptic automorphism group $G$
 		and there eixsts a generator of $G$
 		commuting with any automorphism of the general fiber,
 		then
 		\begin{equation}\label{eqn-3-153}
 		\rank A_{\olb}^{1,0} \leq \left\{\begin{aligned}
 		 &\frac{\left(n^2-\big(\frac{n}{r_{\infty}}\big)^2\right)(\alpha_0-2)}{6\big((n-1)\alpha_0-2n\big)}, &\quad&\text{if~}\ell_1=1;\\[1mm]
 		 &\frac{n^2-\big(\frac{n}{r_{\infty}}\big)^2}{6n}+\frac{(\ell_1^2-1)\alpha_0}{6n(\alpha_0-2)}, &&\text{if~}\ell_1>1.
 		\end{aligned}\right.
 		\end{equation}
 		where $r_{\infty}=\frac{n}{\gcd(n,\alpha_0)}$ and $\ell_1=\gcd(\alpha_0-1,n)$.
 	\end{enumerate}
 \end{lemma}
 \begin{proof}
 	As above, after a possible base change, we assume that
 	the group $G$ acts on $\ols$, and let $\bar f_1:\,\ols_1\to \olb$
 	be the new family of semi-stable $n_1$-superelliptic curves.
 	Assume that $\bar f$ is given by $y^n=F_t(x)$, where $t$ is the parameter.
 	Then both of the two families $\bar f$ and $\bar f_1$, up to base change,
 	are birational to the resolution of cyclic covers of a $\mathbb P^1$-bundle
 	$\varphi:\,Y\to \ol B$ branched over the zero locus of $F_t(x)$ and possibly over the section at the infinity
 	(depending on whether $n_1$ and $n$ divide $\alpha_0$ or not respectively).
 	In other words, choose suitable birational models,
 	we have the following diagram.
 	$$\xymatrix{
 		\wt S \ar[rr]^-{\Pi_{n,n_1}} \ar[dr]_-{\Pi} && \wt S_1\ar[dl]^-{\Pi_1}\\
 		&Y&}$$
 	Here $\Pi_{n,n_1}$ is induced by the rational map $\ol\Pi_{n,n_1}:\,\ols \dashrightarrow \ols_1$,
 	and the difference between the branch loci of $\Pi$ and of $\Pi_1$
 	is at most the section at the infinity; see \eqref{eqn-3-46} and \eqref{eqn-3-96}.
 	
 	Since $\rho_{n,n_1}(C)$ is not a Shimura curve,
 	it follows from \autoref{lem-3-15} that $\rho_{n,n_1}(C)$ is a (special) point in $\mathcal{TS}_{g_1,n_1}$.
 	In other words, the semi-stable family $\bar f_1$ is isotrivial by construction.
 	Being semi-stable, the family $\bar f_1$ is actually a smooth family.
 	In other words, up to some elementary transformations of the $\bbp^1$-bundle $Y$,
 	the branch locus $R_1$ of the cover $\Pi_1$
 	is smooth and the restricted map $\varphi|_{R_1}:\,R_1\to \olb$ is \'etale;
 	here we recall that an {\it elementary transformation} of a $\bbp^1$-bundle $X$
 	is a new $\bbp^1$-bundle $X'$ obtained by first blowing up some point $x\in X$
 	and then contracting the strict inverse image of the fiber of $X$ through $x$.
 	We should remark that an elementary transformation is a birational operation,
 	but it may transfer the section at the infinity to somewhere else.
 	
 	\vspace{1mm}
 	(i).
 	The first statement is clear from the above arguments;
 	otherwise,
 	the branch loci of $\Pi$ and $\Pi_1$ are the same by the above arguments together with \eqref{eqn-3-46} and \eqref{eqn-3-96}.
 	Since $\bar f_1$ is isotrivial,
 	it follows that $\bar f$ is also an iso-trivial family, which is a contradiction since $C$ is a Shimura curve.
 	
 	(ii).
 	By (i) and its proof above,
 	we obtain that the branch locus $R$ of $\Pi$ equals the branch locus $R_1$ of $\Pi_1$
 	plus one another section.
 	On the other hand, as the restricted map $\varphi|_{R_1}:\,R_1\to \olb$ is \'etale,
 	it follows that after a possible suitable finite \'etale base change,
 	$R_1$ becomes $\alpha_0$ disjoint sections.
 	Since $R$ equals $R_1$ plus one another section $D_{\alpha_0+1}$,
 	and the local monodromy around each component in $R_1$ (resp. the component $D_{\alpha_0+1}$) is $1$
 	\big(resp. $a_{\infty}=n\left(\big[\frac{\alpha_0}{n}\big]+1\right)-\alpha_0$\big) by construction,
 	this proves the second statement.
 	
 	(iii).
 	Since $\bar f_1$ is isotrivial,
 	the Higgs bundle associated to $\bar f_1$ is trivial after a suitable unramified base change.
 	Hence this statement follows directly from \eqref{eqn-3-72}.
 	
 	(iv).
 	Note that the rank of the subbundle $F_{\olb,i}^{1,0}$
 	does not decrease after base change.
 	Hence it suffices to prove the statement after any finite base change. 
 	By the above arguments, it follows that $\ol S_1\cong \olb \times \ol F_1$ up to some finite \'etale base change,
 	where $\ol F_1$ is a general fiber of $\bar f_1:\,\ols_1\to\olb$.
 	Composed the rational map $\ol\Pi_{n,n_1}:\,\ols \dashrightarrow \ols_1$ with the second projection $\ols_1 \to \ol F_1$,
 	we obtain a rational map $\ol{pf}_1:\,\ols \dashrightarrow \ol F_1$.
 	Since $g(\ol F_1)>0$ by construction, it follows that $\ol{pf}_1$ is in fact a morphism.
 	By (iii) together with \cite[Theorem\,3.1]{fujita-78} (see also \eqref{eqn-3-32}), we have
 	$$\rank \bar f_*\Omega^1_{\ols/\olb}(\log \Upsilon)_i=\dim H^0\big(\ol S,\,\Omega^1_{\ol S}\big)_{i},
 	\quad\forall~m_1\,|\,i ~\&~ i\neq 0.$$
 	Together with \eqref{eqn-3-72}, we obtain that
 	\begin{equation}\label{eqn-3-81}
 	\big(\ol{pf}_1\big)^*H^0\big(\ol F_1,\,\Omega^1_{\ol F_1}\big)
 	=\bigoplus_{m_1\,|\,i ~\&~ i\neq 0} H^0\big(\ol S,\,\Omega^1_{\ol S}\big)_{i}.
 	\end{equation}
 	Assume that the flat part $F_{\olb,i_0}^{1,0}\neq 0$
 	for some $i_0\geq m_1$ and $m_1{\not|}~i_0$.
 	By (i), one has $\alpha_0=kn_1$ with $k\geq 1$.
 	
 	If $k\geq 2$, i.e., $\alpha_0\geq 2n_1$,
 	then $F_{\olb,n-m_1}^{1,0}\neq 0$ by \eqref{eqn-3-72}.
 	Hence up to base change,
 	we may assume that $F_{\olb,i_0}^{1,0}$ is trivial by \autoref{cor-3-3}.
 	This is equivalent to saying that
 	$H^0\big(\ol S,\,\Omega^1_{\ol S}\big)_{i_0}\neq 0$ by \eqref{eqn-3-32}.
 	By \autoref{cor-3-11}, there exists a unique fibration $\bar f':\,\ols \to \olb'$ such that \eqref{eqn-3-42} holds.
 	Together with \eqref{eqn-3-81}, it follows that $\bar f'$ is the same as $\ol{pf}_1$ obtained above.
 	This is a contradiction by \eqref{eqn-3-42} and \eqref{eqn-3-81}.
 	
 	If $k=1$, i.e., $\alpha_0=n_1$, then $F_{\olb,n-2m_1}^{1,0}\neq 0$ by \eqref{eqn-3-72}.
 	Moreover, if we let $\wt \Gamma$ be the general fiber of $\tilde\varphi$ as in \autoref{subsec-3-4},
 	then
 	$$\begin{aligned}
 	&\wt \Gamma \cdot \bigg(\omega_{\wt Y}(\wt R) \otimes
 	\left({\mathcal L^{(n-2m_1)}}^{-1} \otimes {\mathcal L^{(i_0)}}^{-1}\right)\bigg)\\
 	=~&-2+(\alpha_0+1)-\frac{(n-2m_1)\alpha_0}{n}
 	-\left(\frac{i_0\cdot n}{n}-\left[\frac{i_0(n-\alpha_0)}{n}\right]\right)\\
 	 =~&1-\frac{i_0}{m_1}-\left(\frac{i_0(n-\alpha_0)}{n}-\left[\frac{i_0(n-\alpha_0)}{n}\right]\right)<0.
 	\end{aligned}
 	$$
 	Hence similarly as above, one derive a contradiction.
 	This proves (iv).
 	
 	(v).
 	In this case, as remarked in \autoref{rems-3-1}\,(iii),
 	the group $G\cong \mathbb Z/n\mathbb Z$ can be extended to $\ols$ without any base change.
 	In other words, the above properties (i)-(iii) hold for $\bar f$, and simultaneously
 	one has the Arakelov type equality as in \eqref{eqn-3-191} for $\bar f$.
 	
 	Let $s_{\gamma,\ell}=s_{\gamma,\ell}(\bar f)$ and $s_{\gamma,\ell}'=s_{\gamma,\ell}'(\bar f)$
 	be the local invariants of $\bar f$ introduced in \autoref{def-3-4}.
 	We first claim that
 	\begin{equation}\label{eqn-3-74}
 	\text{$s_{\gamma,\ell}=0$ for any $(\gamma,\ell)$, \quad and $s_{\gamma,\ell}'=0$ unless $(\gamma,\ell)=(2,\ell_1)$.}
 	\end{equation}
 	In fact, by (ii), the branch locus of $\Pi$
 	consists of $\alpha_0$ disjoint sections $\sum\limits_{i=1}^{\alpha_0}D_i$
 	plus one another section $D_{\alpha_0+1}$.
 	Moreover, the local monodromy around $D_i$ is 1 for $1\leq i\leq \alpha_0$, and that around $D_{\alpha_0+1}$ is $a_{\infty}$.
 	Let $\xi_{\gamma,\ell}$ (resp. $\xi_{\gamma,\ell}'$)
 	be the number of the nodes in fibers of $\bar \varphi$
 	with index $(\gamma,\ell,0)$ (resp. $(\gamma,\ell,1)$), counted according to their multiplicities,
 	where $\bar\varphi:\,\oly\triangleq\ols/G\to \olb$ is the induced quotient family as in the proof of \autoref{invariantstheorem}.
 	Then it is easy to see that
 	\begin{equation*}
 	\text{$\xi_{\gamma,\ell}=0$ for any $(\gamma,\ell)$;\quad
 		and $\xi_{\gamma,\ell}'=0$ unless $\gamma=2$}.
 	\end{equation*}
 	Moreover, $\xi_{2,\ell}=0$ unless $\ell=\ell_1$ by \eqref{eqn-3-44}.
 	Hence \eqref{eqn-3-74} follows from \eqref{eqn-3-55}.
 	
 	\vspace{1mm}
 	According to \eqref{eqn-3-74} together with \eqref{chi_fformula} and \eqref{eqn-3-55}, we get
 	\begin{equation}\label{eqn-3-82}
 	\begin{aligned}
 	\deg E_{\ol B}^{1,0}=\deg \bar f_*\omega_{\ols/\olb}&\,=
 	 \left(\frac{\left(n^2-\big(\frac{n}{r_{\infty}}\big)^2\right)(\alpha_0-2)}{\alpha_0}+(\ell_1^2-1)\right)
 	\cdot \frac{s_{2,\ell_1}'}{12\ell_1^2}\\
 	 &\,=\left(\frac{\left(n^2-\big(\frac{n}{r_{\infty}}\big)^2\right)(\alpha_0-2)}{\alpha_0}+(\ell_1^2-1)\right)
 	\cdot \frac{\xi_{2,\ell_1}'}{12n}.
 	\end{aligned}
 	\end{equation}
 	
 	On the other hand, by Step 1 we know that
 	  the $\bbp^1$-bundle $Y$ contains $\alpha_0\geq 4$ disjoint sections up to a suitable unramified base change, hence
 	it follows that $Y\cong \olb \times \bbp^1$.
 	Let $pr_2:\,Y\cong \olb \times \bbp^1 \to \bbp^1$ be the projection, and
 	$$\psi:~D_{\alpha_0+1} \hookrightarrow Y \overset{pr_2}\lra \bbp^1$$
 	be the induced map on $D_{\alpha_0+1}$.
 	Note that $D_i$ is contracted by $pr_2$ for $1\leq i\leq \alpha_0$.
 	It follows that $\psi$ is surjective. Let $R_{\psi}\subseteq D_{\alpha_0+1}$
 	be the ramified divisor of $\psi$. Then by Hurwitz formula,
 	it follows that
 	$$\deg(R_{\psi})=2g(D_{\alpha_0+1})-2-2\deg(\psi)=2g(\olb)-2+2\deg(\psi).$$
 	Let $\nu_i=|D_i\cap D_{\alpha_0+1}|$ be the number of points contained in $D_i\cap D_{\alpha_0+1}$ for $1\leq i\leq \alpha_0$,
 	and $|\Delta|$ be the number of the singular fibers of $\bar f$.
 	Then
 	\begin{equation}\label{eqn-3-76}
 	\xi_{2,1}'=\sum\limits_{i=1}^{\alpha_0} D_i\cdot D_{\alpha_0+1}=\alpha_0 \deg(\psi),
 	\end{equation}
 	and
 	$$\sum_{i=1}^{\alpha_0} \nu_i=|\Delta|\leq s_{2,1}'=\frac{\ell_1^2}{n}\cdot \xi_{2,1}'
 	=\frac{\ell_1^2}{n}\cdot \alpha_0\deg(\psi).$$
 	Hence
 	\begin{equation}\label{eqn-3-85}
 	\begin{aligned}
 	2g(\olb)-2+2\deg(\psi)=\deg(R_{\psi})&\,\geq \sum_{i=1}^{\alpha_0}(\deg(\psi)-\nu_i)\\
 	&\,=\alpha_0\cdot\deg(\psi)-\sum_{i=1}^{\alpha_0}\nu_i =\alpha_0\cdot\deg(\psi)-|\Delta|\\
 	&\,\geq \left(1-\frac{\ell_1^2}{n}\right)\cdot \alpha_0 \deg(\psi)
 	\end{aligned}
 	\end{equation}
 	Combining \eqref{eqn-3-82} with \eqref{eqn-3-76}, we have
 	\begin{equation}\label{eqn-3-87}
 	\deg E_{\ol B}^{1,0}
 	 =\left(\frac{\left(n^2-\big(\frac{n}{r_{\infty}}\big)^2\right)(\alpha_0-2)}{\alpha_0}+(\ell_1^2-1)\right)
 	\cdot \frac{\alpha_0 \deg(\psi)}{12n}.
 	\end{equation}
 	Note that
 	\begin{equation*}
 	|\Delta_{nc}|=0,\quad\text{if~}\ell_1=1;\qquad\text{and}\qquad |\Delta_{nc}|=|\Delta|,\quad\text{if~}\ell_1>1.
 	\end{equation*}
 	Combining this with the Arakelov type equality in \eqref{eqn-3-191} and \eqref{eqn-3-85}, one obtains
 	$$\deg E_{\ol B}^{1,0}=\deg \bar f_*\omega_{\ols/\olb}\geq \left\{\begin{aligned}
 	&\frac{\rank A_{\olb}^{1,0}}{2}\cdot\left(\frac{(n-1)\alpha_0}{n}-2\right)\deg(\psi), &\quad&\text{if~}\ell_1=1;\\
 	&\frac{\rank A_{\olb}^{1,0}}{2}\cdot\left(\alpha_0-2\right)\deg(\psi), &\quad&\text{if~}\ell_1>1.
 	\end{aligned}\right.$$
 	Combining the above inequality together with \eqref{eqn-3-87}, we complete the proof.
 \end{proof}

 In order to apply \autoref{lem-3-19}\,(v), one still needs the next two lemmas.
 \begin{lemma}\label{lem-3-51}
 	Let $\ol F$ be an $n$-superelliptic curve of genus $g>(p-1)(2p-1)$ with $n=2p$ for some prime $p$.
 	Then $\ol F$ admits a unique $n$-superelliptic cover.
 \end{lemma}
 \begin{proof}
 	This follows from the Castelnuovo-Severi inequality (cf. \cite{accola-06}).
 	Indeed, if there were two different $n$-superelliptic covers:
 	$$\pi:~\ol F\lra \bbp^1,\qquad\text{and}\qquad \tilde \pi:~\ol F\lra \bbp^1,$$
 	then the ramified loci of both $\pi$ and $\tilde \pi$
 	have the same number of points with the same ramification indices.
 	Locally, we may assume that $\pi$ and $\tilde \pi$ are defined by
 	$y^n=F(x)$ and $y^n=\wt F(x)$ respectively,
 	where both $F(x)$ and $\wt F(x)$ are separable polynomials of equal degree.
 	Let $\alpha_0=\deg (F)=\deg (\wt F)$,
 	and let $D$ and $\wt D$ be defined by $z^p=F(x)$ and $z^p=\wt F(x)$ respectively.
 	Then $\alpha_0>2p$ by \eqref{eqn-3-9} as $g>(p-1)(2p-1)$,
 	and one has
 	$$g(D)=g(\wt D)=\left\{\begin{aligned}
 	&\frac{(p-1)(\alpha_0-2)}{2}, &~&\text{if~}p~|~\alpha_0;\\
 	&\frac{(p-1)(\alpha_0-1)}{2}, &~&\text{if~}p {\not|~}\alpha_0.
 	\end{aligned}\right.$$
 	
 	If $D\not \cong \wt D$, then $\ol F$ admits two different double covers to $D$ and $\wt D$ respectively.
 	Hence by the Castelnuovo-Severi inequality,
 	$$g\leq 1+2g(D)+2g(\wt D).$$
 	This contradicts \eqref{eqn-3-9} together with the above formulas for $g(D)$ and $g(\wt D)$.
 	
 	If $D\cong \wt D$, then again by the Castelnuovo-Severi inequality,
 	there is a unique action of $G'=\mathbb Z/p\mathbb Z$ on $D\cong \wt D$.
 	By the definitions of $D$ and $\wt D$,
 	this implies that the polynomials $F(x)$ and $\wt F(x)$ are conjugate to each other under the
 	automorphism group of $\bbp^1$.
 	Hence the covers $\pi$ and $\tilde \pi$ are the same, which is also a contradiction.
 \end{proof}

 \begin{lemma}\label{lem-3-52}
 	For any $n$-superelliptic curve $\ol F$,
 	if $\ol F$ admits a unique $n$-superelliptic automorphism group $G\cong \mathbb Z/n\mathbb Z$,
 	then $\sigma\circ\tau=\tau\circ\sigma$ for any $\sigma\in G$ and $\tau \in \Aut(\ol F)$.
 \end{lemma}
 \begin{proof}
 	Let $\ol F$ be defined by $y^n=F(x)$ as usual.
 	As $\ol F$ admits a unique $n$-superelliptic automorphism group,
 	it follows that $\sigma$ acts only $y$ and $\tau$
 	induces an automorphism on the quotient $\bbp^1=\ol F/G$.
 	In other words,
 	$$\sigma(x,y)=(x,\xi y) ~\text{~with $\xi^n=1$},\quad\qquad \tau(x,y)=\big(\tau_1(x), \tau_2(x,y)\big),$$
 	Note that $\tau_1$ is an automorphism of $\bbp^1$
 	and keeps the branch locus of $\pi:\,\ol F\to \bbp^1=\ol F/G$ invariant.
 	
 	If $\pi$ is branched over $\infty$ with local monodromy $a_{\infty}\neq 1$,
 	then $\tau_1$ must keep $\infty$ invariant,
 	and $F\big(\tau_1(x)\big)$ has the same set of roots as $F(x)$.
 	Hence $\tau_1(x)=ax+b$ for some $a,b\in \mathbb C$,
 	and $F\big(\tau_1(x)\big)=k^{*} \cdot F(x)$ with $k^*\neq 0$.
 	As $\tau$ is an automorphism group of $\ol F$, one obtains that
 	$\tau_2(x,y)=\eta\cdot y$ with $\eta^n=k^*$.
 	Therefore, it is clear that $\sigma\circ\tau=\tau\circ\sigma$.
 	
 	It remains to consider the case when $\pi$ is not branched over $\infty$,
 	since the case when $\pi$ is branched over $\infty$ with local monodromy $1$ can
 	be reduced to the former case by automorphism of $\bbp^1$.
 	In this case, $n~|~\alpha_0$, where $\alpha_0=\deg(F(x))$.
 	Moreover, as an automorphism of $\bbp^1$, $\tau_1$ has the form $\tau_1(x)=\frac{ax+b}{cx+d}$.
 	Since $\tau_1$ keeps the set of roots of $F(x)$ invariant,
 	$F\big(\tau_1(x)\big)=\frac{k^{*} \cdot F(x)}{(ax+b)^{\alpha_0}}$.
 	Hence $\tau_2(x,y)=\frac{\eta\cdot y}{(ax+b)^{\alpha_0/n}}$ with $\eta^n=k^*$.
 	One checks easily again that $\sigma\circ\tau=\tau\circ\sigma$.
 	This completes the proof.
 \end{proof}

 \begin{proof}[Proof of \autoref{prop-3-6}]
 	We prove by contradiction. Assume that $\rho_{n,n'}(C)$ is not a Shimura curve.
 	Then  by \autoref{lem-3-19} one has
 	\begin{equation}\label{eqn-3-202}
 	\text{$n'~|~\alpha_0~$ and $~n {\not|~}\alpha_0$}.\qquad
 	\end{equation}
 	\begin{equation}\label{eqn-3-201}
 	F_{\olb,i}^{1,0}=0, \qquad \text{for any $i\geq m'$ with $m'{\not|}~i$, where $m'=n/n'$}.
 	\end{equation}
 	On the other hand, by \eqref{eqn-3-35} one has
 	$$\rank E_{\olb,m'+1}^{1,0} - \rank E_{\olb,n-m'-1}^{1,0}
 	=\alpha_0-1-2\left[\frac{(m'+1)\alpha_0}{n}\right].$$
 	
 	If $n'\geq 4$, or $n'=m'=3$,
 	then one checks easily that
 	$$\rank E_{\olb,m'+1}^{1,0} - \rank E_{\olb,n-m'-1}^{1,0}>0.$$
 	This together with \eqref{eqn-3-171} implies in particular that $F_{\olb,m'+1}^{1,0}\neq 0$,
 	which contradicts \eqref{eqn-3-201}.
 	
 	If $n'=3$ and $m'=2$, then $n=6$ and $\alpha_0=6k+3$ for some $k\geq 1$ by \eqref{eqn-3-202};
 	if $n'=2$, then $n=4$ by the definition of $n'$, and $\alpha_0=4k+2$ for some $k\geq 1$ by \eqref{eqn-3-202}.
 	In any of the two cases above, by \autoref{lem-3-51} and \autoref{lem-3-52},
 	one verifies easily that the assumptions of \autoref{lem-3-19}\,(v) are satisfied.
 	Hence by \eqref{eqn-3-153} one has $\rank A_{\olb}^{1,0}<1$, i.e., $\rank A_{\olb}^{1,0}=0$.
 	In other words, $\deg E_{\olb}^{1,0}=\deg A_{\olb}^{1,0}=0$, which is a contradiction since $\bar f$ is non-isotrivial.
 	This completes the proof.
 \end{proof}

 %
 %

 We end this subsection by showing the following property on a Shimura $C\subseteq \mathcal{TS}_{g,n}$ with $2~|~n$.
 \begin{corollary}\label{cor-3-4}
 	Let $C\subseteq \mathcal{TS}_{g,n}$ be a Shimura curve as in \autoref{prop-3-6}.
 	If $C$ is compact and $n$ is even, then either
 	
 	(i)~ $\rho_{n,2}(C)$ is a point and $A_{\olb,n/2}^{1,0}=0$;\\[1mm]
 or
 	
 	(ii)~ $\rho_{n,2}(C)$ is also a Shimura curve, $\rank A_{\olb,n/2}^{1,0}$ is even, and
 	\begin{equation}\label{eqn-3-208}
 		\rank A_{\olb,n/2}^{1,0} \geq \frac{\rank E_{\olb,n/2}^{1,0}}{2}.
 	\end{equation}
 \end{corollary}
 \begin{proof}
 	Consider the image $\rho_{n,2}(C)\subseteq \mathcal{TH}_{g',2}$,
 	where $\rho$ is defined in \eqref{eqn-3-123}.
 	By \autoref{lem-3-15}, $\rho_{n,2}(C)$ is either a point,
 	or a compact Shimura curve.
 	In the first case, it is clear that $A_{\olb,n/2}^{1,0}=0$ by \autoref{lem-3-18}.
 	In the later case,  from \autoref{lem-evenrank} together with \eqref{eqn-3-72}, it follows that
 	$\rank A_{\olb,n/2}^{1,0}$ is even;
 	and for \eqref{eqn-3-208}, it suffices to prove
 		\begin{equation}\label{eqn-3-209}
 		\rank \wt A_{\olb}^{1,0} \geq \frac{\rank \wt E_{\olb}^{1,0}}{2},
 		\end{equation}
 	where we write $\wt{\bullet}$ for the corresponding objects associated to
 	the family $\bar f_1$ of semi-stable hyperelliptic curves associated to the Shimura curve $\rho_{n,2}(C)$.
 	In this hyperelliptic case, up to \'etale base change, one has $q_{\bar f_1}=\rank \wt E_{\olb}^{1,0}-\rank \wt A_{\olb}^{1,0}$ (cf. \cite[Thm.\,4.7]{lu-zuo-14} or \cite[Thm.\,A.1]{lu-zuo-13}),
 	and hence \eqref{eqn-3-209} follows from \cite[Thm.\,1]{xiao-92}
 	since $\bar f_1$ is clearly non-isotrivial.
 \end{proof}

 \subsection{Proof of \autoref{prop-3-14}}\label{sec-pf-3-14}
 (1).
 We prove by contradiction.
 Assume that $F_{\olb,i_0}^{1,0}\neq 0$ for some $i_0>n/2$.
 Then $F_{\olb,i_0}^{1,0}=E_{\olb,i_0}^{1,0}$,
 since $\rank E_{\olb,i_0}^{1,0}\leq 1$ by \eqref{eqn-3-35} together with the assumption that $\alpha_0=4$.
 Let $\wt R\subseteq \wt Y$ and $\wt \Gamma \subseteq \wt Y$ be the same as in \autoref{lem-3-20}.
 Then
 $$\begin{aligned}
 &\wt \Gamma\cdot \bigg(\omega_{\wt Y}(\wt R) \otimes \left({\mathcal L^{(i)}}^{-1} \otimes {\mathcal L^{(i_0)}}^{-1}\right)\bigg)<0,\qquad
 \forall~1\leq i\leq n-1.
 \end{aligned}$$
 Hence by \autoref{lem-3-20} and \autoref{lem-3-21},
 after a suitable finite \'etale base change,
 there exists a unique fibration $\bar f':\,\ols \to \olb'$ such that
 $$\dim H^0\big(\ols,\Omega_{\ols}^{1}\big)_i=\rank F_{\olb,i}^{1,0}, \quad
 H^0\big(\ols,\Omega_{\ols}^{1}\big)_i\subseteq (\bar f')^* H^0\big(\olb',\Omega_{\olb'}^{1}\big),
 \quad \forall~1\leq i\leq n-1.$$
 In other words, $g(\olb')\geq \rank F_{\olb}^{1,0}$,
 which is a contradiction by \autoref{lem-3-71}.

 (2).
 We prove again by contradiction.
 Assume that there exists some $i_0>n/2$
 such that $\rank F^{1,0}_{\olb,i_0}=\rank E^{1,0}_{\olb,i_0}=1$.
 Then similar as above, by \eqref{eqn-3-35} and \eqref{eqn-3-20} one checks that
 $$
 \wt \Gamma\cdot \bigg(\omega_{\wt Y}(\wt R) \otimes \left({\mathcal L^{(i)}}^{-1} \otimes {\mathcal L^{(i_0)}}^{-1}\right)\bigg)<0,
 \qquad\forall~1\leq i\leq n-1.
 $$
 Therefore, by \autoref{lem-3-20} and \autoref{lem-3-21},
 after a suitable \'etale base change,
 there exists a fibration $\bar f':\,\ols \to \olb'$ such that
 $g(\olb')\geq \rank F^{1,0}_{\olb}$,
 which contradicts \autoref{lem-3-71}.
 %
 \qed

{\vspace{2mm} \bf Acknowledgment.}
We would like to thank Frans Oort for his interest to our paper and many helpful comments.
We would also like to thank Shengli Tan for the discussion on the automorphism group action
in a family of superelliptic curves.


 \providecommand{\bysame}{\leavevmode\hbox to3em{\hrulefill}\thinspace}
 \providecommand{\MR}{\relax\ifhmode\unskip\space\fi MR }
 \providecommand{\MRhref}[2]{%
 	\href{http://www.ams.org/mathscinet-getitem?mr=#1}{#2}
 }
 \providecommand{\href}[2]{#2}

 \end{document}